\documentclass[11pt]{amsart}

\usepackage[hmargin=0.8in,height=8.6in]{geometry}
\usepackage{amssymb,amsthm, times}
\usepackage{delarray,verbatim}
\usepackage{mathrsfs}
\usepackage{ifpdf}
\ifpdf
\usepackage[pdftex]{graphicx}
 \else
\usepackage[dvips]{graphicx}
 \fi

\usepackage{bm}

\usepackage{ifpdf}
\usepackage{color}
\definecolor{webgreen}{rgb}{0,.5,0}
\definecolor{webbrown}{rgb}{.8,0,0}
\definecolor{emphcolor}{rgb}{0.5,0.95,0.95}

\usepackage{hyperref}
\hypersetup{%
          colorlinks=true,
          linkcolor=webbrown,
          filecolor=webbrown,
          citecolor=webgreen,
          breaklinks=true}
\ifpdf \hypersetup{pdftex,
            pdfstartview=FitH, 
            bookmarksopen=true,
            bookmarksnumbered=true
} \else \hypersetup{dvips} \fi

\newcommand {\ud}{{\rm d}}

\numberwithin{equation}{section}

\newtheorem{theorem}{Theorem}[section]
\newtheorem{proposition}{Proposition}[section]
\newtheorem{corollary}{Corollary}[section]
\newtheorem{remark}{Remark}[section]
\newtheorem{lemma}{Lemma}[section]

\newtheorem{assump}{Assumption}[section]

\numberwithin{remark}{section} \numberwithin{proposition}{section}
\numberwithin{corollary}{section}
\newcommand {\R}{\mathbb{R}}

\newcommand {\p}{\mathbb{P}}
\newcommand {\pp}{\mathbf{P}}

\newcommand {\E}{\mathbb{E}}

\newcommand{\diff}{{\rm d}}

\newcommand{\lev}{L\'{e}vy }

\newcommand{\green}{\textcolor[rgb]{0.00,0.50,0.50}}

\newcommand{\e}{\mathbb{E}}
\newcommand{\eee}{\mathbf{E}}

\newcommand{\exit}{{\mbox{\, \vspace{3mm}}}
\hfill\mbox{$\square$}}

\begin{document}

\title[Double continuation regions for American options under Poisson exercise opportunities]{Double continuation regions for American options under Poisson exercise opportunities}

\author[Z. Palmowski]{Zbigniew Palmowski$^{A}$}\thanks{$^A$Faculty of Pure and Applied Mathematics,
Wroc\l aw University of Science and Technology,
Wyb. Wyspia\'nskiego 27, 50-370 Wroc\l aw, Poland,
email: zbigniew.palmowski@pwr.edu.pl}

\author[J.L. P\'erez]{Jos\'e Luis P\'erez$^B$}\thanks{$^B$Department of Probability and Statistics, Centro de Investigaci\'on en Matem\'aticas, A.C. Calle Jalisco S/N
C.P. 36240, Guanajuato, Mexico,
email: jluis.garmendia@cimat.mx}

\author[K. Yamazaki]{Kazutoshi Yamazaki$^C$}\thanks{$^C$Department of Mathematics,
Faculty of Engineering Science, Kansai University, 3-3-35 Yamate-cho, Suita-shi, Osaka 564-8680, Japan,
email: kyamazak@kansai-u.ac.jp}
\thanks{Z. Palmowski is partially supported by Polish National Science Centre Grant
No. 2016/23/B/HS4/00566 (2017-2020). K. Yamazaki is partially supported by MEXT KAKENHI grant no. 17K05377 and  19H01791.}

\maketitle

\begin{abstract}
We consider the \lev model of the perpetual American call and put options with a negative discount rate under Poisson observations.
Similar to  the continuous observation case as in De Donno et al.\ \cite{DPT}, the stopping region that characterizes the optimal stopping time is either a half-line or an interval. The objective of this paper is to obtain explicit expressions of the stopping and continuation regions and the value function, focusing on spectrally positive and negative cases.  To this end, we compute the identities related to the first Poisson arrival time to an interval via the scale function and then apply those identities to the computation of the optimal strategies. 
We also discuss the convergence of the optimal solutions to those in the continuous observation case as the rate of observation increases to infinity.  Numerical experiments are also provided.

\end{abstract}
{\noindent \small{\textbf{Keywords:}\,  American options; optimal stopping; L\'evy processes; Poisson observations; double continuation regions; put-call symmetry }}

\section{Introduction}

%

Research on American options is one of the most actively studied fields at the intersection of finance and optimal stopping. The objective is to derive the optimal exercise strategy that maximizes the expected payoff upon exercise. With the application of the classical optimal stopping theory, the optimal strategy can be characterized as the first entry time of the underlying process to a certain region, often called the \emph{stopping region}, or equivalently the first time it leaves the so-called \emph{continuation region}. Considerable research has focused on the analysis of the stopping and continuation regions. Typically, in the perpetual case driven by a one-dimensional process, the stopping and continuation regions can be shown to be half-lines; hence, the optimal strategy is a barrier-type one, reducing the problem to obtaining the (single) optimal boundary that separates the continuation and stopping regions.


%
%
%

In this paper, we challenge two of the most commonly imposed assumptions in perpetual vanilla American options: (1) \emph{the positivity of the discount rate} and (2) \emph{continuous observations} (where one can exercise the option at any time). Although these assumptions significantly simplify the problem and often guarantee the optimality of a barrier strategy, they are often unrealistic.  In particular, it is of substantial interest to analyze, when these are relaxed, if a barrier strategy remains optimal or instead the forms of the stopping and continuation regions change. 

%
%



\subsection{Optimal stopping with a negative discount rate}

Whereas most of the existing results  assume a positive discount rate, several important results exist for American options with a negative discount rate. 

One of the most well-known examples of when the negative effective discount rate arises is the  \emph{stock loan}, as considered by Xia and Zhou \cite{Xia}. When the loan interest rate is higher than the risk-free rate, the problem reduces to the valuation of an American call option with a negative discount rate.  Other examples include real option problems (see, e.g., Dixit and Pindyck \cite{Dixit}), where the effective discount rate becomes negative when the cost of investment increases at a higher rate than the firm's discount rate.  In addition, the real interest rate can  become negative during low-yield regimes (see Black \cite{Black} for further discussion).  The importance of these models has been rapidly developing in the current low-interest environments. We refer the reader to \cite{Battauz1, Battauz2,DPT} for a detailed literature review on the American option problem with a negative discount rate.

Most research on American (and real) options assumes either geometric Brownian motion or an exponential \lev process for the underlying asset-price process. In these cases, it is easy to demonstrate that the value function is a convex function majoring a linear payoff; hence, the stopping region (where the value function coincides with the payoff function) becomes either a half line or an interval. For the case in which the discount rate is positive, it becomes a half line (except for exotic cases, such as \cite{Broadie_Detemple, Detemple}). However, when the discount rate is negative, the same result may not hold, and the stopping region may become an interval. Thus, the continuation region consists of two separate regions that we call the double continuation regions. 


In this context, many researchers have focused on pursuing the optimality of a barrier strategy by imposing additional constraints on the discount rate and underlying process.  For example, Xia and Zhou \cite{Xia}  considered the stock loan problem in which the asset price is a geometric Brownian motion, and they show the optimality of a barrier strategy  under some assumptions on the parameters of the process. This work has been extended by various researchers and, among others, Leung et al.\ \cite{LYZ} generalized the results to the \lev case and applied them to study the swing options  with multiple exercise opportunities.


The analysis of an interval strategy corresponding to the double continuation region is, on the other hand, rather new and involves more intricate computations. In this context, Battauz et al.\ \cite{Battauz1, Battauz2}  considered the Brownian motion case for the analysis of the double continuation region. Recently, De Donno et al.\ \cite{DPT}  extended the results to the spectrally one-sided  \lev  case and multiple stopping (swing option) cases.

\subsection{Poissonian observation}

In financial mathematics, it is standard to use the continuous-time model, where one can take the best advantage of stochastic analysis, particularly It\^o calculus.  This is a significant advantage over discrete-time models (with deterministic decision times) where essentially only numerical approaches are available. In reality, however,  the decision maker can
observe the asset price and make exercise decisions only at intervals; therefore,  it is important to study the effect on the optimal strategy when the continuous observation assumption is relaxed.


Recently, the analysis of  \lev processes observed at Poisson arrival times has received substantial attention (see, e.g., \cite{Albrecher}), and some researchers have started to apply these results in insurance and financial mathematics. 
To the best of  our knowledge, this is the only example of discrete-time observation models in which analytical approaches are still possible. Due to the memorylessness property of the exponential random variable, the problem remains one-dimensional, without the need to keep track of how much time has passed since the last exercise opportunity.

Regarding the optimal stopping problem under Poisson observations, it has been studied by Dupuis and Wang \cite{Dupuis} and  by P\'erez and Yamazaki \cite{PY_American} for the Brownian motion and the \lev cases, respectively.   They show that, when the discount rate is positive, the optimal strategy is still of barrier-type and that
stopping at the first exercise opportunity at which the asset price is below or above a certain barrier  is optimal. 
Several related stochastic control problems have been analyzed under the same Poisson observation settings. See \cite{Avanzi_Cheung_Wong_Woo, Avanzi_Tu_Wong, Noba_Perez_Yamazaki_Yano} for the optimal dividend problem and \cite{PPSY} for determining the endogenous bankruptcy level.
 
Various motivations exist for considering the Poisson observation model. By restricting the exercise opportunities to Poisson epochs, we can model the scenarios in which
investors can access the information on the option only at random times; for example, in the cases in which one can only observe a jump of an exogenous
stock price or when some investments are available. As noted by \cite{Dupuis}, this time restriction can
be particularly useful in daily financial practice.
Similar considerations are found in the field
of stochastic control  (see \cite{12,13}). 

Similar to other important applications, Poisson observation models can potentially be used for approximating optimal strategies in the deterministic discrete-time models (see Section 1 of \cite{PPSY} for the accuracy of approximations). As discussed in Section \ref{subsection_remark}, these models can also be used to approximate the continuous observation case \cite{DPT}.





\subsection{This paper} \label{subsection_this_paper}

 \begin{figure}[htbp]
\begin{center}
\begin{minipage}{1.0\textwidth}
\centering
\begin{tabular}{c}
 \includegraphics[scale=0.5]{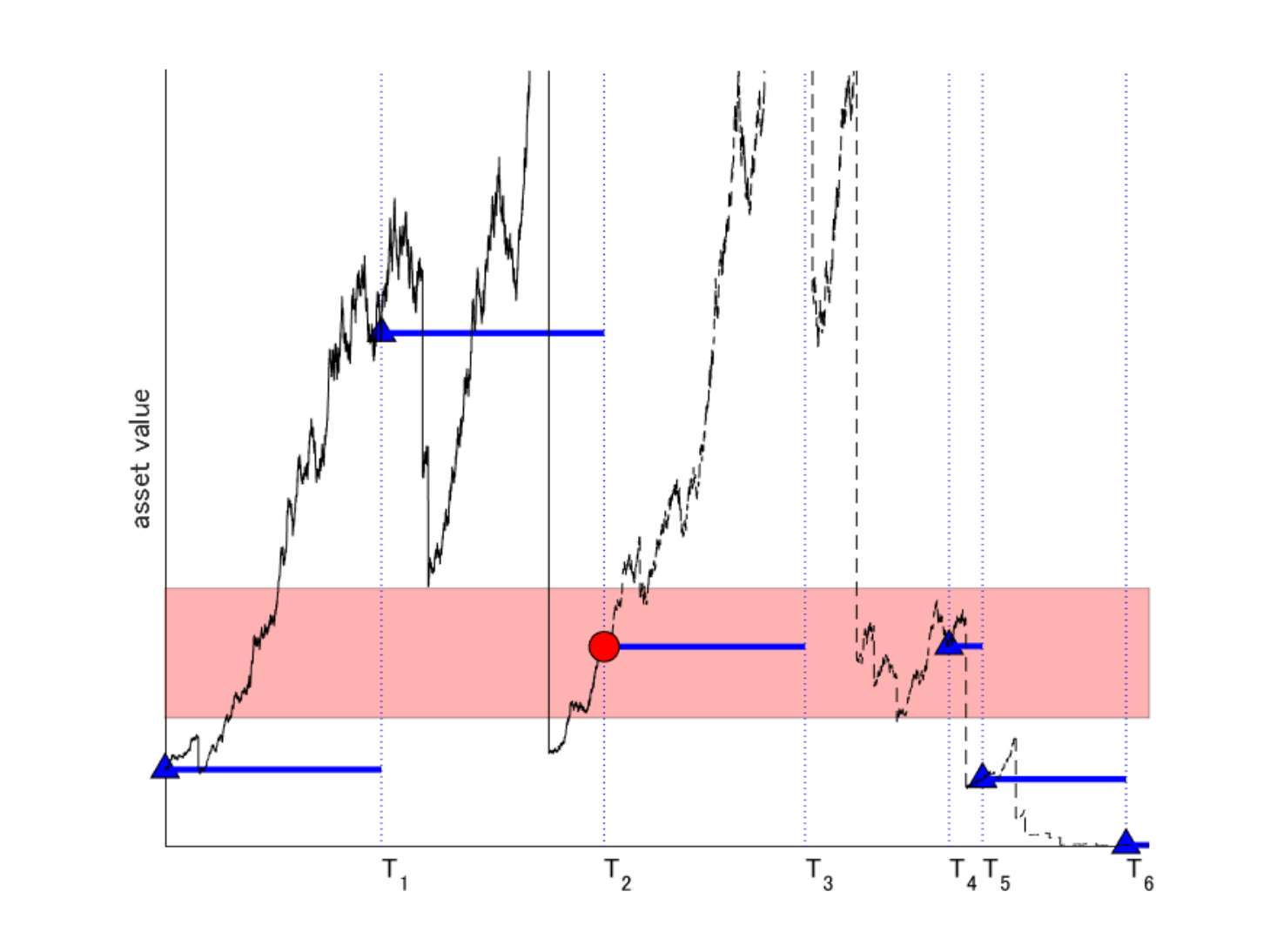}
\end{tabular}
\caption{Sample paths of the asset price $S$ (black lines) and the observed price $S^\lambda$ (horizontal blue lines) along with the Poisson arrival times $\mathcal{T}^\lambda$ (indicated by dotted vertical lines). The stopping region $[L, U]$ is given by the rectangle colored in red.
The asset price at the exercise time and other observation times are indicated by the red circle and blue triangles, respectively.  Here, the exercise time corresponds to $T_2^\lambda$, but the asset value has crossed $[L,U]$ before.
}  \label{plot_simulated}
\end{minipage}
\end{center}
\end{figure}

In this paper, we consider perpetual American put and call options under Poisson observations with a negative discount rate.   Given an asset-price process $S = (S_t: t \geq 0)$, we consider the scenario in which exercise opportunities are given as epochs
 $\mathcal{T}^\lambda := (T_n^\lambda: n \geq 1)$, modeled by the jump times of an independent Poisson process $(N^\lambda_t: t \geq 0)$ with a fixed rate $\lambda$. 
We are particularly interested in when the optimal strategy becomes the following form:
\begin{align}
\inf \{ t \in \mathcal{T}^\lambda: S_t \in [L,U] \}, \label{our_default_continuous}
\end{align}
for some $L < U$.
Notice that this can also be written as the following classical entry time
\begin{align*}
\inf \{ t > 0: S^\lambda_{t} \in [L,U] \}, 
\end{align*}
of the asset price $S^\lambda = (S^\lambda_t: t \geq 0)$ if it is only updated at $\mathcal{T}^\lambda$:
\begin{align}
S^\lambda_t := S_{T^\lambda_{N^\lambda_t}}, \quad t \geq 0. \label{S_lambda}
\end{align}
Here $T^\lambda_{N^\lambda_t}$ is the \emph{most recent exercise opportunity} before $t$.
In Figure \ref{plot_simulated}, we plot the sample paths of $S$,  $S^\lambda$, $\mathcal{T}^\lambda$ and the corresponding exercise time \eqref{our_default_continuous}.

Our analysis begins with the general \lev case in which we show that the stopping region is necessarily a connected region, which takes the form of a half-line or a finite interval.  Furthermore, we obtain sufficient conditions for the optimal strategy to take the form \eqref{our_default_continuous}.



To present a more explicit solution to the problem,
we then focus on the spectrally one-sided (asymmetric) \lev process or, equivalently, the \lev process with only negative jumps or only positive jumps. 
Our first task is to obtain the joint Laplace transform of the first Poisson observation time at which the process is in an interval and the position of the process at that instance. We express this Laplace transform in terms of the \emph{scale function} of a spectrally negative \lev process, and, as a direct corollary, the expected payoff under the interval strategy \eqref{our_default_continuous}. With the spectrally one-sided assumption, semi-explicit expressions are elicited, without focusing on a particular set of jump measures. 

Using these expressions in terms of the scale function, we conduct both analytical and computational analyses on American put and call options when the \lev process is spectrally one-sided. We first consider the put option and analyze the first-order conditions that the optimal upper and lower boundaries must satisfy. For the call option, we verify the put-call symmetry formula (see
e.g. \cite{CarrChesney,EberleinPapantaleon, FM}) and reduce the call option problem to a put option problem. 

%
%
%
%
%

These results are confirmed numerically using the examples of a \lev process with exponential downward or upward jumps.  We demonstrate that using the obtained analytical results, the optimal strategy and the optimal value function can be computed instantaneously, enabling us to conduct a series of numerical experiments.  We demonstrate that the stopping region becomes an interval and confirm the optimality by comparing it with the expected payoffs under different strategies. We also study the influence of the choice of the rate of observation $\lambda$ the optimal solutions.


\subsection{Other remarks} \label{subsection_remark}

One of our main motivations of this study is to derive an efficient numerical approach for the computation of optimal solutions in the continuous observation case \cite{DPT} that involves the integration of the resolvent measure with respect to the \lev measure; this is required due to the fact that the process can jump to an interval or jump over it.
In our case, on the other hand, the obtained expression is simpler and works for a general spectrally one-sided \lev process, without the need of integration with respect to the \lev measure.
To determine whether our results can be used as an approximation of the results by \cite{DPT}, we confirm both analytically and numerically that the optimal strategies and value function converge to those by \cite{DPT} as the rate of observation $\lambda$ goes to infinity.

As noted in Section \ref{subsection_this_paper}, our problem can be considered as a classical optimal stopping problem (with continuous observation) driven by the process $S^\lambda$, as  in \eqref{S_lambda}, which contains both positive and negative jumps even when $S$ itself is spectrally one-sided (again see Figure \ref{plot_simulated}).
Existing results featuring asset-price processes with two-sided jumps are rather limited in the study of American options.
However, we provide a new analytically tractable case for $S^\lambda$, containing two-sided jumps.  By appropriately selecting the driving process $S$ and $\lambda$, one can construct a wide range of stochastic processes with two-sided jumps.

\subsection{Relevant literature}

In this paper, we adopt the \lev model in which
 the dynamics of asset prices are described with more accuracy with the addition of the  possibility of jumps.
Indeed, several empirical studies have concluded that the log-prices of stocks and other assets have a  heavier left tail than the normal distribution on which the seminal Black-Scholes model was founded. L\'evy processes have a long tradition of modeling financial markets (see e.g.
\cite{B10,Erik1, B24, Cont, B42,B80,Merton, Schoutens}). 
For a more general study of financial models using \lev processes, the reader should   refer to \cite{Cont}. 

%

Regarding the vanilla American options driven by \lev processes, as demonstrated by Mordecki \cite{Mordecki}, the optimality of a barrier strategy generally holds if the discount rate is positive.
Many have  succeeded in showing the optimality of a barrier strategy in related optimal stopping problems \cite{alili, Asmussen_2004, BoyLev, Chan, Darling, Gapaev1, Gapaev2}. 
However, compared to the abundance of established results on  perpetual American options, research on the case of a negative discount rate is significantly limited.  

In this paper, we take advantage of the scale function, which is known to exist for one-dimensional diffusions and spectrally one-sided \lev processes. Using this, one can solve the problem for a wide class of stochastic processes without focusing on a particular type.  Regarding the application of the scale function in optimal stopping, we refer to, among others, \cite{Dayanik, DeAngelis}  for the diffusion case and \cite{DPT, Ott, Rodosthenous} for the \lev case.




%

The remainder of the paper is organized as follows. Section \ref{section_general_case} models the problem and obtains the main result for the general \lev case, together with the asymptotic analysis as the rate of observation goes to infinity.
 Section \ref{section_preliminaries} reviews the spectrally negative \lev process and its fluctuation theory.  Next, Section  \ref{section_hitting_time} identifies the quantity related to the first entry time to an interval under Poisson observation times.  Section \ref{section_put_SN} considers American put options for both spectrally negative and positive cases. These results are then extended to the American call options via the put-call symmetry in Section  \ref{section_call} . Finally, Section \ref{section_numerics} is devoted to numerical experiments. 
Throughout the paper, we will follow the convention that $\inf \varnothing = \infty$ and $\sup \varnothing =0$. 

\section{General \lev case} \label{section_general_case}





Throughout this paper, we let $X = (X_t: t \geq 0)$ be a \lev process defined on a probability space $(\Omega, \mathcal{F}, \pp)$
 and  $S_t = \exp (X_t)$ be the price of a stock at time $t \geq 0$.  For each $s>0$, 
we denote by $\pp_s$ the law of $S$ when it starts at $S_0= s$ (i.e. $X_0 = \log s$) and write for convenience  $\pp$ in place of $\pp_1$. In addition, we shall write $\eee_s$ and $\eee$ for the associated expectation operators. We define 
\begin{align}
\mathcal{T}^\lambda := ( T_1^\lambda, T_2^\lambda, \ldots ) \label{T_lambda}
\end{align}
 as the jump times of an {\it independent Poisson process} $N^\lambda$ with rate $\lambda > 0$.  Let $\mathbb{F} = (\mathcal{F}_t)_{t \geq 0}$ be the filtration generated by the processes $(X,N^\lambda)$ and $\mathbb{T}$ the set of  $\mathbb{F}$-stopping times.  The set of strategies is  given by $\mathcal{T}^{\lambda} \cup \{\infty\}$-valued stopping times:
\begin{align*}
\mathcal{A} := \{ \tau \in \mathbb{T}: \tau \in \mathcal{T}^{\lambda}\cup \{\infty\} \; {\rm a.s.} \}.
\end{align*}


We consider perpetual American-type put/call options:
\begin{align}
V_i(s) = \sup_{\tau \in \mathcal{A}} \eee_s [e^{-r \tau} G_i(S_\tau) 1_{\{ \tau < \infty \}}
], \quad i = p, c, \label{value_function} \end{align}
for the payoff functions
\begin{align*}
G_p(s) := (K-s)^+ \quad \textrm{and} \quad G_c(s) := (s-K)^+, \quad s > 0,
\end{align*}
where $K > 0$ is the strike price. We are particularly interested in the case discount rate
\begin{equation} \label{assump_negative_r}
r<0
\end{equation}
since the positive case was already analyzed in \cite{PY_American}.

\subsection{Assumptions}
Throughout this paper, in addition to the assumption \eqref{assump_negative_r}, in order to focus on the case the value function is finite, we assume the following three assumptions.
\begin{assump} \label{assump_lambda_r}
We assume $\lambda +r > 0$.
\end{assump}
Notice that this is a natural assumption and it holds if and only if $\E [e^{-r T_1^\lambda}] < \infty$; if this is violated, the expected net present value of the wealth of a unit value at the next observation time becomes infinity.

For the call case, we additionally assume the following.
\begin{assump} \label{assump_lambda_r_alpha} For the call option ($i = c$), we assume $\eee S_1 < \infty$ and  $\lambda + r - \log \eee S_1 > 0$ so that 
\begin{align} \label{call_n_finite}
 \eee \big[e^{-{r} T_n^\lambda} S_{T_n^\lambda} \big] = \eee \big[\eee [e^{-{r} T_n^\lambda} S_{T_n^\lambda} |T_n^\lambda] \big]  =  \eee [ e^{-({r}-\log \eee S_1) T_n^\lambda}] = \Big( \frac \lambda {\lambda + r - \log \eee S_1}\Big)^n < \infty, \quad n \geq 1.
\end{align}
\end{assump}
Finally, we assume the following.
\begin{assump} \label{assump_tail_value_function}
For $i = p,c$, we assume 
$\sup_{\tau \in\mathcal{A}} \eee_s [e^{-{r} \tau} G_i(S_\tau)
1_{\{ T_N^\lambda<\tau <\infty  \}} ]\xrightarrow{N \uparrow \infty} 0$ for $s > 0$.
\end{assump}
Following \cite{DPT}, we obtain a sufficient condition for Assumption \ref{assump_tail_value_function} for the put case as follows; a sufficient condition for the call case is given in Lemma \ref{lemma_assumption_finiteness_call}.
\begin{lemma}\label{cond_last_time}
Assumption \ref{assump_tail_value_function} for the put case ($i = p$) is satisfied if  $\eee_s [e^{-{r}T_{\text{last}}(K)}
]<\infty$ for $s > 0$ where
\[
T_{\text{last}}(K) := \sup\{t\geq0: S_t\leq K \}. \label{tau_last}
\]
Note that this guarantees $T_{\text{last}}(K) < \infty$ a.s. 
\end{lemma}
\begin{proof}
By this assumption, \eqref{assump_negative_r}, and dominated convergence, we have
$0 \leq \sup_{\tau \in\mathcal{A}} \eee_s [e^{-{r} \tau} (K-S_\tau)^+
1_{\{ T_N^\lambda<\tau <\infty  \}} )\leq K \eee_s [e^{-{r}T_{\text{last}}(K)}
1_{\{ T_{\text{last}}(K)  > T_N^\lambda \}}
]\xrightarrow{N \uparrow \infty} 0$.
\end{proof}
Concise sufficient conditions for this result are given in Lemmas  \ref{lemma_tail_SN}  and \ref{lemma_tail_SP} for spectrally negative and positive L\'evy processes, respectively.


\subsection{Optimal strategies for a general \lev model} \label{section_optimality_barrier}


In this section, we show that the optimal stopping times for the problem \eqref{value_function} for both call and put cases are of the form
\begin{align} \label{barrier_strategies}
\tau_{[L,U]} := \inf \{\tau \in \mathcal{T}^{\lambda}: S_{\tau } \in [L, U] \}
\end{align}
for suitably chosen barriers $L \in [0, \infty)$ and $U\in(0, \infty]$, or otherwise the stopping region is empty. With abuse of notation, it is understood that $\tau_{[0,U]} := \inf \{\tau \in \mathcal{T}^{\lambda}: S_{\tau } \leq U \}$ and  $\tau_{[L,\infty]} := \inf \{\tau \in \mathcal{T}^{\lambda}: S_{\tau } \geq L \}$.

To show this, we consider the value function of an auxiliary problem where \emph{immediate stopping is also allowed}:
\begin{align} \label{extended_problem}
\bar{V}_i(s) := \sup_{\tau \in \bar{\mathcal{A}}} \eee_s [e^{-{r} \tau} G_i(S_\tau) 1_{\{ \tau < \infty \}}], \quad i = p, c, \quad s > 0, \end{align}
where $\bar{\mathcal{A}} := \{ \tau \in \mathbb{T}: \tau \in \bar{\mathcal{T}}^\lambda \cup \{\infty\} \; {\rm a.s.} \}$
with $\bar{\mathcal{T}}^\lambda := \mathcal{T}^\lambda \cup \{0\}$.

To see why we consider this version, note that by the strong Markov property,
\begin{align*}
V_i(s) = \eee_s [e^{- {r} T_1^\lambda} \bar{V}_i(S_{T_1^\lambda})], \quad i = p, c, \quad s > 0.
\end{align*}
If \eqref{extended_problem} is solved by  a stopping time
\begin{align*}
\bar{\tau}_{[L,U]} := \inf \{ \tau \in \bar{\mathcal{T}}^\lambda : S_{\tau } \in [L, U] \},
\end{align*}
it is clear that \eqref{value_function}  is solved by \eqref{barrier_strategies} for the same values of $L$ and $U$. Hence, we shall analyze $\bar{V}_i$ below.

Similarly to the proof of Proposition 3.1 of  \cite{PY_American}, we first show the following crucial fact.
\begin{proposition} \label{proposition_convexity}The mappings $s \mapsto \bar{V}_p(s)$ and $s \mapsto \bar{V}_c(s)$ are finite and  convex on $(0, \infty)$.
\end{proposition}

\begin{proof} 

	Define the value function of a finite-maturity case with maturity  $0 \leq N < \infty$:
\begin{align*}
\bar{V}_{i,n}^N(s) &:= \sup_{\tau \in \mathcal{A}_{n,N}} \eee \big[e^{-{r} (\tau- T_n^\lambda)} G_i(S_\tau)
| S_{T_n^\lambda} = s
\big], \quad 0 \leq n \leq N, \quad i = p,c, \quad s > 0,
\end{align*} 	
where 
\begin{align*}
\mathcal{A}_{n,N} &:= \{ \tau \in \bar{\mathcal{A}}: T_n^\lambda \leq \tau \leq T_N^\lambda  \; \, {\rm a.s.} \}, \quad 0 \leq n \leq N,
\end{align*}
with $T_0^\lambda := 0$.
In other words, this is the expected value on condition that $S_{T_n^\lambda} =s $ and the controller has not stopped before $T_n^\lambda$ and optimally stops afterwards.

Similarly, we define its infinite-horizon case:
\begin{align*}
\bar{V}_{i,n}(s) &:= \sup_{\tau \in \mathcal{A}_n} \eee \big[e^{-{r} (\tau- T_n^\lambda)} G_i(S_\tau)
1_{\{\tau < \infty \}}
| S_{T_n^\lambda} = s \big],  \quad i = p,c, \quad s > 0,
\end{align*}
where
\begin{align*}
	\mathcal{A}_n &:= \{ \tau \in \bar{\mathcal{A}}: \tau \geq T_n^\lambda \, {\rm a.s.} \}, \quad n \geq 0.
\end{align*}
It is clear that, for all $n \geq 0$, $s > 0$, and $i = p,c$,
\begin{align*}
\bar{V}_{i,n}^N(s) = \bar{V}_{i,0}^{N-n}(s),  \; N \geq n, \quad \textrm{and} \quad \bar{V}_{i,n}(s) = \bar{V}_i(s). 
\end{align*}

For $1 \leq  n \leq N$, thanks to the positivity of the payoff,  we have $\bar{V}_{i,n}^N(s) = \bar{V}_{i,0}^{N-n} (s) \leq  \sum_{j = 0}^{N-n} \eee_s \big[e^{-{r} T_j^\lambda} G_i(S_{T_j^\lambda})
\big]$,  which is finite by Assumption \ref{assump_lambda_r} for the put case and by Assumption \ref{assump_lambda_r_alpha}  for the call case. By this and Assumption $\ref{assump_tail_value_function}$, $\bar{V}_{i,n}(s)$ is finite as well.

By backward induction (similarly to the case of discrete-time optimal stopping problems),
and following the proof of
Proposition 3.1 of \cite{PY_American},  it can be shown that $\bar{V}_{i,0}^N(s)$ is convex for each $N$. See also \cite{Shiryaev}. Hence, in order to see if the convexity holds also for $\bar{V}_{i}$, it suffices to show that $\bar{V}_{i,0}^N(s) \xrightarrow{N \uparrow \infty} \bar{V}_{i}(s)$
for each $s > 0$ and $i = p,c$. 
This indeed holds because
\begin{align*}
0 \leq \bar{V}_{i}(s) - \bar{V}_{i,0}^N(s)
&= \sup_{\tau \in \bar{\mathcal{A}}} \eee_s [e^{-{r} \tau} G_i(S_\tau) 1_{\{\tau < \infty \}}
 ] - \sup_{\tau \in \bar{\mathcal{A}}} \eee_s [e^{-{r} \tau} G_i(S_\tau)
1_{\{ \tau \leq T_N^\lambda \}} ]  \leq\sup_{\tau \in \bar{\mathcal{A}}} \eee_s [e^{-{r} \tau} G_i(S_\tau)
1_{\{ \tau > T_N^\lambda \}}],
\end{align*} 
which vanishes as $N \rightarrow \infty$ by Assumption \ref{assump_tail_value_function}.

\end{proof}
	The proof of the following is deferred to Appendix \ref{appen_lemma_varepsilon_stopping}.
\begin{lemma}\label{lemma_varepsilon_stopping}
We have $\inf_{0 < s \leq K} (\bar{V}_p(s) - G_p(s)) = 0$ and $\inf_{s \geq K} (\bar{V}_{c}(s) - G_{c}(s)) = 0$. 
\end{lemma}

For $i = p,c$, let us define the stopping region 
\begin{align*}
\mathcal{D}_{i} := \{ s > 0: \bar{V}_{i}(s) = G_i(s) \}, 
\end{align*}
and similarly for the classical (continuous observation) case \cite{DPT} by $\mathcal{D}_{i, \infty} := \{ s > 0: V_{i,\infty}(s) = G_i(s) \}$ where $V_{i,\infty}$ is the value function in the classical case. The corresponding continuation regions are defined as their complements.  When $\mathcal{D}_{i, \infty} \neq \varnothing$, as in Lemma 2 in \cite{DPT}, there exist  $0 \leq  L_{p,\infty}^* \leq U_{p,\infty}^* < K$ and $K<  L_{c,\infty}^* \leq U_{c,\infty}^* \leq \infty$ such that $\textrm{cl} (\mathcal{D}_{i, \infty})  = [L_{i,\infty}^*, U_{i,\infty}^*]$, where $\textrm{cl}(\mathcal{D}_{i, \infty})$ denotes the closure of $\mathcal{D}_{i, \infty}$.

\begin{remark} \label{remark_V_G_vanish}
(1) From Lemma \ref{lemma_varepsilon_stopping}, the convexity of $\bar{V}_{p}$, and the fact that $\bar{V}_p(K) > 0$ and $G_p$ is linear on $(0,K)$, if $\mathcal{D}_p = \varnothing$, then we must have $\lim_{s \downarrow 0} (\bar{V}_{p} (s) - G_p(s)) = 0$. (2) Similarly, for the call case, if  $\mathcal{D}_c = \varnothing$, we have $\lim_{s \uparrow \infty} (\bar{V}_{c} (s) - G_c(s)) = 0$.
\end{remark}

We can in fact show that the stopping region for the put case is non-empty; a sufficient condition for the non-emptiness for the call case is given later in Lemma \ref{lemma_nonempty_call}.
\begin{lemma} \label{lemma_stopping_empty}  
	We have $\lim_{s \downarrow 0} (\bar{V}_p(s) - G_p(s)) \geq \lim_{s \downarrow 0} (V_p(s) - G_p(s)) > 0$ and $\mathcal{D}_p \neq \varnothing$.
\end{lemma} 
\begin{proof}
Fix any  $U < K$. When $S_0 = 1$, because $\{ \tau_{[0,U / s ]}=  T_1^\lambda \} \supset \{ \sup_{0 \leq t \leq T_1^\lambda} S_t \leq  U / s \}$ for any $s>0$, and  $\sup_{0 \leq t \leq T_1^\lambda} S_t< \infty$ a.s., it follows $\pp$-a.s.\  that $\tau_{[0, U / s ]} \xrightarrow{s \downarrow 0} T_1^\lambda$.  
On the other hand, 
\[
s S_{\tau_{[0,\frac U s ]}} \leq s \sup_{0 \leq t \leq T_1^\lambda} S_t + U 1_{\{ \tau_{[0,\frac U s ]} > T_1^\lambda \}}\leq  s \sup_{0 \leq t \leq T_1^\lambda} S_t + U 1_{\{ \sup_{0 \leq t \leq T_1^\lambda} S_t > \frac U s \}} \xrightarrow{s \downarrow 0} 0 \quad \textrm{$\pp$-a.s.}
\]
Therefore using $\eee_s [e^{-{r} \tau_{[0,U]}} G_p(S_{\tau_{[0,U]}}) 1_{\{ \tau_{[0,U]} < \infty \}}] = \eee [e^{-{r} \tau_{[0, U / s]}} G_p(s S_{\tau_{[0, U / s ]}}) 1_{\{ \tau_{[0, U / s]} < \infty \}}]$,
we obtain, by Fatou's lemma and the fact that $\tau_{[0,U]} \in \mathcal{A}$,
\[ \lim_{s \downarrow 0} \bar{V}_p(s) \geq \lim_{s \downarrow 0} V_p(s)  \geq \liminf_{s \downarrow 0}\eee_s [e^{-{r} \tau_{[0,U]}} G_p(S_{\tau_{[0,U]}}) 1_{\{ \tau_{[0,U]} < \infty \}}] = K\eee \left[e^{-r T_1^\lambda}\right]>K = \lim_{s \downarrow 0} G_p(s).\]	
Now,  in view of Remark \ref{remark_V_G_vanish}, we must have $\mathcal{D}_p \neq \varnothing$.
\end{proof}


The main result of this section is given as follows.

\begin{theorem} \label{theorem_barrier_optimal}
(1) For the put option ($i = p$), there exist
\begin{align} \label{L_range_put}
0 < L_{p}^* \leq L_{p,\infty}^* \leq U_{p,\infty}^* \leq U_{p}^* < K
\end{align}
 such that $\tau_{[L_{p}^*, U_{p}^*]}$ solves \eqref{value_function}. 

(2) For the call option ($i = c$), one of the following holds true.

(i) There exist
\begin{align*} 
 K<L_{c}^* \leq L_{c,\infty}^* \leq U_{c,\infty}^* \leq U_{c}^* \leq \infty
\end{align*}
 such that  $\tau_{[L_{c}^*, U_{c}^*]}$ solves \eqref{value_function}.

(ii) We have  $\bar{V}_c(s) > G_c(s)$ for all $s > 0$.

\end{theorem}
\begin{proof} 


(1) 
We consider the auxiliary problem  \eqref{extended_problem}. Because $\mathcal{A} \subset \mathbb{T}$ and immediate stopping gives $(K-s)^+$,
\begin{align}
K- s\leq (K-s)^+ \leq \bar{V}_p(s) \leq V_{p, \infty} (s), \quad s > 0.  \label{V_dom_K_s}
\end{align}
Define $\tilde{\mathcal{D}} := \{ s > 0: \bar{V}_{p}(s) = K-s\}$.
By  \eqref{V_dom_K_s} and the convexity of $\bar{V}_p$ as in Proposition \ref{proposition_convexity},  we have either  Case (A):  $\tilde{\mathcal{D}} = [L_{p}^*, U_{p}^* ]$ for some $0 \leq L_{p}^* \leq U_{p}^* \leq \infty$  or otherwise Case (B): $\tilde{\mathcal{D}} = \varnothing$.

(A) Suppose $\tilde{\mathcal{D}}  = [L_{p}^*, U_{p}^* ]$. We show $L_{p}^*$ and $U_{p}^*$ satisfy \eqref{L_range_put} and in addition
\begin{align}
\mathcal{D}_p = \tilde{\mathcal{D}}. \label{V_p_plus}
\end{align}

(a) For $s \in \mathcal{D}_{p, \infty}  
 \subset (0, K)$, we have $V_{p, \infty}(s) = (K-s)^+ =K-s$ and hence \eqref{V_dom_K_s} gives $\bar{V}_{p}(s) = (K-s)^+ = K-s$ as well and hence $s \in \tilde{\mathcal{D}}$. This shows $L_{p}^* \leq L_{p,\infty}^* \leq U_{p,\infty}^* \leq U_{p}^*$. Moreover, by Lemma \ref{lemma_stopping_empty}, we have $L_p^* > 0$.

(b) In order to show that $U^*_p <  K$ and \eqref{V_p_plus},
it suffices to show $\bar{V}_p(s) > 0$ for all $s > 0$.  Indeed, this holds because the payoff function is nonnegative and $S$ can reach any level below $K$ with a positive probability.

Now, by the dynamic programming principle (see, e.g., Theorem 1.11 of Peskir and Shiryaev \cite{PS}), we have
\begin{align*}
\bar{V}_{p}(s) = \max \big((K- s)^+, \eee_s [e^{-{r} T_1^\lambda} \bar{V}_{p}(S_{T_1^\lambda})] \big), \quad s > 0,
\end{align*}
and $\inf \{ \tau \in \bar{\mathcal{T}}^\lambda : \bar{V}_{p}(S_{\tau }) = (K-S_{\tau })^+ \} = \inf \{ \tau  \in \bar{\mathcal{T}}^\lambda : S_{\tau } \in [L_{p}^*, U_{p}^* ] \} = \bar{\tau}_{[L_{p}^*, U_{p}^*]}$ 
is the optimal strategy for the problem \eqref{extended_problem}.
Hence $\tau_{[L_{p}^*, U_{p}^*]}$ is optimal for \eqref{value_function}.

(B) Suppose $\tilde{\mathcal{D}} = \varnothing$. Again because $\bar{V}_{p} > 0$ and it dominates the payoff function $G_p$, we must have $\mathcal{D}_p = \varnothing$, but this does not happen by Lemma \ref{lemma_stopping_empty}. This completes the proof for the put case.

(2) For the call case, similar arguments show that either (i) holds or otherwise  $\{ s > 0: \bar{V}_{c}(s) = s-K\} = \varnothing$. In the latter case, because $\bar{V}_{c} > 0$ and it dominates the payoff function $G_c$, we must have (ii). 


\end{proof}



\begin{remark} \label{remark_emptyset} 
	For $i = p,c$,  if $\mathcal{D}_{i, \infty} \neq \varnothing$ then we must have  $\mathcal{D}_{i} \neq \varnothing$ (because $V_{i,\infty}$ dominates $\bar{V}_{i}$).
\end{remark}
The following is immediate by this remark and Lemma \ref{lemma_stopping_empty}.
\begin{corollary}
We have $\mathcal{D}_{p,\infty} \neq \varnothing$ and $\tau_{[L_{p,\infty}^*, U_{p, \infty}^*]}$ solves the classical case for some
$0 <  L_{p,\infty}^* \leq U_{p,\infty}^* < K$.
\end{corollary}

While the stopping region for the call case can be empty, it is not empty as long as $r < \log \eee S_1$ as follows.
\begin{lemma} \label{lemma_nonempty_call}
	If $r < \log \eee S_1$, we have $\liminf_{s \rightarrow \infty}  ({V_c(s)} / {G_c(s)}) > 1$ and $\mathcal{D}_c \neq \varnothing$ with $U_c^* < \infty$.
\end{lemma} 
\begin{proof}
Consider the strategy $\tau_{[L, \infty]} \in \mathcal{A}$ 
for any $L > K$. Following the proof of Lemma \ref{lemma_stopping_empty}, we have $\tau_{[L/s, \infty]} \to T_1^\lambda$ and $S_{\tau_{[L/s, \infty]}} \to S_{T_1^\lambda}$ $\pp$-a.s.\ as $s \rightarrow \infty$. This, Fatou's lemma and \eqref{call_n_finite} give
\begin{align*}
\liminf_{s \rightarrow \infty} \frac {V_c(s)} s &\geq \liminf_{s \rightarrow \infty} \eee_s \Big[ e^{-r \tau_{[L, \infty]}} \frac {S_{\tau_{[L, \infty]}}- K} s 1_{\{ \tau_{[L, \infty]} < \infty \}} \Big] =  \liminf_{s \rightarrow \infty} \eee \Big[ e^{-r \tau_{[L/s, \infty]}}\frac {s S_{\tau_{[L/s, \infty]}}- K} s 1_{\{ \tau_{[L/s, \infty]} < \infty \}} \Big] \\
&\geq   \eee \Big[ e^{-r T_1^\lambda} \liminf_{s \rightarrow \infty} \Big(S_{\tau_{[L/s, \infty]}}- \frac K s \Big) \Big] = \eee \Big[ e^{-r T_1^\lambda} S_{T_1^\lambda} \Big] 
 =  \frac \lambda {\lambda + r - \log \eee S_1} > 1.
\end{align*}
Now because $\lim_{s \rightarrow \infty} (G_c(s) / s) = 1$ and  by Remark \ref{remark_V_G_vanish}, the proof is complete.
\end{proof}

\subsection{Convergence as $\lambda \rightarrow \infty$} 

Before concluding this section, we show the convergence to the classical case \cite{DPT} as the rate of observation $\lambda$ goes to infinity.
 Solely in this subsection, in order to spell out the dependence on the rate of observation,  for $i = p,c$, let $V_{i, \lambda}(s)$   be the value function, $L_{i,\lambda}^*$ and $U_{i,\lambda}^*$ the optimal barriers and $\mathcal{D}_{i, \lambda}$ the stopping region
 when the rate of observation is $\lambda > 0$. 
  Similarly, we  denote the value functions for the auxiliary case (see \eqref{extended_problem}) by $\bar{V}_{i, \lambda}(s)$.
 
Throughout this subsection, we assume the following.
\begin{assump} \label{assump_D_nonempty} For the call case,  we assume that $\mathcal{D}_{c, \infty} \neq \varnothing$. By Remark \ref{remark_emptyset}, this guarantees $\mathcal{D}_{c, \lambda} \neq \varnothing$ for each $\lambda > 0$.  
\end{assump}



\begin{lemma} \label{lemma_increasing_lambda}
Fix $i = p,c$. 
\begin{enumerate}
\item For each $s > 0$,  $\lambda \mapsto V_{i, \lambda}(s)$   is non-decreasing, and in particular $V_{i, \lambda}(s) \leq V_{i, \infty}(s)$.
\item We have $\lambda \mapsto L_{i,\lambda}^*$ is non-decreasing and $\lambda \mapsto  U_{i,\lambda}^*$ is non-increasing.
\end{enumerate}
\end{lemma}
\begin{proof}
We only prove for the put case; the same arguments can be applied to the call case.

(1)
Fix $\lambda_2 > \lambda_1$. The proof holds by the fact that the case with $\lambda_2$ has more opportunities than the case with $\lambda_1$.

More precisely, given $\mathcal{T}^{\lambda_1}$ the set of jump times of a Poisson process $N^{\lambda_1}$, consider its superset $\tilde{\mathcal{T}} = \mathcal{T}^{\lambda_1} \cup \mathcal{T}^{\lambda_2-\lambda_1}$,
where $\mathcal{T}^{\lambda_2-\lambda_1}$ is the set of jump times of a Poisson process $N^{\lambda_2-\lambda_1}$, independent of $N^{\lambda_1}$.
Then, the value function $V_{p,\lambda_2}$  can
be obtained by considering the set $\tilde{\mathcal{T}}$.  Because $\tilde{\mathcal{T}} \supset \mathcal{T}^{\lambda_1}$, we must have $V_{p, \lambda_1}(s) \leq V_{p, \lambda_2}(s)$.

(2) Again, fix $\lambda_2 > \lambda_1$. 
By modifying the proof of (1), the same is true for the auxiliary case (with $\mathcal{A}$ replaced by $\bar{\mathcal{A}}$), and hence
\begin{equation}\label{aux_sp}
 K-s\leq(K-s)^+\leq\bar{V}_{p, \lambda_1}(s) \leq \bar{V}_{p, \lambda_2}(s), \quad s > 0.	
\end{equation}
For $s \in \mathcal{D}_{p, \lambda_2} = [L_{p, \lambda_2}^*, U_{p, \lambda_2}^*] \subset (0,K) $, 
we have $\bar{V}_{p, \lambda_2}(s) = (K-s)^+ =K-s$ and hence \eqref{aux_sp} gives $\bar{V}_{p, \lambda_1}(s) = (K-s)^+ = K-s$ as well (i.e.\ $s \in \mathcal{D}_{p, \lambda_1} =   [L_{p, \lambda_1}^*, U_{p, \lambda_1}^*]$). This shows the claim.

\end{proof}

By using Lemma \ref{lemma_increasing_lambda}, we now show the following convergence results.

\begin{theorem} \label{theorem_convergence}
Fix $i = p,c$. 
\begin{enumerate}
\item The function $V_{i, \lambda}(s) \xrightarrow{\lambda \uparrow \infty} V_{i, \infty}(s)$ uniformly in compact sets on $(0,\infty)$.
\item We have $L_{i,\lambda}^* \xrightarrow{\lambda \uparrow \infty}  L_{i,\infty}^*$ and  $U_{i,\lambda}^* \xrightarrow{\lambda \uparrow \infty}  U_{i,\infty}^*$.
\end{enumerate}
\end{theorem}

\begin{proof}
(1) In this proof, we assume, for each $\lambda > 0$, $\mathcal{T}^\lambda= (\frac {T_1} \lambda, \frac {T_2} \lambda, ...)$ 
with $T_1, T_2, \ldots$ the arrival times of a common Poisson process with parameter $1$ independent of $X$. Note that this does not cause any issue because $\mathcal{T}^\lambda$ is the set of jump times of a Poisson process $N^\lambda$ with parameter $\lambda$ independent of $X$.




Fix $i = p,c$. By \cite{DPT} and Assumption \ref{assump_D_nonempty}, we have $V_{i, \infty}(s)=\eee_s  [e^{-r \tau_{[L_{i,\infty}^*, U_{i,\infty}^*]}} G_i (S_{\tau_{[L_{i,\infty}^*, U_{i,\infty}^*]}})  1_{\{ \tau_{[L_{i,\infty}^*, U_{i,\infty}^*]} < \infty \}} ]$ for
$\tau_{[L_{i,\infty}^*, U_{i,\infty}^*]}:=\inf\{t\geq 0: S_t\in \mathcal{D}_{i, \infty}  = [L_{i,\infty}^*, U_{i,\infty}^*]\}$. 
Denote the first Poisson arrival time after $\tau_{[L_{i,\infty}^*, U_{i,\infty}^*]}$ by $\tau^\lambda:= \inf \{T\in \mathcal{T}^\lambda: T\geq \tau_{[L_{i,\infty}^*, U_{i,\infty}^*]}\}$ (where it is understood that $\tau^\lambda = \infty$ if $\tau_{[L_{i,\infty}^*, U_{i,\infty}^*]} = \infty$)
and note that it is an $\mathbb{F}$-stopping time (i.e.\ belongs to the set $\mathcal{A}$). By Lemma \ref{lemma_increasing_lambda}(1),  for each $\lambda > 0$,
\begin{align} \label{V_i_bound}
V_{i, \infty}(s) \geq V_{i, \lambda}(s) \geq \eee_s  [e^{-r\tau^\lambda}G_i(S_{\tau^\lambda})  1_{\{ \tau^\lambda < \infty \}}].
\end{align}

By the memoryless property of the exponential random variable, we have $\tau^\lambda - \tau_{[L_{i,\infty}^*, U_{i,\infty}^*]} \sim {\rm Pois}(\lambda)$ conditionally on $\{\tau_{[L_{i,\infty}^*, U_{i,\infty}^*]} < \infty \}$ and hence
\[\sum_{\lambda=1}^\infty \pp (\tau^\lambda-\tau_{[L_{i,\infty}^*, U_{i,\infty}^*]}>\epsilon | \tau_{[L_{i,\infty}^*, U_{i,\infty}^*]} < \infty) = \sum_{\lambda=1}^\infty e^{-\lambda \epsilon}<\infty, \quad \epsilon > 0. 
\]
Hence, Borel-Cantelli lemma gives
$\tau^\lambda \xrightarrow{\lambda \uparrow \infty} \tau_{[L_{i,\infty}^*, U_{i,\infty}^*]}$
 a.s. on  $\{ \tau_{[L_{i,\infty}^*, U_{i,\infty}^*]} < \infty \}$. 


For the call case, notice that Assumption  \ref{assump_D_nonempty}  guarantees that $V_{c, \infty}$ is finite by \cite{DPT}.
In addition, because the \lev process $X$ has  right-continuous paths a.s., 
we have that $S_{\tau^\lambda} \xrightarrow{\lambda \uparrow \infty}S_{\tau_{[L_{i,\infty}^*, U_{i,\infty}^*]}}$ a.s.\ given $\{ \tau_{[L_{i,\infty}^*, U_{i,\infty}^*]} < \infty \}$.
Now Fatou's lemma and \eqref{V_i_bound} give 
\[
V_{c,\infty}(s) \geq \lim_{\lambda \to\infty}V_{c,\lambda}(s) \geq \liminf_{\lambda \to\infty}\eee_s [e^{-r\tau^\lambda}(S_{\tau^\lambda}-K)^+  1_{\{ \tau^\lambda < \infty \}}]= V_{c,\infty}(s).
\]		
The put case holds similarly.


To show the uniform convergence, by the continuity (implied by the convexity by Proposition \ref{proposition_convexity}), the monotonicity as in Lemma \ref{lemma_increasing_lambda} and Dini's theorem, the proof is complete.

(2) 
Fix $i = p,c$. 
By Lemma \ref{lemma_increasing_lambda}(2), we have
$L^\#_{i} := \lim_{\lambda \rightarrow \infty}L^*_{i,\lambda}  \leq L^*_{i,\infty}$. 
To derive a contradiction, let us assume $L^\#_{i} < L^*_{i,\infty}$ and choose $L^\#_{i} < \tilde{L} < L^*_{i,\infty}$. This implies that,   for all $\lambda >0$, $\tilde{L} \in \mathcal{D}_{i, \lambda}$ and hence $\bar{V}_{i,\lambda}(\tilde{L}) = G_i(\tilde{L})$. However,  this is a contradiction because (1) gives
$\lim_{\lambda \rightarrow \infty} \bar{V}_{i,\lambda}(\tilde{L}) = V_{i,\infty}(\tilde{L}) > G_i(\tilde{L})$, 
where the last strict inequality holds because $\tilde{L} \notin \mathcal{D}_{i, \infty}$.
 A similar argument shows that $\lim_{\lambda \rightarrow \infty}U^*_{i,\lambda}  = U^*_{i,\infty}$.


\end{proof}

%
%


		\section{Spectrally negative \lev processes and scale functions}
		\label{section_preliminaries}


Throughout this section, let us assume that the \lev process $X$ is spectrally negative, meaning  it has no positive jumps and that it is not the negative of a subordinator. For notational convenience, we define $\p_{x} := \pp_{\exp(x)}$, under which $X_0=x$ (equivalently $S_0 = e^x$).
In particular, we write  $\p$ in place of $\p_0$. Similarly, we write $\e_x := \eee_{\exp(x)}$ and $\e := \eee_{1}$.



We define the Laplace exponent of $X$ by
		\begin{align*}
			\psi(\theta):=\log \e\big[{\rm e}^{\theta X_1}\big] = \gamma\theta+\frac{\eta^2}{2}\theta^2+\int_{(-\infty,0)}\big({\rm e}^{\theta z}-1 -\theta z \mathbf{1}_{\{z > -1 \}}\big)\Pi(\ud z), \quad \theta \geq 0,
		\end{align*}
		%
		where $\gamma\in \R$, $\eta \ge 0$, and $\Pi$ is a \lev measure on $(-\infty,0)$ satisfying
		\[
		\int_{(-\infty,0)}(1\land z^2)\Pi(\ud z)<\infty.
		\]
		\subsection{Scale functions}
For
$q \geq 0$, we define the $q$-scale function as follows. Let \begin{align}
\Phi(q) := \sup \{ s > 0: \psi(s) = q\}. \label{Phi_big}
\end{align}
The scale function $W^{({q})}$ of $X$ is a mapping from $\R$ to $[0, \infty)$ that takes value zero on the negative half-line, while, on the positive half-line, it is a continuous and strictly increasing function defined by its Laplace transform:
		\begin{align} \label{scale_function_laplace}
			\begin{split}
				\int_0^\infty  \mathrm{e}^{-\theta z} W^{({q})}(z) \diff z &= \frac 1 {\psi(\theta)-q}, \quad \theta > \Phi({q}).
			\end{split}
		\end{align}

We define, for $\theta\in\R$ and $x \in \R$, 
\begin{align*}
	\overline{W}^{(q)} (x; \theta) := \int_0^{x}e^{- \theta z}W^{(q)}(z)\diff z
\end{align*}
and 
		\begin{align} \label{Z_theta}
			Z^{(q)}(x; \theta) &:=e^{\theta x} \left( 1 + (q- \psi(\theta ))\overline{W}^{(q)} (x; \theta)
\right),
		\end{align}
where for $\theta<0$ we only consider the case in which $\psi(\theta)$ can be defined by analytic extension.

				In particular, for $x \in \R$, we let $Z^{(q)}(x) =Z^{(q)}(x; 0)$ and, for $\lambda > 0$,
		\begin{align*} 
			Z^{(q+\lambda)}(x; \Phi(q)) =e^{\Phi(q) x} \left( 1 + \lambda
\overline{W}^{(q+\lambda)} (x; \Phi(q))
	\right).
			\end{align*}
\subsection{Related fluctuation identities} Here, we review several fluctuation identities that will be used later in this paper.

Fix $q \geq 0$. Let us denote, for $a \in \R$, the first passage times:
		\begin{align*}
	\tilde{T}_a^+ :=\inf\{t>0:X_t>a\}\quad \textrm{and} \quad \tilde{T}_a^- :=\inf\{t>0:X_t< a \}.
		\end{align*}
By  identity (8.11) of \cite{K} and Theorem 3.12 in \cite{K}, respectively, we have, for all $x \leq a$,
\begin{align}
\E_x \big[e^{-q \tilde{T}_a^+} 1_{\{\tilde{T}_0^- > \tilde{T}_a^+\}} \big] &= \frac {W^{(q)}(x)} {W^{(q)}(a)},  \label{upcrossing_identity_2} \\
\E_x \big[e^{-q \tilde{T}_a^+} 1_{\{\tilde{T}_a^+<\infty \}} \big] &= e^{-\Phi(q) (a-x)}. \label{upcrossing_identity}
\end{align}
By identity (5) in \cite{Albrecher} and as its limiting case, for $\theta \geq 0$, 
	\begin{align}\label{der_gamma_0_a}
\E_x\left[e^{-q \tilde{T}_0^-+\theta X_{\tilde{T}_0^-}}1_{\{\tilde{T}_0^-<\tilde{T}_a^+\}}\right]	&=Z^{(q)}(x;\theta)-\frac{W^{(q)}(x)}{W^{(q)}(a)}Z^{(q)}(a;\theta), \quad x \leq a, \\
\E_{x} \left[ e^{-q \tilde{T}_0^- + \theta X_{\tilde{T}_0^-}}1_{\{\tilde{T}_0^-<\infty\}}\right]
&=Z^{({q})}(x;\theta)-\frac{\psi(\theta)-q}{\theta-\Phi(q)}W^{({q})}(x), \quad x \in \R, \notag
\end{align}
where the case $\theta = \Phi(q)$ is understood to be the limiting case.

From Corollaries 8.7 and 8.8 of \cite{K}, for any Borel set $A\subseteq[0,a]$
and
$B\subseteq(-\infty,a]$, respectively,
\begin{align}
	\label{resolvent_density_0}
	\E_x \Big[ \int_0^{\tilde{T}_{0}^-\wedge \tilde{T}_a^+ } e^{-qt} 1_{\left\{ X_t \in A \right\}} \diff t\Big] &= \int_A u^{(q)}(x,y;a) \diff y, \quad 0\leq x \leq a, \\
	\E_x \Big[ \int_0^{\tilde{T}_{a}^+ } e^{-qt} 1_{\left\{ X_t \in B \right\}} \diff t\Big] &= \int_{B} r^{(q)}(x,y;a) \diff y, \quad x \leq a, \label{killed_resolvent}
\end{align}
where
\begin{align*}
u^{(q)}(x,y;a) &:= \frac{W^{(q)}(x)}{W^{(q)}(a)}{W^{(q)}} (a-y) -{W^{(q)}} (x-y), \qquad\text{$0\leq y\leq a$,} \\
r^{(q)}(x,y; a)&:=e^{-\Phi(q)(a-x)}W^{(q)}(a-y)-W^{(q)}(x-y),\qquad\text{$ y\leq a$.}
\end{align*}







\section{first entry time to an interval under Poisson observation} \label{section_hitting_time}

		
		

Throughout this section, we continue to assume that $X$ is a spectrally negative L\'evy process, $\lambda > 0$ and $q \geq 0$ (which will be extended to the case $q < 0$ in later sections). Recall $\mathcal{T}^\lambda := (T_n^\lambda; n \geq 1)$ is the set of jump times of an independent Poisson process $N^\lambda$.
	Recall \eqref{barrier_strategies} and consider
		\begin{align*}
			\tilde{\tau}_{[l,u]}  := \tau_{[e^l, e^u]} =  \inf \left\{ 
			T \in \mathcal{T}^\lambda: X_{T} \in[l,u] \right\}, \quad  l<u. 
		\end{align*}

Define for $x,a \in \mathbb{R}$ and $\theta \in \R$ for which $\psi(\theta)$ is well-defined,
		\begin{align}\label{mathcal_L}
			\mathscr{Z}^{(q, \lambda)}_a (x; \theta) &:= Z^{(q+\lambda)}(x;\theta)-\lambda\int_a^xW^{(q)}(x-y)Z^{(q+\lambda)}(y;\theta)\diff y \notag\\
			&=Z^{(q)}(x;\theta)+\lambda\int_0^aW^{(q)}(x-y)Z^{(q+\lambda)}(y;\theta)\diff y,
		\end{align}
where the second equality can be obtained similarly as identity (7) of \cite{LRZ}.
	In particular,
						\begin{align*}
			\mathscr{Z}^{(q, \lambda)}_a (x; \theta) = Z^{(q+\lambda)}(x;\theta), \quad  x \leq a.
		\end{align*}
			By Lemma 2.1 in \cite{LRZ}, we have that, for $\theta\geq0$, $a<b$, and $x \leq b$,
			\begin{align}\label{LRZ_Z}
				\E_x\left[e^{-q \tilde{T}_a^-}Z^{(q+\lambda)}(X_{\tilde{T}_a^-};\theta)1_{\{\tilde{T}_a^-<\tilde{T}_b^+\}}\right]&=\mathscr{Z}^{(q, \lambda)}_a (x; \theta)-\frac{W^{(q)}(x-a)}{W^{(q)}(b-a)}\mathscr{Z}^{(q, \lambda)}_a (b; \theta).
			\end{align}
We also define
\begin{align}\label{fun_L_new}
L^{(q,\lambda)}(x,a;\theta) &:=   e^{\theta x} \Big( \overline{W}^{(q)} (x-a; \theta) - \overline{W}^{(q+\lambda)} (x; \theta)+\lambda\int_0^{x-a}W^{(q)}(y)e^{-\theta y} \overline{W}^{(q+\lambda)} (x-y; \theta) \diff y \Big),
\end{align}
which can be simplified as follows;
its proof is deferred to Appendix \ref{appen_lemma_4.1_a}. 
\begin{lemma} \label{lemma_L_same}
For $\theta \in \R$ 
such that $\psi(\theta) \neq q+\lambda$ and $\psi(\theta)$ is well defined, we have for $a,x \in \R$
\begin{align}\label{fun_L}
L^{(q,\lambda)}(x,a;\theta) = \frac {e^{\theta a}Z^{(q)}(x-a;\theta)-\mathscr{Z}^{(q, \lambda)}_a (x; \theta)} {q+\lambda - \psi(\theta)}.
\end{align}
\end{lemma}
	Note that in view of \eqref{fun_L_new}, we have
\begin{align}\label{fun_L_new_below_a}
L^{(q,\lambda)}(x,a;\theta) &=  - e^{\theta x}  \overline{W}^{(q+\lambda)} (x; \theta), \quad x \leq a, \; \theta \in \R.
\end{align}

The following two theorems are the main results of this section.
\begin{theorem} \label{theorem_g}
	We have, for $0 < a < b$, $x \leq b$, and $\theta \in \R$, 
\begin{align}\label{fun_g}
\begin{split}
			\E_x\left[e^{-q \tilde{\tau}_{[0,a]}+\theta X_{\tilde{\tau}_{[0,a]}}}1_{\{\tilde{\tau}_{[0,a]}<\tilde{T}_b^+\}}\right] &=
			\lambda \Big[ L^{(q,\lambda)}(x,a;\theta)
			 - L^{(q,\lambda)}(b,a;\theta)
			\frac{\mathscr{Z}^{(q, \lambda)}_a (x; \Phi(q))}{\mathscr{Z}^{(q, \lambda)}_a (b; \Phi(q))} \Big].
			\end{split}
		\end{align}
\end{theorem}

\begin{theorem} \label{theorem_limiting_case}
We have, for $a>0$, $\theta \in \R$, and $x\in\R$,
\begin{align}\label{flu_int}
\begin{split}
\E_x\Big[&e^{-q \tilde{\tau}_{[0,a]}+\theta X_{\tilde{\tau}_{[0,a]}}}1_{\{\tilde{\tau}_{[0,a]}<\infty\}}\Big]=
\lambda L^{(q,\lambda)}(x,a; \theta) -M^{(q,\lambda)}(a; \theta)
\frac{ \mathscr{Z}^{(q, \lambda)}_a (x; \Phi(q)) }
{N^{(q,\lambda)}(a)},
\end{split}
\end{align}
where, for $y\in\R$, 
\begin{align}\label{fun_M}
M^{(q,\lambda)}(y; \theta) &:= \left\{ \begin{array}{ll}
\frac{\lambda e^{-\Phi(q)y}}{\Phi(q)-\theta}\Big[e^{\theta y}+\lambda e^{\theta y}\overline{W}^{(q+\lambda)} (y; \theta) -Z^{(q+\lambda)}(y;\Phi(q)) \Big], & \theta\neq \Phi(q),  \\  -\lambda \Big[ y +\lambda \int_0^{y} (y-z) e^{-\Phi(q) z}W^{(q+\lambda)}(z) \diff z \Big], & \theta = \Phi(q), \end{array} \right. \\
N^{(q,\lambda)}(y) &:= \psi'(\Phi(q))+ \lambda \int_0^{y}e^{-\Phi(q)z}Z^{(q+\lambda)}(z;\Phi(q))\diff z. \nonumber
\end{align}
\end{theorem}

By taking $a\downarrow 0$ in Theorem \ref{theorem_limiting_case}, we have the following. 
\begin{corollary} \label{cor_S_at_zero_inf}  We have $\inf \{\tau \in \mathcal{T}^{\lambda}: X_{\tau} = 0 \} = \infty$ $\p_x$-a.s. for all $x \in \R$.
\end{corollary}

The next subsection is devoted to the proof of Theorem \ref{theorem_g}.
Theorem \ref{theorem_limiting_case} can be proved by taking $b \rightarrow \infty$ and its proof is
deferred to Appendix \ref{appen_Theorem_4.2}.

\subsection{Proof of Theorem \ref{theorem_g}}

Throughout this proof, we denote
		\begin{align*}
		g(x,a,b; \theta) :=\E_x\left[e^{-q \tilde{\tau}_{[0,a]}+\theta X_{\tilde{\tau}_{[0,a]}}}1_{\{\tilde{\tau}_{[0,a]}<\tilde{T}_b^+\}}\right], \quad 0 < a < b, \; x \in \R, \;  \textrm{and } \theta \geq  0.
\end{align*}
Also  let the random variable $\mathbf{e}_\lambda$ be an exponential random variable with parameter $\lambda$
independent of $X$.

(1) Suppose that $\theta\geq0$ and $\psi(\theta) \neq q + \lambda$. 


By \eqref{upcrossing_identity} and
because $X$ does not have positive jumps, we have
	\begin{align}\label{g_less_0}
		g(x,a,b; \theta)=g(0,a,b; \theta)\E_x\left[e^{-q \tilde{T}_0^+}1_{\{\tilde{T}_0^+<\infty\}}\right]=g(0,a,b; \theta)e^{\Phi(q)x}, \quad x  \leq  0.
	\end{align}
Below, we compute $g(0,a,b; \theta)$ to obtain an explicit expression of \eqref{g_less_0}, and then $g(x,a,b; \theta)$ for $x > 0$.

	(i) First we will write $g(0, a,b; \theta)$ in terms of $g(a, a,b; \theta)$. To this end we decompose:
	\begin{align}\label{g_0_ub}
	g(0,a,b; \theta)&= g_1 +g_2 + g_3,
	\end{align}
	where, with $\mathbf{e}_\lambda$ modeling the first jump time of $N^\lambda$, by the strong Markov property and the fact that $X_{\tilde{T}_a^+} = a$ on $\{ \tilde{T}_a^+ < \infty \}$ when $X_0 = 0 <  a,$  
	\begin{align*}
	g_1 &:= \E\left[e^{-q\mathbf{e}_\lambda+\theta X_{\mathbf{e}_\lambda}}1_{\{0 \leq X_{\mathbf{e}_\lambda} \leq a ,\mathbf{e}_\lambda<\tilde{T}_a^+\}}\right], \\
	g_2 &:= \E \left[e^{-q\tilde{T}_a^+}1_{\{\tilde{T}_a^+<\mathbf{e}_\lambda\}}\right] g(a, a,b; \theta), \\
	g_3 &:= \E\left[e^{-q\mathbf{e}_\lambda} g( X_{\mathbf{e}_\lambda},a,b; \theta)1_{\{X_{\mathbf{e}_\lambda}<0,\mathbf{e}_\lambda<\tilde{T}_a^+\}}\right].
	\end{align*}
	Here, using  \eqref{killed_resolvent} and \eqref{upcrossing_identity}, respectively,
	we have that 
	\begin{align*}
	g_1&= \lambda \E\left[ \int_0^{\tilde{T}_a^+} e^{-\lambda t} e^{-q t+\theta X_{t}}1_{\{0 \leq X_{t} \leq a \}}\diff t \right]  = \lambda \int_0^ae^{\theta y}
r^{(q+\lambda)}(0,y;a)
\diff y =e^{-(\Phi(q+\lambda) - \theta) a} \lambda \overline{W}^{(q+\lambda)} (a; \theta), \\
	g_2 &= \E \left[e^{-(q+ \lambda) \tilde{T}_a^+} \right]  g(a,a,b; \theta) =g(a,a,b; \theta)e^{-\Phi(q+\lambda)a}.
	\end{align*}
	From \eqref{g_less_0} and using \eqref{killed_resolvent} together with the spatial homogeneity of $X$, we have 
	\begin{multline} \label{g_3_step}
	g_3
	=g(0,a,b; \theta)  \E\left[e^{-q\mathbf{e}_\lambda+\Phi(q)X_{\mathbf{e}_\lambda}}1_{\{X_{\mathbf{e}_\lambda}<0,\mathbf{e}_\lambda< \tilde{T}_a^+\}}\right]\\
	=g(0,a,b; \theta) \lambda \E\Big[ \int_0^{\tilde{T}_a^+} e^{-\lambda t} e^{-q t+\Phi(q)X_{t}}1_{\{X_{t}<0 \}}\diff t \Big]
	=g(0,a,b; \theta) \lambda \int_{-\infty}^0e^{\Phi(q) y}
r^{(q+\lambda)}(0,y;a)
\diff y.
	\end{multline}
	We will simplify \eqref{g_3_step} using the lemma below. Its proof is deferred to Appendix \ref{appen_lemma_4.1}.
	\begin{lemma} \label{lemma_resolvent}
		For $\beta \geq \alpha>0$, we have 
		\begin{align}\label{g_0_ub_2}
			\lambda \int_{-\infty}^0e^{\Phi(q) y} r^{(q+\lambda)}(\alpha,y;\beta)
\diff y=Z^{(q+\lambda)}(\alpha;\Phi(q))-e^{-\Phi(q+\lambda)(\beta-\alpha)}Z^{(q+\lambda)}(\beta;\Phi(q)).
		\end{align}
	\end{lemma}
Now by \eqref{g_3_step} and Lemma \ref{lemma_resolvent}, we have
	$g_3 =g(0,a,b; \theta) [ 1-e^{-\Phi(q+\lambda)a}Z^{(q+\lambda)}(a;\Phi(q))]$. Substituting these values of $g_i$, $1 \leq i \leq 3$ in \eqref{g_0_ub}, we get, after simplification,
%
%
%
	\begin{align}\label{g_0_0}
	g(0,a,b; \theta)=\frac{ \lambda e^{\theta a} \overline{W}^{(q+\lambda)} (a; \theta)+g(a,a,b; \theta)}{Z^{(q+\lambda)}(a;\Phi(q))}.
	\end{align}
	(ii) 
	For $x\leq a$, by an application of the strong Markov property and \eqref{g_less_0}, 
	\begin{align}
\label{g_deecomposition}
\begin{split}
		g(x,a,b; \theta)
		&=\E_x\left[e^{-q \tilde{T}_a^+}1_{\{\tilde{T}_a^+<\mathbf{e}_\lambda\wedge \tilde{T}_0^-\}}\right] g(a,a,b; \theta) \\
		&+\E_x\left[e^{-q \tilde{T}_0^-+\Phi(q)X_{\tilde{T}_0^-}}1_{\{\tilde{T}_0^-<\mathbf{e}_\lambda\wedge \tilde{T}_a^+\}}\right] g(0,a,b; \theta) + \E_x\left[e^{-q\mathbf{e}_\lambda+\theta X_{\mathbf{e}_\lambda}}1_{\{\mathbf{e}_\lambda< \tilde{T}_0^-\wedge \tilde{T}_a^+\}}\right].
\end{split}
	\end{align}
	Here, we have, by \eqref{upcrossing_identity_2},
	\begin{align*}
		\E_x\left[e^{-q \tilde{T}_a^+}1_{\{\tilde{T}_a^+<\mathbf{e}_\lambda\wedge \tilde{T}_0^-\}}\right]=\E_x\left[e^{-(q+\lambda) \tilde{T}_a^+}1_{\{\tilde{T}_a^+< \tilde{T}_0^-\}}\right]=\frac{W^{(q+\lambda)}(x)}{W^{(q+\lambda)}(a)}, \quad x \leq a.
	\end{align*}
	By \eqref{der_gamma_0_a}, we have
	\begin{align*}
		\E_x\left[e^{-q \tilde{T}_0^-+\Phi(q)X_{ \tilde{T}_0^-}}1_{\{ \tilde{T}_0^-<\mathbf{e}_\lambda\wedge \tilde{T}_a^+\}}\right]&=\E_x\left[e^{-(q+\lambda) \tilde{T}_0^- + \Phi(q)X_{ \tilde{T}_0^-}}1_{\{\tilde{T}_0^-< \tilde{T}_a^+\}}\right]\\
		&=Z^{(q+\lambda)}(x;\Phi(q))-\frac{W^{(q+\lambda)}(x)}{W^{(q+\lambda)}(a)}Z^{(q+\lambda)}(a;\Phi(q)), \quad x \leq a.
	\end{align*}
	On the other hand, 
	by \eqref{resolvent_density_0}, 
	\begin{multline*}
		\E_x\left[e^{-q\mathbf{e}_\lambda+\theta X_{\mathbf{e}_\lambda}}1_{\{\mathbf{e}_\lambda< \tilde{T}_0^-\wedge \tilde{T}_a^+\}}\right]= \lambda \E_x\left[\int_0^{\tilde{T}_0^-\wedge \tilde{T}_a^+}e^{-(q+\lambda)s+\theta X_s} \diff s \right]\\
		= \lambda \int_0^ae^{\theta y}
u^{(q+\lambda)}(x,y;a)
\diff y =\lambda \left(\frac{W^{(q+\lambda)}(x)}{W^{(q+\lambda)}(a)}e^{\theta a}\overline{W}^{(q+\lambda)} (a; \theta)-e^{\theta x}\overline{W}^{(q+\lambda)} (x; \theta)\right), \quad x \leq a.
	\end{multline*}
Substituting these expressions and \eqref{g_0_0} in \eqref{g_deecomposition}, we get, for 
$x\leq a$,
	\begin{align} \label{g_0_a_2}
	\begin{split}
		g(x,a,b; \theta)&=g(a,a,b; \theta)\frac{W^{(q+\lambda)}(x)}{W^{(q+\lambda)}(a)} +g(0,a,b; \theta)\Big(Z^{(q+\lambda)}(x;\Phi(q))-\frac{W^{(q+\lambda)}(x)}{W^{(q+\lambda)}(a)}Z^{(q+\lambda)}(a;\Phi(q))\Big) \\
&+ \lambda \Big(\frac{W^{(q+\lambda)}(x)}{W^{(q+\lambda)}(a)}e^{\theta a} \overline{W}^{(q+\lambda)} (a; \theta)-e^{\theta x}\overline{W}^{(q+\lambda)} (x; \theta)\Big)\\
		&=\frac{Z^{(q+\lambda)}(x;\Phi(q))}{Z^{(q+\lambda)}(a;\Phi(q))}\left(\lambda e^{\theta a}\overline{W}^{(q+\lambda)} (a; \theta)+g(a,a,b; \theta)\right)-\lambda e^{\theta x} \overline{W}^{(q+\lambda)} (x; \theta) \\
		&=\frac{Z^{(q+\lambda)}(x;\Phi(q))}{Z^{(q+\lambda)}(a;\Phi(q))}\left(\lambda e^{\theta a} \overline{W}^{(q+\lambda)} (a; \theta)+g(a,a,b; \theta)\right)-\frac{\lambda \left(Z^{(q+\lambda)}(x;\theta)-e^{\theta x}\right)}{q+\lambda-\psi(\theta)}.
		\end{split}
	\end{align}

	On the other hand, for $x \in \R$,
the strong Markov property gives
	\begin{align}\label{g_x>a}
		g(x,a,b; \theta)=\E_x\left[e^{-q \tilde{T}_a^-}g(X_{\tilde{T}_a^-}, a,b;\theta)1_{\{\tilde{T}_a^-<\tilde{T}_b^+\}}\right].
	\end{align}
	Using \eqref{g_0_a_2} in \eqref{g_x>a} together with \eqref{der_gamma_0_a} and \eqref{LRZ_Z}, we obtain, for $x\leq b$,
	\begin{align}\label{g_0_a_3}
	\begin{split}
		g(x,a,b; \theta)
		&=\frac{\lambda e^{\theta a}\overline{W}^{(q+\lambda)} (a; \theta)+g(a,a,b; \theta)}{Z^{(q+\lambda)}(a;\Phi(q))} \Big[\mathscr{Z}^{(q, \lambda)}_a (x; \Phi(q))-\frac{W^{(q)}(x-a)}{W^{(q)}(b-a)}\mathscr{Z}^{(q, \lambda)}_a (b; \Phi(q))\Big] \\
		&-\frac{\lambda}{q+\lambda-\psi(\theta)}\Big[\mathscr{Z}^{(q, \lambda)}_a (x; \theta)-\frac{W^{(q)}(x-a)}{W^{(q)}(b-a)}\mathscr{Z}^{(q, \lambda)}_a (b; \theta)\Big] \\
		&+\frac{\lambda}{q+\lambda-\psi(\theta)}\Big[Z^{(q)}(x-a;\theta)-\frac{W^{(q)}(x-a)}{W^{(q)}(b-a)}Z^{(q)}(b-a;\theta)\Big]e^{\theta a}.
		\end{split}
	\end{align}

	(iii) Now let us compute $g(a,a,b; \theta)$.
To this end, we note that the strong Markov property gives
	\begin{align}\label{g(a)_00}
	g(a,a,b; \theta)&= h_1 + h_2 + h_3,
	\end{align}
	where 
	\begin{align*}
	h_1 &:= \E_a\left[e^{-q\mathbf{e}_\lambda+\theta X_{\mathbf{e}_\lambda}} 1_{\{0 \leq X_{\mathbf{e}_\lambda} < a, \mathbf{e}_\lambda<\tilde{T}_b^+\}}\right], \\
	h_2 &:= \E_a\left[e^{-q\mathbf{e}_\lambda}g(X_{\mathbf{e}_\lambda},a,b; \theta)1_{\{X_{\mathbf{e}_\lambda}\geq a, \mathbf{e}_\lambda< \tilde{T}_b^+\}}\right], \\
	h_3 &:= \E_a\left[e^{-q\mathbf{e}_\lambda}g(X_{\mathbf{e}_\lambda},a,b; \theta)1_{\{X_{\mathbf{e}_\lambda}<0, \mathbf{e}_\lambda< \tilde{T}_b^+\}}\right].
	\end{align*}
	
First, using \eqref{killed_resolvent},
	\begin{align}\label{g(a)_6}
	h_1
	&=\lambda \E_a\left[\int_0^{\tilde{T}_b^+}e^{-(q+\lambda)s}e^{\theta X_s}1_{\{0 \leq X_s < a \}}\diff s\right] = \lambda \int_0^ae^{\theta y}
r^{(q+\lambda)}(a,y;b)
\diff y\notag\\
	&=\lambda \frac{e^{-\Phi(q+\lambda)(b-a)}}{q+ \lambda-\psi(\theta)}\left(Z^{(q+\lambda)}(b;\theta)-e^{\theta a}Z^{(q+\lambda)}(b-a;\theta)\right) - \lambda e^{\theta a} \overline{W}^{(q+\lambda)} (a; \theta).
	\end{align}
	
	Second, by \eqref{killed_resolvent},
	\begin{align}\label{g(a)_0}
	h_2&=\lambda \E_a \left[\int_0^{\tilde{T}_b^+}e^{-(q+\lambda)s}g(X_s,a,b; \theta)1_{\{X_s\geq a\}}\diff s\right] =\lambda e^{-\Phi(q+\lambda)(b-a)}\int_a^bg(y,a,b; \theta)W^{(q+\lambda)}(b-y)\diff y.
	\end{align}
To compute the previous expression, we will use the following result. Its proof is deferred to Appendix \ref{aux_lemma}.
	\begin{lemma}\label{lemma_aux}
		For $a\leq b$, $\theta\in\R$ for which $\psi(\theta)$ is well defined, we have
	\begin{align}\label{g(a)_2}
			\begin{split}
				\lambda \int_a^bW^{(q+\lambda)}(b-y)\Bigg[\mathscr{Z}^{(q, \lambda)}_a (y; \theta) -&\frac{W^{(q)}(y-a)}{W^{(q)}(b-a)}\mathscr{Z}^{(q, \lambda)}_a (b; \theta)\Bigg]\diff y \\&=Z^{(q+\lambda)}(b;\theta)-\frac{W^{(q+\lambda)}(b-a)}{W^{(q)}(b-a)} \mathscr{Z}^{(q, \lambda)}_a (b; \theta),
			\end{split} \\
	\label{g(a)_3}
\begin{split}
			\lambda \int_a^bW^{(q+\lambda)}(b-y)\Bigg[Z^{(q)}(y-a;\theta)&-\frac{W^{(q)}(y-a)}{W^{(q)}(b-a)}Z^{(q)}(b-a;\theta)\Bigg]\diff y \\
			&=Z^{(q+\lambda)}(b-a;\theta)-\frac{W^{(q+\lambda)}(b-a)}{W^{(q)}(b-a)}Z^{(q)}(b-a;\theta).
\end{split}
		\end{align}
	\end{lemma}
	

Now applying \eqref{g_0_a_3}, \eqref{g(a)_2}, and \eqref{g(a)_3} in \eqref{g(a)_0} we obtain that 
	\begin{align}\label{g(a)_4}
	h_2
	&=\frac{e^{-\Phi(q+\lambda)(b-a)}}{Z^{(q+\lambda)}(a;\Phi(q))}\left(\lambda e^{\theta a} \overline{W}^{(q+\lambda)} (a; \theta)+g(a,a,b; \theta)\right) \\ & \times \Bigg[Z^{(q+\lambda)}(b;\Phi(q)) -\frac{W^{(q+\lambda)}(b-a)}{W^{(q)}(b-a)}\mathscr{Z}^{(q, \lambda)}_a (b; \Phi(q))\Bigg]\notag\\
	&- \lambda \frac{e^{-\Phi(q+\lambda)(b-a)}}{q+\lambda-\psi(\theta)}\Bigg[Z^{(q+\lambda)}(b;\theta)-\frac{W^{(q+\lambda)}(b-a)}{W^{(q)}(b-a)}\mathscr{Z}^{(q, \lambda)}_a (b; \theta) \Bigg]\notag\\
	&+ \lambda \frac{e^{-\Phi(q+\lambda)(b-a)}}{q+\lambda-\psi(\theta)}\Bigg[Z^{(q+\lambda)}(b-a;\theta)-\frac{W^{(q+\lambda)}(b-a)}{W^{(q)}(b-a)}Z^{(q)}(b-a;\theta)\Bigg]e^{\theta a}. \notag
	\end{align}
	
Third, using \eqref{g_less_0} together with \eqref{killed_resolvent}, and then using \eqref{g_0_0} and Lemma \ref{lemma_resolvent},
	\begin{align} \label{g(a)_5}
	\begin{split}
	h_3
	&=g(0,a,b; \theta) \lambda \int_{-\infty}^0e^{\Phi(q)y}
r^{(q+\lambda)}(a,y;b)
\diff y \\
	&=\frac{\lambda e^{\theta a}\overline{W}^{(q+\lambda)} (a; \theta)+g(a,a,b; \theta)}{Z^{(q+\lambda)}(a;\Phi(q))} \left(Z^{(q+\lambda)}(a;\Phi(q))-e^{-\Phi(q+\lambda)(b-a)}Z^{(q+\lambda)}(b;\Phi(q))\right). 
	\end{split}
	\end{align}

	Hence applying  \eqref{g(a)_6}, \eqref{g(a)_4}, and \eqref{g(a)_5} in \eqref{g(a)_00} and then solving for $g(a,a,b; \theta)$,
	\begin{align*}
	g(a,a,b; \theta)&\frac{e^{-\Phi(q+\lambda)(b-a)}}{Z^{(q+\lambda)}(a;\Phi(q))}\frac{W^{(q+\lambda)}(b-a)}{W^{(q)}(b-a)}\mathscr{Z}^{(q, \lambda)}_a (b; \Phi(q)) \notag\\
	&=-\frac{e^{-\Phi(q+\lambda)(b-a)}}{Z^{(q+\lambda)}(a;\Phi(q))} \lambda e^{\theta a} \overline{W}^{(q+\lambda)} (a; \theta)\frac{W^{(q+\lambda)}(b-a)}{W^{(q)}(b-a)} \mathscr{Z}^{(q, \lambda)}_a (b; \Phi(q)) \\
	&+ \lambda \frac{e^{-\Phi(q+\lambda)(b-a)}}{q+ \lambda -\psi(\theta)}\frac{W^{(q+\lambda)}(b-a)}{W^{(q)}(b-a)} \mathscr{Z}^{(q, \lambda)}_a (b; \theta) - \lambda \frac{e^{-\Phi(q+\lambda)(b-a)}}{q+\lambda-\psi(\theta)}\frac{W^{(q+\lambda)}(b-a)}{W^{(q)}(b-a)}e^{\theta a}Z^{(q)}(b-a;\theta),
	\end{align*}
yielding
	\begin{align}\label{g_a_expl}
		g(a,a,b; \theta)&=- \lambda e^{\theta a} \overline{W}^{(q+\lambda)} (a; \theta)\notag\\
		&+ \lambda \frac{Z^{(q+\lambda)}(a;\Phi(q))}{q+\lambda -\psi(\theta)}\frac{\mathscr{Z}^{(q, \lambda)}_a (b; \theta)}{\mathscr{Z}^{(q, \lambda)}_a (b; \Phi(q))}-\frac{ \lambda e^{\theta a}}{q+\lambda-\psi(\theta)}\frac{Z^{(q+\lambda)}(a;\Phi(q))}{\mathscr{Z}^{(q, \lambda)}_a (b; \Phi(q))}Z^{(q)}(b-a;\theta).
	\end{align}
		Now, using \eqref{g_a_expl} in \eqref{g_0_a_3}, we obtain \eqref{fun_g}, with $L^{(q,\lambda)}$ as in \eqref{fun_L}.
By Lemma  \ref{lemma_L_same}, the claim holds for $\theta\geq0$ and $\psi(\theta) \neq q + \lambda$. 



(2)  By the fact that $X_{\tilde{\tau}_{[0,a]}}$ is bounded and $L^{(q,\lambda)}$ is well defined for all $\theta \in \R$, we can use analytic continuation and the identity \eqref{fun_g} still holds for $\theta<0$ as well. In addition, the identity holds when $\psi(\theta) = q + \lambda$ because $L^{(q,\lambda)}$ is continuous in $\theta$ in view of the expression \eqref{fun_L_new}. 

\exit

\section{American put options} \label{section_put_SN}

This section considers the American put options for both the spectrally negative and positive cases.
Recall as in Theorem \ref{theorem_barrier_optimal} that, under Assumptions \ref{assump_lambda_r} and \ref{assump_tail_value_function}, 
there exist $0 < L_{p}^* \leq L_{p,\infty}^* \leq U_{p,\infty}^* \leq U_{p}^* < K$
such that $\tau_{[L_{p}^*, U_{p}^*]} = \tilde{\tau}_{[\log L_{p}^*, \log U_{p}^*]}$ is the optimal strategy.  In this section, we pursue more explicit values of
\begin{align}
l_p^* := \log L_p^* \quad \textrm{and} \quad u_p^* := \log U_p^* \label{opt_barrier_log}
\end{align}
 by considering the first-order conditions. Throughout this section, we continue to assume \green{\eqref{assump_negative_r} and} Assumption \ref{assump_lambda_r} and make sufficient conditions for Assumption \ref{assump_tail_value_function} that are easy to check. 
 
\begin{remark} By Corollary \ref{cor_S_at_zero_inf}, we must have $u_p^* \neq l_p^*$.
\end{remark}


\subsection{Preliminaries} \label{subsection_preliminary}
Before solving the spectrally negative and positive
 cases, we consider  the case of a more general payoff function
\begin{align} \label{v_p_theta}
v_p (x;l,u,\theta) &:=\E_x\Big[e^{-r\tilde{\tau}_{[l,u]}}(K-e^{\theta X_{\tilde{\tau}_{[l,u]}}})1_{\{\tilde{\tau}_{[l,u]}<\infty\}}\Big], \quad  l < u \; \textrm{and} \; \theta, x \in \mathbb{R},
\end{align}
for a spectrally negative \lev process $X$ with Laplace exponent $\psi$ as defined in Section \ref{section_preliminaries}. With the flexibility of choosing $\theta$, the results can be used for both the spectrally negative and positive cases in Sections  \ref{subsection_SN} and  \ref{subsection_SP}, respectively. In particular, we obtain the first-order conditions of $v_p (x;l,u,\theta)$ with respect to $l$ and $u$.


Because we consider the discount $r < 0$,  
we extend the domain of the Laplace exponent $\psi$ from $[0,\infty)$ to $(A, \infty)$ where 
$A:= \inf \{\theta \in \mathbb{R}:\int_{\{y < -1\}}e^{\theta y}\Pi( \diff y)<\infty\}$, and that of $\Phi$ as in  \eqref{Phi_big} and
let
\[	
\Phi(r)=\sup\{s>A:\psi(s)=r\},
\]
and assume immediately below that it is well-defined. 
Note that $\Phi(r) \neq 0$ because $\psi(0) = 0 \neq r$.

\begin{remark} \label{sf_extension}
	By Lemma 8.3 in \cite{KKR}, the scale function $W^{(r)}(x)$ can be extended to  $r \in  \mathbb{C}$. As a particular case we can consider $r<0$.
\end{remark}



As in \cite{DPT}, we assume the following.
\begin{assump}\label{assum_phi} 
	We assume that $\Phi(r)$ is well-defined throughout this section.
\end{assump}

This assumption is necessary to make sure that the obtained expression in Theorem \ref{theorem_limiting_case} makes sense and is also used to guarantee that Assumption \ref{assump_tail_value_function}  holds (see Lemmas \ref{lemma_tail_SN} and \ref{lemma_tail_SN} for the spectrally negative and positive cases, respectively). See Figures \ref{figure_psi_SN} and \ref{figure_psi_SP} for sample plots of $\psi$ for the illustration of when Assumption \ref{assum_phi} is satisfied.





By Remark \ref{sf_extension} and Assumption \ref{assum_phi}, Theorem \ref{theorem_limiting_case} holds for the considered $r < 0$ and we have the following.

\begin{proposition} \label{proposition_v_p_SN}
For $l<u$, $\theta \in \mathbb{R}$, and $x\in\R$,
\begin{align}\label{v_p}
\begin{split}
v_p (x;l,u,\theta) 
&= \lambda \Big(K  L^{(r,\lambda)}(x-l,u-l; 0) -  e^{\theta l}L^{(r,\lambda)}(x-l,u-l; \theta) \Big) + \tilde{v}_p(l,u;\theta)   e^{\Phi(r) l} \mathscr{Z}^{(r, \lambda)}_{u-l} (x-l; \Phi(r)),
\end{split}
\end{align}
where
\begin{align}\label{v<l}
\tilde{v}_p(l,u;\theta)
:=  e^{-\Phi(r)l}v_p(l;l,u, \theta) = \frac{[e^{\theta l} M^{(r,\lambda)}(u-l;\theta) - K M^{(r,\lambda)}(u-l; 0)] e^{-\Phi(r)l}}{N^{(r,\lambda)}(u-l) }, \quad l < u.
\end{align}
In particular, for $x<l$, we have that 
\begin{align}\label{v_x<l}
\begin{split}
v_p(x;l,u,\theta)
&= e^{\Phi(r)x} \tilde{v}_p(l,u;\theta).
\end{split}
\end{align}
\end{proposition}

\begin{proof}
Using Theorem \ref{theorem_limiting_case} and the spatial homogeneity of L\'evy processes, for $l< u$,
\begin{align*}
\E_x\Big[e^{-r\tilde{\tau}_{[l,u]}+\theta X_{\tilde{\tau}_{[l,u]}}}1_{\{\tilde{\tau}_{[l,u]}<\infty\}}\Big]&= e^{l \theta}\E_{x-l}\Big[e^{-r\tilde{\tau}_{[0,u-l]}+\theta X_{\tilde{\tau}_{[0,u-l]}}} 1_{\{\tilde{\tau}_{[0,u-l]}<\infty\}}\Big] \\
&=\lambda e^{l\theta}L^{(r,\lambda)}(x-l,u-l; \theta)-e^{l\theta}M^{(r,\lambda)}(u-l; \theta)
\frac{ \mathscr{Z}^{(r, \lambda)}_{u-l} (x-l; \Phi(r)) }{N^{(r,\lambda)}(u-l)}.
\end{align*}
In addition, we have $L^{(r,\lambda)}(y,u-l; \theta) = 0$ for $y \leq 0$ by \eqref{fun_L_new_below_a}.
Combining this and the case with $\theta = 0$, we have the result when $r \geq 0$.
By analytical extension, the result holds for $r < 0$ satisfying Assumption \ref{assum_phi}.
\end{proof}

\subsubsection{First-order condition}


In view of \eqref{v_x<l} and our discussion in Section \ref{section_optimality_barrier} (that the optimal barriers are invariant of the starting value of $X$),
here we pursue the maximizer of the mapping $(l,u) \mapsto \tilde{v}_p(l,u; \theta)$. Hence, we compute
 the first-order conditions with respect to both $l$ and $u$.

We first consider the first-order condition with respect to $u$. To this end, we  obtain the following result whose proof is deferred to Appendix \ref{appen_M_derivative}.
\begin{lemma} \label{lemma_M_derivative}
For $a,\theta \in \mathbb{R}$, we have
$M^{(q,\lambda)\prime}(a; \theta)
 = -  \lambda e^{-(\Phi(q)-\theta) a}  (1+\lambda \overline{W}^{(q+\lambda)} (a; \theta) )$.
\end{lemma}

Applying Lemma  \ref{lemma_M_derivative} in \eqref{v<l}  gives
\begin{align*}
\begin{split}
\frac{\partial}{\partial u} \tilde{v}_p(l,u;\theta)
&=  \frac \partial {\partial u} \Big( e^{\theta l} M^{(r,\lambda)}(u-l; \theta) - K M^{(r,\lambda)}(u-l; 0) \Big) \frac{\displaystyle e^{-\Phi(r)l}}{N^{(r,\lambda)}(u-l)} \\
&- \tilde{v}_p(l,u;\theta) \frac { \lambda e^{-\Phi(r) (u-l) }Z^{(r+\lambda)}(u-l; \Phi(r))} {N^{(r,\lambda)}(u-l)} \\
&=- \lambda e^{\Phi(r)l}Z^{(r+\lambda)}(u-l;\Phi(r))j(l,u; \theta)
\frac{\displaystyle e^{-\Phi(r)u}}{N^{(r,\lambda)}(u-l)},
\end{split}
\end{align*}
where, for $l < u$ and $\theta \in \R$,
\begin{align}\label{aux_5_c_1}
j(l,u; \theta) :=\tilde{v}_p(l,u;\theta) - \frac{K\big[1+ \lambda \overline{W}^{(r+\lambda)} (u-l;0) \big]- e^{\theta u}\big[1+ \lambda \overline{W}^{(r+\lambda)} (u-l;\theta)\big]} {e^{\Phi(r)l}Z^{(r+\lambda)}(u-l;\Phi(r))}.
\end{align}
Hence,  the first-order condition $\frac{\partial}{\partial u} \tilde{v}_p(l,u;\theta)=0$ is equivalent to the condition
\begin{align} \label{cond_A}
\mathfrak{C}^{\theta}_1: j(l,u; \theta) = 0.
\end{align}

Regarding the first-order condition with respect to $l$, by differentiating \eqref{v<l} with respect to $l$, we obtain
\begin{align*}
\frac \partial {\partial l} \tilde{v}_p(l,u;\theta)
&=-\frac \partial {\partial u} \tilde{v}_p(l,u;\theta)+\frac{\left[ K M^{(r,\lambda)}(u-l; 0) \Phi(r) -  e^{\theta l} M^{(r,\lambda)}(u-l; \theta)  (\Phi(r)-\theta) \right]e^{-\Phi(r)l}}{N^{(r,\lambda)}(u-l) }.
\end{align*}
Therefore, under $\mathfrak{C}^{\theta}_1$, we have $\frac \partial {\partial l}\tilde{v}_p(l,u;\theta) = 0$ if and only if the following condition holds:
\begin{align}
	\mathfrak{C}^{\theta}_2: h(l,u;\theta)= 0, \label{cond_c_l_M}
\end{align}
where
\begin{align*}
	h(l,u;\theta) := \frac {e^{\Phi(r)(u-l)}} \lambda \Big[ K M^{(r,\lambda)}(u-l; 0) \Phi(r) -  e^{\theta l} M^{(r,\lambda)}(u-l; \theta)  (\Phi(r)-\theta) \Big], \quad l < u, \quad \theta \in \R.
\end{align*}

Under $\mathfrak{C}^{\theta}_1$ and $\mathfrak{C}^{\theta}_2$, the form of $v_p$ as in \eqref{v_p_theta} is simplified as follows.
\begin{proposition}\label{vf_p_sn}
Suppose $(\tilde{l},\tilde{u})$ satisfy $\mathfrak{C}^{\theta}_1$ and $\mathfrak{C}^{\theta}_2$.
 Then, we have
\begin{align*}
v_p(x;\tilde{l},\tilde{u},\theta)&= \lambda \Big(K  L^{(r,\lambda)}(x-\tilde{l},\tilde{u}-\tilde{l}; 0) -  e^{\theta \tilde{l}}L^{(r,\lambda)}(x-\tilde{l},\tilde{u}-\tilde{l}; \theta) \Big) + (K - e^{\theta \tilde{l}})  \mathscr{Z}^{(r, \lambda)}_{\tilde{u}-\tilde{l}} (x-\tilde{l}; \Phi(r)).
\end{align*}
In particular, for $x < \tilde{l}$, we have
\begin{align*}
v_p(x;\tilde{l},\tilde{u},\theta) = e^{-\Phi(r)(\tilde{l}-x)} (K - e^{\theta \tilde{l}}). 
\end{align*}
\end{proposition}
\begin{proof}
Under $\mathfrak{C}^{\theta}_2$, we have
\begin{align}\label{aux_3_c_1}
&K\big(1+ \lambda \overline{W}^{(r+\lambda)} (\tilde{u}-\tilde{l};0) \big)- e^{\theta \tilde{u}}\big(1+ \lambda \overline{W}^{(r+\lambda)} (\tilde{u}-\tilde{l};\theta)\big) = (K - e^{\theta \tilde{l}})  Z^{(r+\lambda)}(\tilde{u}-\tilde{l};\Phi(r)).
\end{align}
Hence, under both $\mathfrak{C}^{\theta}_1$ and $\mathfrak{C}^{\theta}_2$, by  \eqref{cond_A},
\begin{align}\label{v_p_C_u_C_l}
\tilde{v}_p(\tilde{l},\tilde{u};\theta)=\frac{(K - e^{\theta \tilde{l}})  Z^{(r+\lambda)}(\tilde{u}-\tilde{l};\Phi(r))} {e^{\Phi(r) \tilde{l}}Z^{(r+\lambda)}(\tilde{u}-\tilde{l};\Phi(r))} = (K - e^{\theta \tilde{l}}) e^{-\Phi(r) \tilde{l}}.
\end{align}
Substituting this in \eqref{v_p}, the proof is complete.
\end{proof}


\begin{lemma} \label{remark_simplified_cond}
Suppose $\Phi(r) \neq \theta$.
A pair of barriers $(\tilde{l},\tilde{u})$ satisfy $\mathfrak{C}_1^\theta$ and $\mathfrak{C}_2^\theta$ if and only if they satisfy   $\tilde{\mathfrak{C}}^{\theta}_1$ and $\mathfrak{C}_2^\theta$ where
\begin{align} \label{cond_c_l_simplified}
\tilde{\mathfrak{C}}^{\theta}_1:  \tilde{j}(l,u; \theta) = 0,
\end{align}
with
\begin{align}\label{aux_1_c_1}
\tilde{j}(l,u;\theta):=  \frac {\theta K M^{(r,\lambda)}(u-l; 0)} { \Phi(r)-\theta} - (K - e^{\theta l}) N^{(r,\lambda)}(u-l), \quad l < u, \quad \theta \in \R.
\end{align}
\end{lemma}
\begin{proof}
(i)	Suppose $(\tilde{l},\tilde{u})$ satisfy $\mathfrak{C}^{\theta}_1$ and $\mathfrak{C}^{\theta}_2$.
By combining 
\eqref{v<l} and  \eqref{v_p_C_u_C_l}, we must have
\begin{align*}
  \frac{[e^{\theta \tilde{l}} M^{(r,\lambda)}(\tilde{u}-\tilde{l}; \theta) - K M^{(r,\lambda)}(\tilde{u}-\tilde{l}; 0)] e^{-\Phi(r) \tilde{l}}}{N^{(r,\lambda)}(\tilde{u}-\tilde{l})} = (K - e^{\theta \tilde{l}}) e^{-\Phi(r) \tilde{l}}.
\end{align*}
On the other hand, condition $\mathfrak{C}^{\theta}_2$ is equivalent to
\begin{align}\label{aux_2_c_1}
K  M^{(r,\lambda)}(\tilde{u}-\tilde{l}; 0) \frac {\Phi(r)} { \Phi(r)-\theta} =  e^{\theta \tilde{l}} M^{(r,\lambda)}(\tilde{u}-\tilde{l}; \theta) .
\end{align}
Combining these,  $(\tilde{l},\tilde{u})$ satisfy $\tilde{\mathfrak{C}}^{\theta}_1$.

(ii) 
Conversely, assume that $(\hat{l},\hat{u})$ satisfy $\tilde{\mathfrak{C}}^{\theta}_1$ and $\mathfrak{C}^{\theta}_2$. Then combining \eqref{aux_1_c_1} and \eqref{aux_2_c_1}, we obtain
	\begin{align*}
		e^{\theta \hat{l}} M^{(r,\lambda)}(\hat{u}-\hat{l};\theta) - K M^{(r,\lambda)}(\hat{u}-\hat{l}; 0)=\frac{\theta K M^{(r,\lambda)}(\hat{u}-\hat{l}; 0)}{\Phi(r)-\theta} =(K - e^{\theta \hat{l}}) N^{(r,\lambda)}(\hat{u}-\hat{l}).
	\end{align*}
	Hence using the above identity in \eqref{v<l} we obtain
	\begin{align}\label{aux_4_c_1}
		\tilde{v}_p(\hat{l},\hat{u};\theta)=(K - e^{\theta \hat{l}})e^{-\Phi(r) \hat{l}}.
	\end{align}
	Therefore using \eqref{aux_3_c_1} together with \eqref{aux_4_c_1} in \eqref{aux_5_c_1} we conclude that $j(\hat{l},\hat{u}; \theta) = 0$ and hence $(\hat{l},\hat{u})$ satisfy $\mathfrak{C}^{\theta}_1$. 
\end{proof}



\begin{remark} \label{remark_h}\rm
We have, for $u > l$,
\begin{align*}
\begin{split}
h(l,u; \theta)
&=  K \Big[1+\lambda \overline{W}^{(r+\lambda)} (u-l;0)-Z^{(r+\lambda)}(u-l;\Phi(r)) \Big] \\ &-   \Big[e^{\theta u}+\lambda e^{\theta u}\overline{W}^{(r+\lambda)} (u-l;\theta)- e^{\theta l} Z^{(r+\lambda)}(u-l;\Phi(r)) \Big] \\
&= (K - e^{\theta u}) - (K - e^{\theta l}) Z^{(r+\lambda)}(u-l;\Phi(r))  + \lambda \int_0^{u-l} (K- e^{\theta (u-y)})W^{(r+\lambda)}(y) \diff y.
\end{split}
\end{align*}
Differentiating this,
\begin{align}\label{h_derivative<u}
\frac \partial {\partial l} h(l,u; \theta)
&= \theta e^{\theta l} Z^{(r+\lambda)}(u-l;\Phi(r)) + (K - e^{\theta l}) \Big( \Phi(r) Z^{(r+\lambda)}(u-l;\Phi(r))  + \lambda W^{(r+\lambda)}(u-l) \Big)\notag\\ &- \lambda  (K- e^{\theta l})W^{(r+\lambda)}(u-l) \notag\\
&= \big[\theta e^{\theta l} + (K - e^{\theta l})  \Phi(r) \big] Z^{(r+\lambda)}(u-l;\Phi(r)).
\end{align}
In particular,
\begin{align}
\frac \partial {\partial l} h(l,u; \theta) \Big|_{l = u-}
&= \theta e^{\theta u} + (K - e^{\theta u})  \Phi(r), \label{h_derivative_at_u} \\
h(u,u; \theta) &= 0. \label{h_u_u}
\end{align}
\end{remark}

\subsection{Spectrally negative \lev case}  \label{subsection_SN}

The spectrally negative \lev case corresponds to the maximization of 
\begin{align*}
	(l,u) \mapsto v_p^{SN}(x;l,u) := v_p(x;l,u,1).
\end{align*}
Recall as in Theorem  \ref{theorem_barrier_optimal}  that the optimal stopping region is given by $\mathcal{D}_p:= [L_p^*, U_p^*] = [e^{l_p^*}, e^{u_p^*}]$ satisfying \eqref{L_range_put}, invariant of the starting value $x$. The optimal barriers $(l_p^*, u_p^*)$ must be such that $\tilde{v}_p(l,u; 1)$ as in \eqref{v<l}  is maximized, and hence we can use directly the results in the previous subsection by setting $\theta = 1$.


\begin{figure}[htbp]
	\begin{center}
		\begin{minipage}{1.0\textwidth}
			\centering
			\begin{tabular}{c}
				\includegraphics[scale=0.4]{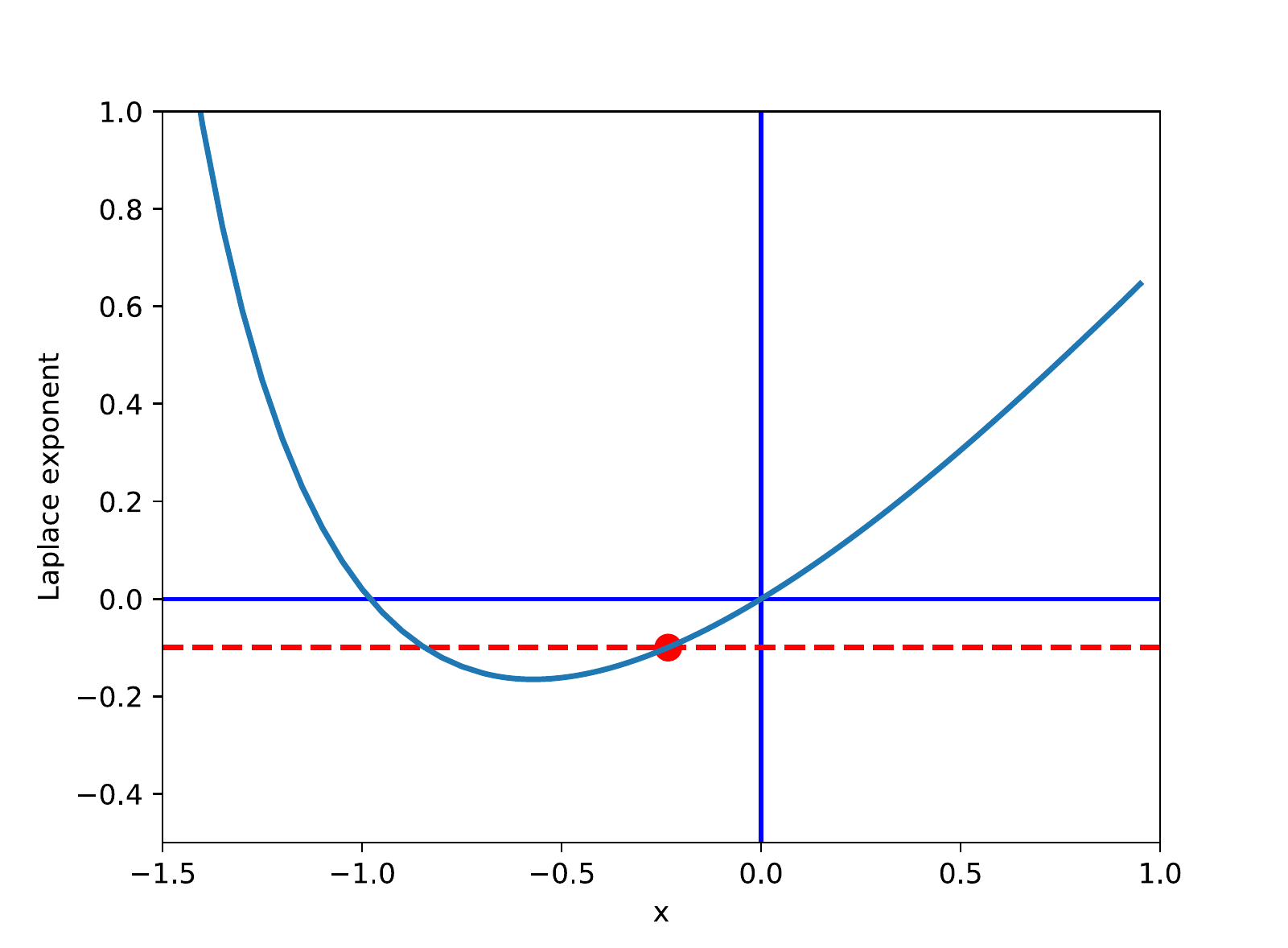} 
			\end{tabular}
		\end{minipage}
		\caption{\footnotesize The plot of $\psi$ and $\Phi(r)$ when $r < 0$ and $\psi'(0) > 0$. 
		In order for $\Phi(r)$ to be well-defined, $y = \psi(x)$ (solid curve) and $y = r$ (dashed line) need to cross. 
		} \label{figure_psi_SN}
	\end{center}
\end{figure}

 Additionally to Assumption \ref{assum_phi} we make the following assumption on the L\'evy process $X$.
\begin{assump} \label{assump_X_SN}
We assume that $X$ drifts to infinity (i.e.\ $\E X_1 = \psi'(0) > 0$) throughout this subsection.
\end{assump}

\begin{remark} \label{remark_Phi_sign}  
Under Assumption \ref{assum_phi},  Assumption \ref{assump_X_SN} holds if and only if
	\begin{equation}\label{phineg}
	\Phi(r) < 0
	\end{equation}
 for $r<0$ (because $\psi'(0) > 0$ and $\psi$ is convex on $[0,\infty)$).  In Figure \ref{figure_psi_SN}, we plot the Laplace exponent $\psi$ along with $\Phi(r)$ for the case $\psi'(0) > 0$ and $\Phi(r)$ is well-defined. 
 \end{remark}

 Assumptions \ref{assum_phi}  and \ref{assump_X_SN} guarantee the following.
 \begin{lemma} \label{lemma_tail_SN}
 Assumption \ref{assump_tail_value_function} for the put case ($i = p$) is satisfied.
 \end{lemma}
 \begin{proof}
 It suffices to show that the condition in Lemma \ref{cond_last_time} is satisfied. Indeed, it is equivalent to the condition that $\E_x [e^{-{r}\tau_{\text{last}}(\log K)}
]<\infty$ where
$\tau_{\text{last}}(\log K):=\sup\{t\geq0:X_t\leq \log K  \}$. This is indeed satisfied by the proof of Lemma 1 in \cite{DPT} together with Theorem 2 in \cite{Eriklast}, because $\psi'(0) > 0$ and $\Phi(r)$ is well-defined. 

	\end{proof}

\subsubsection{First-order conditions} \label{subsub_first_SN}
As $\tilde{v}_p(l,u; 1)$ has been shown to be  smooth in $l$ and $u$ in Section \ref{subsection_preliminary}, the optimal barriers $(l_p^*, u_p^*)$ as in \eqref{opt_barrier_log} must satisfy the first-order conditions with respect to both $l$ and $u$.  Using \eqref{cond_c_l_M}, we have that $\frac \partial {\partial u}\tilde{v}_p(l,u;1) = 0$ if and only if 
\begin{align*} 
	\mathfrak{C}^{SN}_u: j(l,u; 1) = 0,
\end{align*}
and, under $\mathfrak{C}_u^{SN}$, the condition $\frac \partial {\partial l}\tilde{v}_p(l,u;1) = 0$ is equivalent to
\begin{align*}
	\mathfrak{C}^{SN}_l: h(l,u; 1)= 0. 
\end{align*}
By Lemma \ref{remark_simplified_cond}, a pair of barriers $(l,u)$ satisfy $\mathfrak{C}_l^{SN}$ and $\mathfrak{C}_u^{SN}$ if and only if they satisfy  $\mathfrak{C}_l^{SN}$ and $\tilde{\mathfrak{C}}_u^{SN }$ where \begin{align*}
\tilde{\mathfrak{C}}_u^{SN }:   \tilde{j}(l,u;1) = 0.
\end{align*}

\begin{remark}
Given \eqref{phineg},  note that the existence of a pair satisfying these is guaranteed, because at least the optimal barriers $(l_p^*, u_p^*)$ satisfy them. Unfortunately, however, the uniqueness of the solution to these may not hold.
%
\end{remark}


The following is a direct corollary of Proposition \ref{vf_p_sn}.
\begin{theorem}\label{cor_vf_SN}
Suppose $(\tilde{l}_p^{SN}, \tilde{u}_p^{SN})$ be such that $\mathfrak{C}_l^{SN}$ and $\mathfrak{C}_u^{SN}$ (or equivalently $\mathfrak{C}_l^{SN}$ and $\tilde{\mathfrak{C}}_u^{SN }$) are satisfied. Then, for $x \in \R$,
\begin{align*}
v_p^{SN}(x;\tilde{l}_p^{SN},\tilde{u}_p^{SN})&= \lambda \Big(K  L^{(r,\lambda)}(x-\tilde{l}_p^{SN},\tilde{u}_p^{SN}-\tilde{l}_p^{SN}; 0) -  e^{\tilde{l}_p^{SN}}L^{(r,\lambda)}(x-\tilde{l}_p^{SN},\tilde{u}_p^{SN}-\tilde{l}_p^{SN}; 1) \Big) \\&+ (K - e^{ \tilde{l}_p^{SN}})  \mathscr{Z}^{(r, \lambda)}_{\tilde{u}_p^{SN}-\tilde{l}_p^{SN}} (x-\tilde{l}_p^{SN}; \Phi(r)).
\end{align*}
In particular, for $x < \tilde{l}_p^{SN}$, we have
$v_p^{SN}(x;\tilde{l}_p^{SN},\tilde{u}_p^{SN}) = e^{-\Phi(r)(\tilde{l}_p^{SN}-x)} (K - e^{\tilde{l}_p^{SN}})$.
\end{theorem}

\subsubsection{Computation of 
$(\tilde{l}_p^{SN}, \tilde{u}_p^{SN})$
} 

 \begin{figure}[htbp]
\begin{center}
\begin{minipage}{1.0\textwidth}
\centering
\begin{tabular}{cc}
 \includegraphics[scale=0.4]{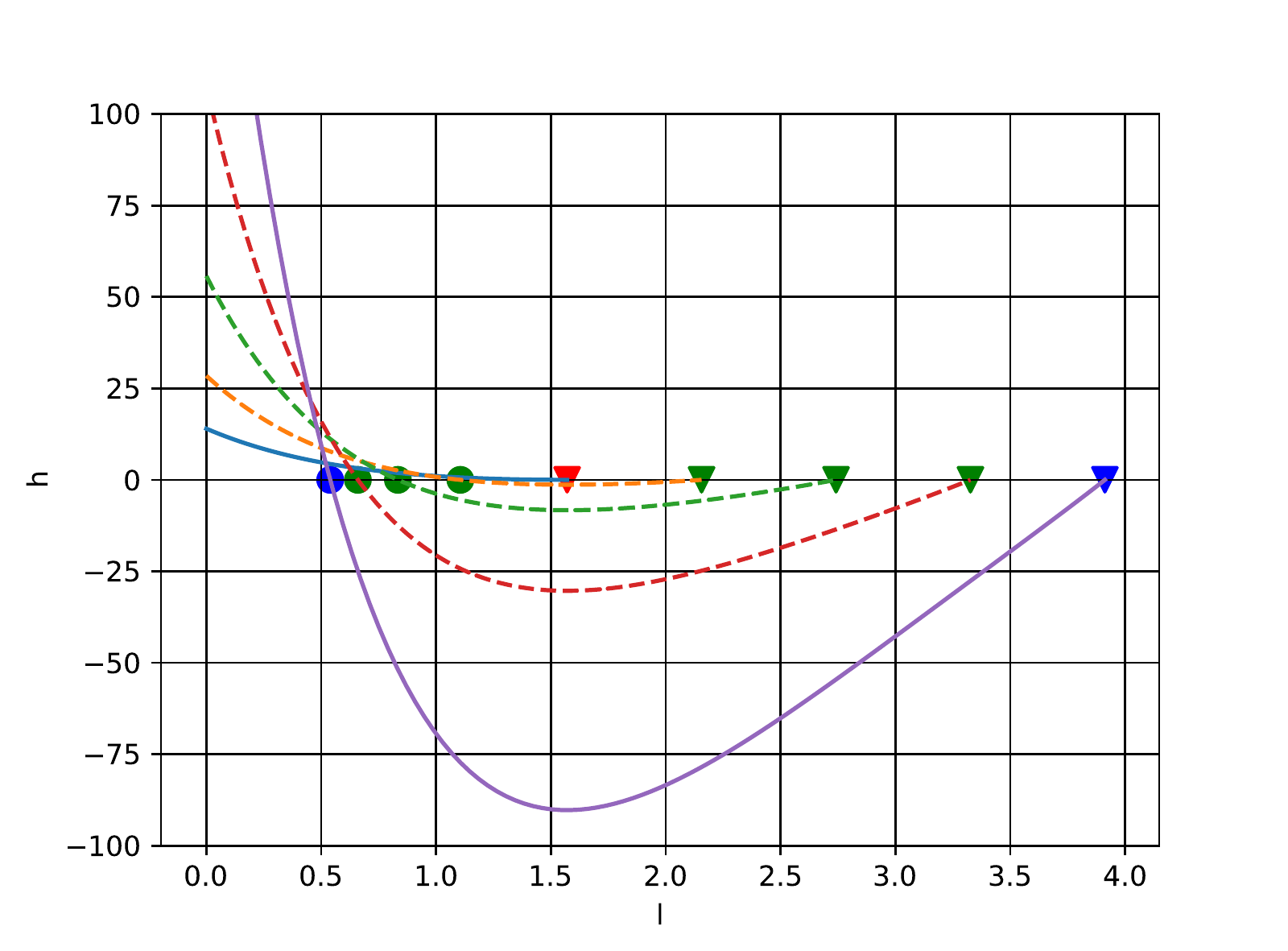} & \includegraphics[scale=0.4]{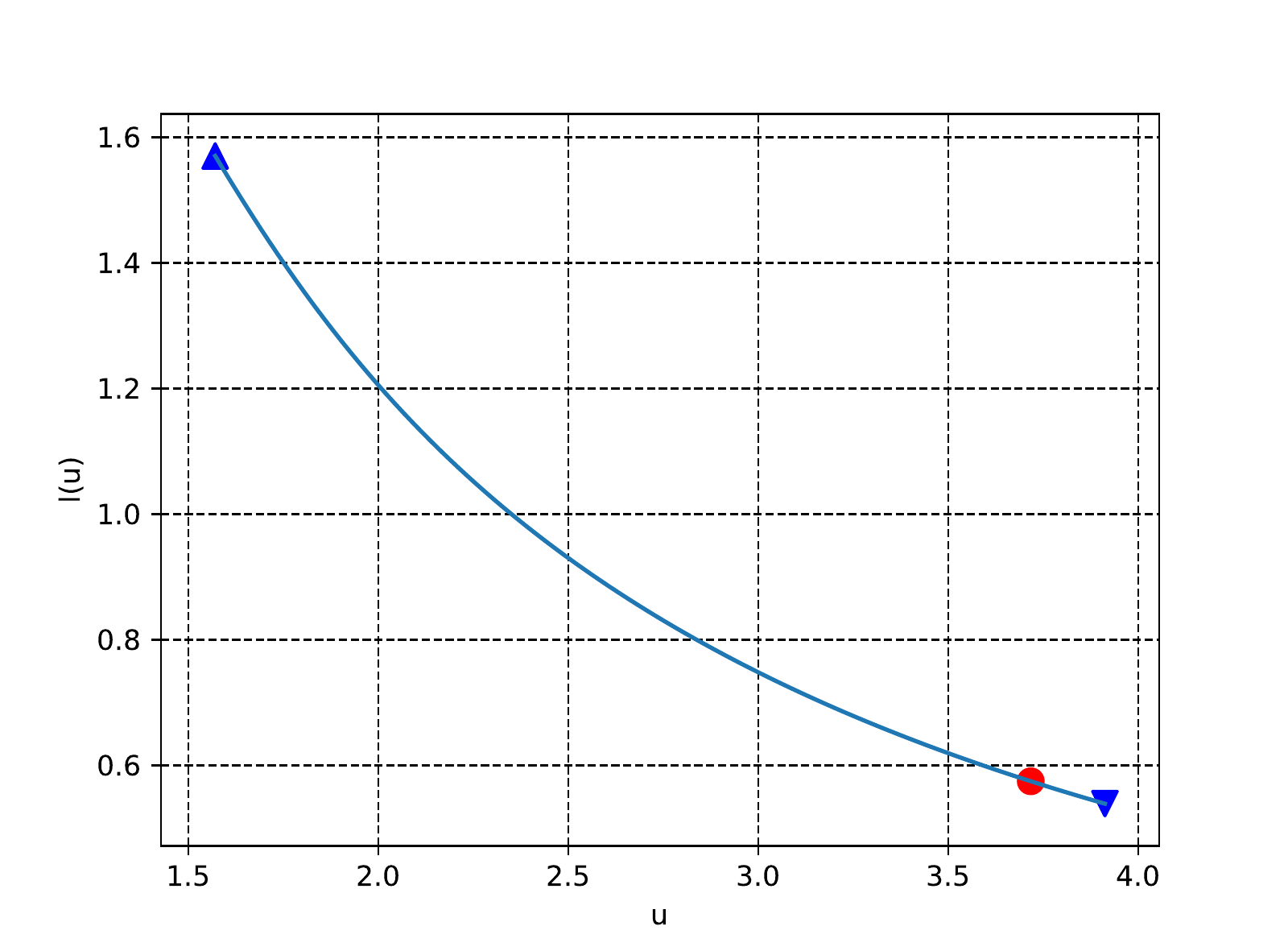} \\
 (1) $l \mapsto h(l,u;1)$ &  (2) $u \mapsto l(u)$ \\
  \multicolumn{2}{c}{ \includegraphics[scale=0.4]{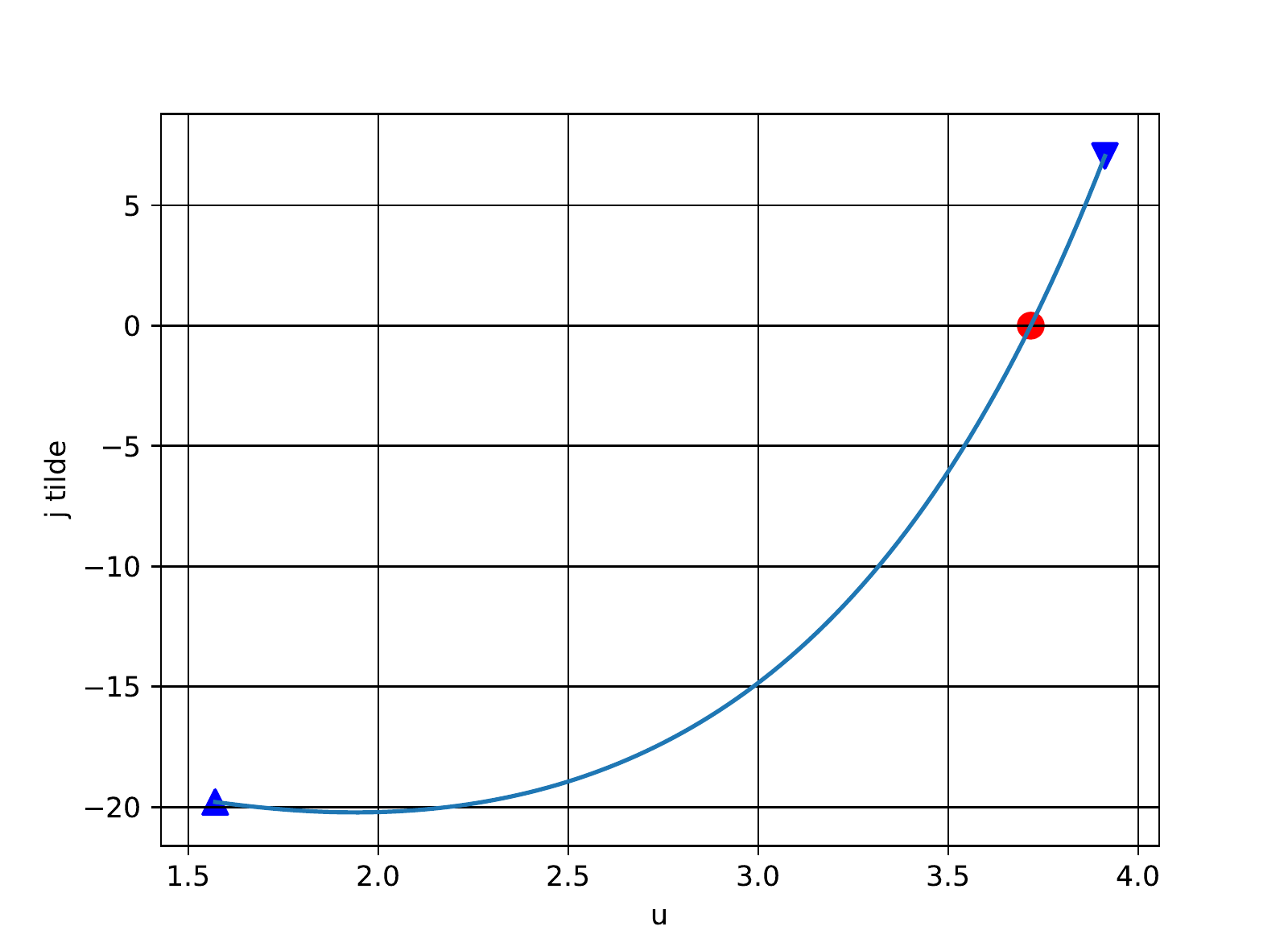} }  \\
  \multicolumn{2}{c}{ (3) $u \mapsto \tilde{j}(l(u), u; 1)$ }
 \end{tabular}
\end{minipage}
\caption{\footnotesize 
	(1) The function  $l \mapsto h(l,u; 1)$  for $u = \underline{u}(=1.570)$, $2.156$, $2.741$, $3.326$, $\log K (= 3.912)$ (solid lines for $u = \underline{u}, \log K$ and dashed lines for others). The points at $l=u$ are indicated by down-pointing triangles. The roots $l(u)$ are indicated by circles.  (2) The mapping $u \mapsto l(u)$ for $u \in [\underline{u}, \log K]$. (3) The function $u \mapsto \tilde{j}(l(u), u; 1)$ for $u \in [\underline{u}, \log K]$.  For (2) and (3), the points at $u=\underline{u}+$ are indicated by up-pointing triangles whereas those at $u=\log K$ are indicated by down-pointing triangles. The points at $u=\tilde{u}^{SN}_p$ are indicated by circles.
} \label{figure_h_SN}
\end{center}
\end{figure}

Recall \eqref{phineg}. By \eqref{h_derivative_at_u}, we have
\begin{align}
\frac \partial {\partial l} h(l,u;1) \Big|_{l = u-}
 \leq 0 \;  \textrm{ if and only if } \; u  \leq  \log ( -K \Phi(r)) - \log (1- \Phi(r)) =:\underline{u}. \label{u_underline_cond}
 \end{align}

\begin{lemma} \label{lemma_l_u}
For $u \in (\underline{u}, \log K)$, there exists a unique root $l(u) < u$ such that $h(l(u), u; 1) = 0$ (which satisfies $\mathfrak{C}^{SN}_l$). For $u \in (-\infty, \underline{u}]$, there does not exist $l < u$ such that  $h(l, u; 1) = 0$.
\end{lemma}
\begin{proof}
(i) Consider $u \leq \underline{u}$ such that \eqref{u_underline_cond} holds.
Because $l \mapsto e^l + (K - e^l)  \Phi(r)$ is increasing, by \eqref{h_derivative<u} and \eqref{u_underline_cond}, $\frac \partial {\partial l} h(l,u;1)$ is uniformly negative for $l < u$ and hence \eqref{h_u_u} gives that $h(l,u;1)$ is uniformly positive. Hence, there does not exist $l < u$ such that $h(l,u;1) = 0$.
\\
(ii) Consider $\underline{u} < u < \log K$ such that \eqref{u_underline_cond} does not hold.
Then,  again by \eqref{h_derivative<u} and \eqref{u_underline_cond}, there exists $\bar{l}(u) < u$ such that this $\frac \partial {\partial l} h(l,u;1)$ is negative on $(-\infty, \bar{l}(u))$ and positive on $(\bar{l}(u), u)$.
By noting that $e^l + (K - e^l)  \Phi(r) \xrightarrow{l \downarrow -\infty} K \Phi(r) < 0$, we have  $\frac \partial {\partial l} h(l,u;1) \xrightarrow{l \downarrow -\infty} - \infty$.
These and \eqref{h_u_u} imply that there exists a unique $-\infty < l(u) \leq \bar{l}(u) < u$ such that  \eqref{cond_c_l_M} holds for $l = l(u)$.
\end{proof}

By these observations, in order to compute $(\tilde{l}_p^{SN}, \tilde{u}_p^{SN})$ satisfying both  \eqref{cond_c_l_M} and  \eqref{cond_c_l_simplified},  we can focus on $\underline{u} < u < \log K$ and choose $(l(u), u)$ that satisfies \eqref{cond_c_l_simplified}. Notice that we already know that the optimal barriers $(l^*_p, u^*_p)$ exist and they satisfy the first-order conditions $\mathfrak{C}^{SN}_l$ and $\tilde{\mathfrak{C}}_u^{SN }$. Hence, by Lemma \ref{lemma_l_u}, $u^*_p$ must lie in $(\underline{u}, \log K)$. Now, $u^*_p$ must be one of the solutions to $\tilde{\mathfrak{C}}_u^{SN }: \tilde{j}(l(u), u; 1)=0$, whose existence is guaranteed because we already know that $u^*_p$ satisfies it. If $\tilde{j}(l(u), u; 1) = 0$ has a unique solution, the solution must be the optimal barrier $u^*_p$ and $l^*_p = l(u^*_p)$.

To illustrate this, in Figure \ref{figure_h_SN}, we plot the functions (1)  $l \mapsto h(l,u;1)$, (2) $u \mapsto l(u)$, and (3) $u \mapsto \tilde{j}(l(u), u; 1)$ in the example provided in Section \ref{section_numerics}. 
As discussed above, there exists a unique $l(u)$ such that $h(l(u), u; 1) = 0$ for each $u \in (\underline{u}, \log K)$. In this example, $u \mapsto l(u)$ appears to be monotone.  The root of $\tilde{j}(l(u), u; 1) = 0$ becomes $\tilde{u}_p^{SN}$ and $\tilde{l}_p^{SN} = l(\tilde{u}_p^{SN})$. Here, as there is only one solution to $\tilde{j}(l(u), u; 1) = 0$, the pair $(\tilde{l}_p^{SN}, \tilde{u}_p^{SN})$ satisfying the first-order conditions is unique and therefore it must be the optimal upper and lower boundaries $(\tilde{l}_p^*, \tilde{u}_p^*)$.

\subsection{Spectrally positive case} \label{subsection_SP}

We now consider the case $X$ is a spectrally positive \lev process. 
Here we assume that the dual  (spectrally negative) \lev process $X^d = -X$ 
has its Laplace exponent $\psi$, the $q$-scale function $W^{(q)}$  and right inverse $\Phi(q)$ for each $q$.
As in the spectrally negative case, we make the following assumption. 
\begin{assump}\label{assum_phi_SP} 
We assume that $X$ drifts to infinity  (i.e.\ $-\E X_1 = \E X_1^d  = \psi'(0) < 0$) throughout this subsection.
\end{assump}
\begin{remark} \label{remark_Phi_sign_SP} 
For the spectrally positive case under Assumption \ref{assum_phi},  Assumption \ref{assum_phi_SP} holds if and only if
\begin{equation}
\Phi(r) > 0 \label{Phi_cond_SP}
\end{equation}
 for $r<0$ (because $\psi'(0) < 0$ and $\psi$ is convex on $[0, \infty)$).   In Figure \ref{figure_psi_SP}, we plot the Laplace exponent $\psi$ along with $\Phi(r)$ for the case $\psi'(0) < 0$ and $\Phi(r)$ is well-defined.
\end{remark}

\begin{figure}[htbp]
	\begin{center}
		\begin{minipage}{1.0\textwidth}
			\centering
			\begin{tabular}{c}
\includegraphics[scale=0.4]{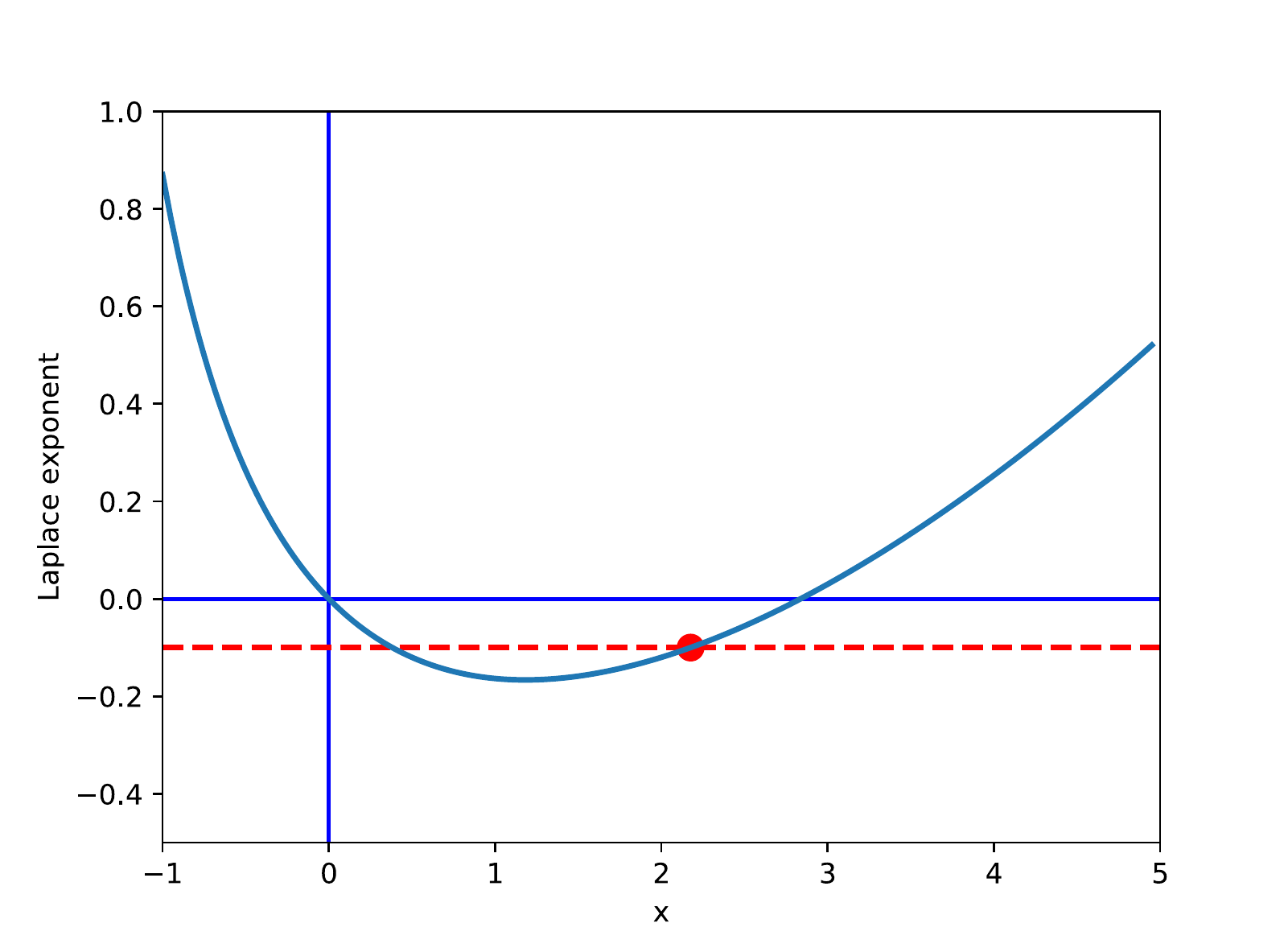} 
			\end{tabular}
		\end{minipage}
		\caption{\footnotesize The plot of $\psi$ and $\Phi(r)$ when $r < 0$ and $\psi'(0) < 0$. 
		In order for $\Phi(r)$ to be well-defined, $y = \psi(x)$ (solid curve) and $y = r$ (dashed line) need to cross. 
		} \label{figure_psi_SP}
	\end{center}
\end{figure}

%

 Analogously to the spectrally negative case, Assumptions \ref{assum_phi}  and \ref{assum_phi_SP}  guarantee the following.
 \begin{lemma} \label{lemma_tail_SP} Assumption \ref{assump_tail_value_function} for the put case ($i = p$) is satisfied.
 \end{lemma}

The next proposition holds by Proposition \ref{proposition_v_p_SN}.
\begin{proposition}
	For  $l<u$ and $x\in\R$, 
	\begin{align*}
			v_p^{SP}(x;l,u) &:=\E_x\Big[e^{-r\tilde{\tau}_{[l,u]}}(K-e^{X_{\tilde{\tau}_{[l,u]}}})1_{\{\tilde{\tau}_{[l,u]}<\infty\}}\Big]= v_p(-x;-u,-l,-1).
	\end{align*}
	In particular, for $x>u$, we have that
	\begin{align*}
		\begin{split}
			v_p^{SP}(x;l,u)
			= e^{-\Phi(r)x} \tilde{v}_p(-u,-l;-1).
		\end{split}
	\end{align*}

\end{proposition}


\subsubsection{First-order condition} \label{subsubsection_first_order_SP}
Similarly to the spectrally negative case, the optimal barriers $(l_p^*, u_p^*)$ as in \eqref{opt_barrier_log}  must maximize $(l,u) \mapsto \tilde{v}_p(-u,-l;-1)$. 

Proceeding as in \eqref{cond_A}, the first-order condition with respect to $l$ (i.e.\ $\frac{\partial}{\partial l}\tilde{v}_p(-u,-l;-1)=0$) is equivalent to 
\begin{align*} 
\mathfrak{C}_l^{SP}:  j(-u, -l; -1) = 0.
\end{align*}
Under $\mathfrak{C}_l^{SP}$, 
by similar arguments to those used in \eqref{cond_c_l_M}, we obtain that $\frac{\partial}{\partial u}\tilde{v}_p(-u,-l;-1)=0$ is equivalent to
\begin{align}
\mathfrak{C}_u^{SP}: h(-u,-l;-1) = 0.
 \label{cond_c_l_M_SP}
\end{align}


Proceeding as in Lemma \ref{remark_simplified_cond}, a pair of barriers $(l,u)$ satisfy $\mathfrak{C}_l^{SP}$ and $\mathfrak{C}_u^{SP}$ if and only if they satisfy  $\tilde{\mathfrak{C}}_l^{SP }$ and $\mathfrak{C}_u^{SP}$ where
\begin{align} \label{cond_c_l_simplified_SP}
\tilde{\mathfrak{C}}_l^{SP }: \tilde{j}(-u,-l;-1) = 0.
\end{align}

As a corollary of Proposition \ref{vf_p_sn}, we have the following.
\begin{theorem}\label{vf_p_sp}
	Suppose $(\tilde{l}_p^{SP}, \tilde{u}_p^{SP})$ be such that $\mathfrak{C}_l^{SP}$ and $\mathfrak{C}_u^{SP}$ (equivalently $\tilde{\mathfrak{C}}_l^{SP}$ and $\mathfrak{C}_u^{SP}$) are satisfied. Then, 
	\begin{align*}
		v_p^{SP}(x;\tilde{l}_p^{SP},\tilde{u}_p^{SP})=
		&\lambda \Big(K  L^{(r,\lambda)}(\tilde{u}_p^{SP}-x,\tilde{u}_p^{SP}-\tilde{l}_p^{SP}; 0) -  e^{\tilde{u}_p^{SP}} L^{(r,\lambda)}(\tilde{u}_p^{SP}-x,\tilde{u}_p^{SP}-\tilde{l}_p^{SP}; -1) \Big)\\&+ (K - e^{\tilde{u}_p^{SP}})  \mathscr{Z}^{(r, \lambda)}_{\tilde{u}_p^{SP}-\tilde{l}_p^{SP}} (\tilde{u}_p^{SP}-x; \Phi(r)).
	\end{align*}
	In particular, for $x > \tilde{u}_p^{SP}$, we have
	\begin{align*}
		v_p^{SP}(x;\tilde{l}_p^{SP},\tilde{u}_p^{SP})
		= e^{\Phi(r)(\tilde{u}_p^{SP}-x)} (K - e^{\tilde{u}_p^{SP}}). 
	\end{align*}
\end{theorem}



\subsubsection{Computation of $(\tilde{l}_p^{SP}, \tilde{u}_p^{SP})$.}\label{comp_sp}

 \begin{figure}[htbp]
\begin{center}
\begin{minipage}{1.0\textwidth}
\centering
\begin{tabular}{ccc}
 \includegraphics[scale=0.4]{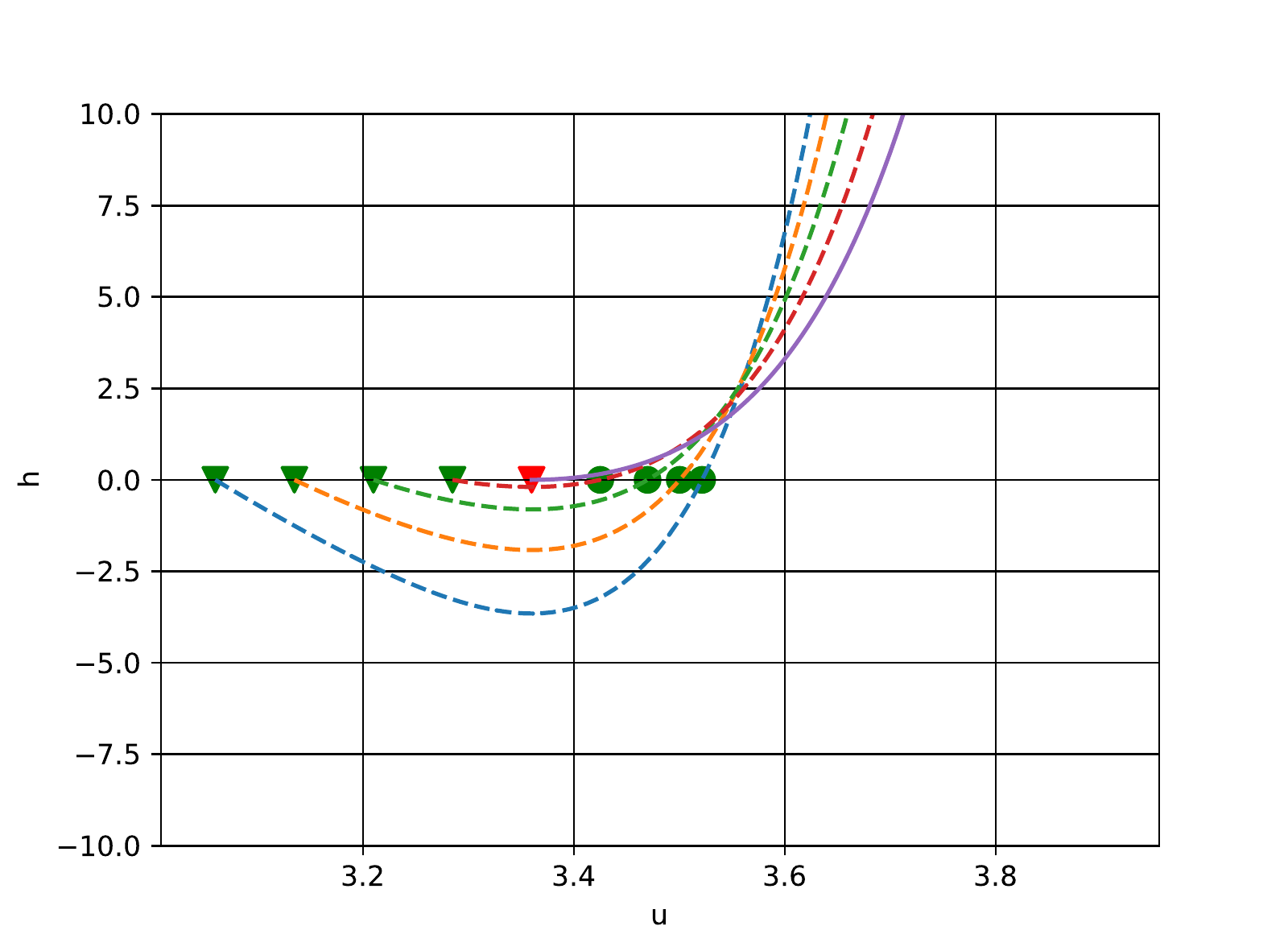} & \includegraphics[scale=0.4]{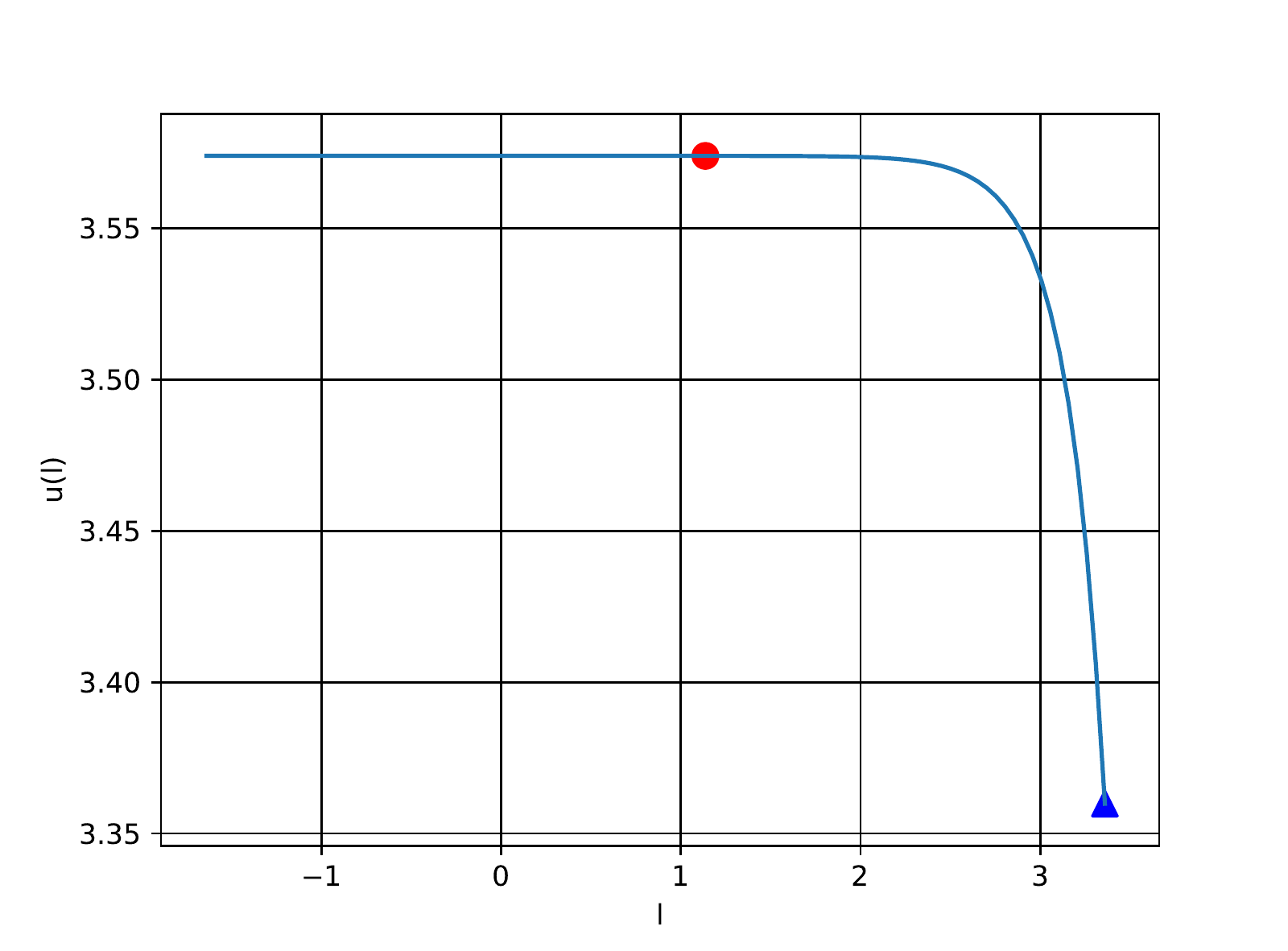} \\
 (1) $u \mapsto h(-u,-l; -1)$ &  (2) $l \mapsto u(l)$ \\
 \multicolumn{2}{c}{ \includegraphics[scale=0.4]{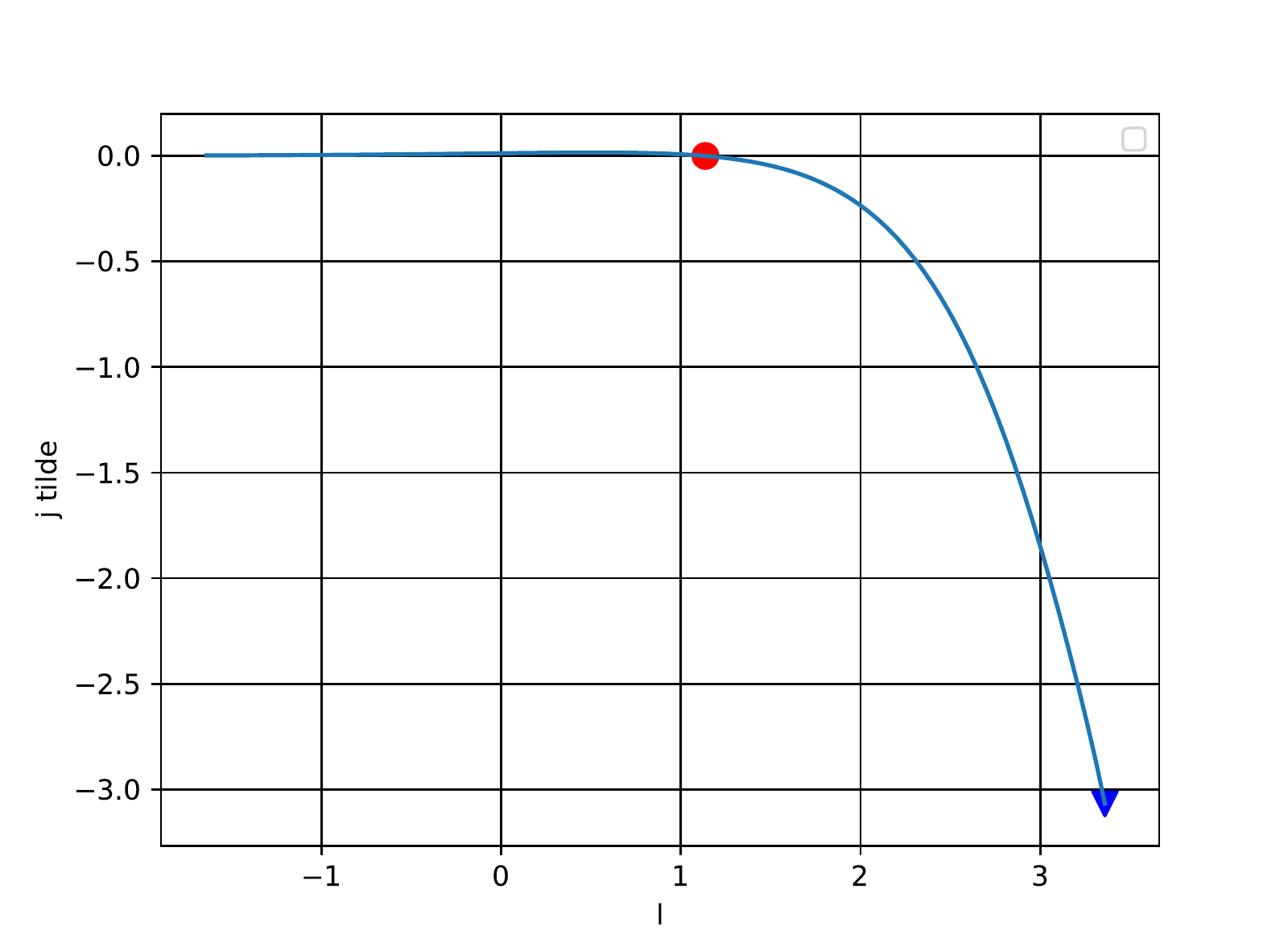} }  \\
  \multicolumn{2}{c}{ (3) $l \mapsto \tilde{j}(-u(l), -l; -1)$ }
 \end{tabular}
\end{minipage}
\caption{\footnotesize 
	(1) The function  $u \mapsto h(-u, -l;-1)$  for $l = 3.059$,
$3.134$, $3.209$,
$3.284$,
$\overline{l} (=3.359)$ (solid  line for $l = \overline{l}$ and dashed lines for others). The points at $u=l$ are indicated by triangles. The roots $u(l)$ are indicated by circles.  (2) The mapping $l \mapsto u(l)$ for $l \leq \overline{l}$. (3) The function $l \mapsto \tilde{j}(-u(l), -l; -1)$ for $l \leq \overline{l}$.  For (2) and (3), the points at $l=\overline{l}$ are indicated by triangles. The points at $l = \tilde{l}_p^{SP}$ are indicated by circles.
} \label{figure_h_SP}
\end{center}
\end{figure}

By Remark \ref{remark_h}, we have
	\begin{align} \label{h_tilde_wrt_u}
\begin{split}
		\frac \partial {\partial u} h(-u,-l;-1)
		&= (e^u - (K - e^u)  \Phi(r)) Z^{(r+\lambda)}(u-l;\Phi(r)).
\end{split}
	\end{align}
	In particular, $\frac \partial {\partial u} h(-u,-l;-1) \Big|_{u = l+}= e^l - (K - e^l)  \Phi(r)$ and
	$h(-l,-l;-1) = 0$.
	
	For $l < \log K$, we have
	\begin{align}
		h(-\log K,-l;-1)
		&= (K-e^l )  + \lambda \int_0^{\log K -l} (K-e^{l+y})W^{(r+\lambda)}(y) \diff y > 0. \label{aux}
	\end{align}

Define
\[
\overline{l} := \log \frac {K  \Phi(r)} { 1 + \Phi(r) }.
\]

\begin{lemma} \label{lemma_l_u_SP}
For $l \in (-\infty, \overline{l})$, there exists a unique root $u(l) \in (l, \log K)$ such that $h(-u(l), -l; -1) = 0$. For $l \in [\underline{l}, \log K)$, there does not exist $l < u < \log K$ such that  $h(-u, -l; -1) = 0$.
\end{lemma}
\begin{proof}
(i) Suppose
\begin{align*}
	\frac \partial {\partial u} h(-u,-l;-1) \Big|_{u = l+}
	&= e^l - (K - e^l)  \Phi(r) < 0 \Longleftrightarrow l  < \overline{l}. \end{align*}
Then because 
 $u \mapsto  e^u - (K - e^u)  \Phi(r)$ is increasing and $Z^{(r+\lambda)}(u-l;\Phi(r))$ is uniformly positive, in view of  \eqref{aux}, there exists, for each $l < \log  K$, $\bar{u}(l)$ such that  $\frac \partial {\partial u} h(-u,-l;-1)$ is negative on $(l, \bar{u}(l))$ and positive on $(\bar{u}(l), \infty)$. Because $h(-l,-l;-1) = 0$ and $h(-\log K,-l;-1) > 0$ by \eqref{aux}, there exists a unique $\bar{u}(l)\leq u(l)<\log K$ such that  \eqref{cond_c_l_M_SP} holds. 

(ii) Suppose 
\begin{align*}
\frac \partial {\partial u} h(-u,-l;-1) \Big|_{u = l+}
&= e^l - (K - e^l)  \Phi(r) \geq 0  \Longleftrightarrow l  \geq \overline{l}. \end{align*}
Then because $u \rightarrow e^u - (K - e^u)  \Phi(r)$ is increasing and $Z^{(r+\lambda)}(u-l;\Phi(r))$ is uniformly positive, $\frac \partial {\partial u} h(-u,-l;-1)$ is uniformly positive by \eqref{h_tilde_wrt_u} and hence  $h(-u,-l;-1)$ is uniformly positive. Hence, there does not exist $l < u$ such that $h(-u,-l;-1) = 0$.
\end{proof}


In view of these observations,  similarly to the spectrally negative case, in order to compute $(\tilde{l}_p^{SP}, \tilde{u}_p^{SP})$ satisfying  \eqref{cond_c_l_M_SP} and  \eqref{cond_c_l_simplified_SP},  we first focus on $l < \overline{l}$ and choose $(l, u(l))$  that satisfies \eqref{cond_c_l_simplified_SP}. Notice that we already know that the optimal barriers $(l^*_p, u^*_p)$ exist and they satisfy the first-order conditions $\mathfrak{C}^{SP}_l$ and $\tilde{\mathfrak{C}}_u^{SP}$. Hence, by Lemma \ref{lemma_l_u_SP}, $l^*_p$ must lie in $(-\infty, \bar{l})$. Now, $l^*_p$ must be one of the solutions to $\tilde{\mathfrak{C}}_l^{SP}: \tilde{j}(-u(l), -l; -1)=0$, whose existence is guaranteed because we already know that $l^*_p$ satisfies it. If $ \tilde{j}(-u(l), -l; -1) = 0$ has a unique solution, the solution must be the optimal barrier $l^*_p$ and $u^*_p = u(l^*_p)$.


In Figure \ref{figure_h_SP}, we plot the functions (1)  $u \mapsto h(-u,-l;-1)$ for various values of $l$ less than or equal to $\overline{l}$, (2) $l \mapsto u(l)$, and (3) $l \mapsto \tilde{j}(-u(l), -l; -1)$ in the example provided in Section \ref{section_numerics}. As discussed above, there exists a unique $u(l)$ such that $h(-u(l), -l; -1)=0$ for each $l \in (-\infty, \overline{l})$. The root of $\tilde{j}(-u(l), -l; -1) = 0$ becomes $\tilde{l}_p^{SP}$ and $\tilde{u}_p^{SP} = u(\tilde{l}_p^{SP})$. In this example, as there is only one $l$ such that $ \tilde{j}(-u(l), -l; -1) = 0$, $(\tilde{l}_p^{SP}, \tilde{u}_p^{SP})$ is unique and therefore it must be the optimal barrier $(l^*_p, u^*_p)$.

\section{The put-call symmetry and American call options} \label{section_call}




In this section, we consider the call option case. To this end, we first derive the put-call symmetry formula so that the results for the put option case in Section \ref{section_put_SN} can be directly used. Because this technique involves a change of measure, throughout this section
we will denote by $\p^{\Psi}$ 
the law of the L\'evy process $X$ with its Laplace exponent 
\[ \Psi(z) := \log \E [e^{z X_1}]
\textrm{ defined for $z \in \mathbb{R}_\Psi := \{z \in\mathbb{R}: \Psi(z) < \infty\}$}.
\]
	We let $\pp^\Psi_s$  be the law of $S$ when $S_0 = s$ (and let $\pp^{\Psi} = \pp^{\Psi}_1$) so that $\p^{\Psi}_x = \pp^\Psi_{\exp(x)}$ (and let $\p^{\Psi} = \p^{\Psi}_0$). Their expectations are defined accordingly.
Throughout this section, we assume that Assumption \ref{assump_lambda_r_alpha} holds under the measure $\pp^\Psi$ and hence we have $\eee^\Psi S_1 < \infty$ and equivalently $1 \in \mathbb{R}_\Psi$.
	
We also define
	\begin{align}
	\Psi_1(z)&:=\Psi(1+z)-\Psi(1)  \quad \textrm{for $z$ satisfying $1+z\in \mathbb{R}_\Psi$,}  \label{Psi_1} \\
	 \Psi_1^d(z)&:= \Psi_1(-z)= \Psi(1-z)- \Psi(1) \quad \textrm{for $z$ satisfying $1-z\in \mathbb{R}_\Psi$.} \label{Psi_1_d} 
	 \end{align} 
	 Note that with $(\mathcal{F}_t^X: t \geq 0)$ the filtration generated by $X$, as in page 82 of \cite{K}, the change of measure gives
\begin{align} \label{change_measure}
\frac{\diff \mathbb{P}^{\Psi_1}}{\diff \mathbb{P}^{\Psi}}\Big|_{\mathcal{F}_t^X}=e^{X_t-\Psi(1)t}, \quad t \geq 0. 
	\end{align}
	 

\subsection{The put-call symmetry}


We will show the very well-known relation between the values of the American put and call options referred to as \textit{the put-call symmetry}. For the rest of this section, we will use the following notation
\begin{align*}
	v_i^{\Psi}(x,K,r;l,u):=\E_x^{\Psi}\Big[e^{-r \tilde{\tau}_{[l,u]}} G_i(e^{ X_{\tilde{\tau}_{[l,u]}}})1_{\{\tilde{\tau}_{[l,u]}<\infty\}}\Big] =\eee_{\exp(x)}^{\Psi}\Big[e^{-r \tau_{[e^l,e^u]}} G_i(S_{\tau_{[e^l,e^u]}})1_{\{\tau_{[e^l,e^u]}<\infty\}}\Big],
\end{align*}
for $l\leq u, \; x \in \R,  \; i = p,c$.




\begin{theorem}[Put-call symmetry]  \label{prop_symmetry}
We have
	\begin{align*}
		v_c^{\Psi}(x,K,r;l,u)&=v_p^{\Psi_1^d}(\log K,e^x,r- \Psi (1);\log K+x-u,\log K+x-l),\quad x \in \R, \; \text{$\log K<l < u$}.
	\end{align*} 


\end{theorem}
\begin{proof}
By the spatial homogeneity of L\'evy processes, 
the change of measure \eqref{change_measure} and recalling \eqref{T_lambda},  
	\begin{align}
		%
		v_c^{\Psi}(x,K,r;l,u) 
		&=  e^x \E^{\Psi}\Big[e^{-r\tilde{\tau}_{[l-x,u-x]}+ X_{\tilde{\tau}_{[l-x,u-x]}}} (1-e^{- (X_{\tilde{\tau}_{[l-x,u-x]}}+x)} K)1_{\{\tilde{\tau}_{[l-x,u-x]}<\infty\}}\Big]\notag \\
				&=  e^x \E \Big[ \E^{\Psi}\Big[e^{-r\tilde{\tau}_{[l-x,u-x]}+ X_{\tilde{\tau}_{[l-x,u-x]}}} (1-e^{- (X_{\tilde{\tau}_{[l-x,u-x]}}+x)} K)1_{\{\tilde{\tau}_{[l-x,u-x]}<\infty\}} \Big| \mathcal{T}^\lambda \Big] \Big] \notag \\
		&=e^x\E^{\Psi_1}\Big[e^{-(r-\Psi(1))\tilde{\tau}_{[l-x,u-x]}}(1- e^{- (X_{\tilde{\tau}_{[l-x,u-x]}} +x - \log K )})1_{\{\tilde{\tau}_{[l-x,u-x]}<\infty\}}\Big]. \notag
	\end{align}
	Therefore, with $\hat{\tau} := \tilde{\tau}_{[l-x-\log K,u-x-\log K]}$,
	\begin{align*}
		v_c^{\Psi}(x,K,r;l,u) &=e^x\E_{-\log K}^{\Psi_1}\Big[e^{-(r-\Psi(1)) \hat{\tau}}(1-e^{- (X_{\hat{\tau}}+x)})1_{\{\hat{\tau}<\infty\}}\Big]\\
		&=\E_{-\log K}^{\Psi_1}\Big[e^{-(r-\Psi(1))\hat{\tau}}(e^x-e^{- X_{\hat{\tau}}})1_{\{\hat{\tau}<\infty\}}\Big] \\
				&= v_p^{\Psi_1^d}(\log K,e^x,r- \Psi (1);\log K+x-u,\log K+x-l).
	\end{align*}
\end{proof}


Similar argument via the change of measure  shows the following.
\begin{lemma} \label{lemma_assumption_finiteness_call}

For the call option, Assumption \ref{assump_tail_value_function} is satisfied on condition that one of the following holds:
\begin{enumerate}
\item 
$r < \Psi(1)$ with  
\begin{align} \label{finiteness_call}
\eee_{s^{-1}}^{\Psi_1^d} [e^{-(r-\Psi(1))T_{\text{last}}(K^{-1} )}
	]<\infty, \quad s > 0
\end{align}
 where
\begin{align}
T_{\text{last}}\left(K^{-1} \right) := \sup\{t\geq0: S_t\leq  K^{-1} \}. \label{tau_last_change_measure}
\end{align}
\item $r \geq \Psi(1)$.
\end{enumerate}
\end{lemma}
\begin{proof}
We have, for $N \geq 1$ and $x=\log s$, by the change of measure \eqref{change_measure},
		\begin{align} \label{call_to_put}
		\sup_{\tau \in \mathcal{A}}\eee_s\Big[e^{-r \tau}(S_\tau-K)^+1_{\{ T_N^\lambda < \tau < \infty  \}}\Big] 
		&=\sup_{\tau \in \mathcal{A}}  \E \Big[e^{-r\tau + X_\tau} ( e^x-e^{- X_{\tau}} K)^+ 1_{\{ T_N^\lambda < \tau < \infty  \} }\Big]\notag\\
& =
 \sup_{\tau \in \mathcal{A}}\E \Big[ \E^{\Psi_1}\Big[e^{-(r-\Psi(1))\tau}( e^x - e^{- X_{\tau}} K)^+1_{\{ T_N^\lambda < \tau < \infty \}}    \Big| \mathcal{T}^\lambda \Big]  \Big] \notag\\
&=\sup_{\tau \in \mathcal{A}} \E^{\Psi_1}\Big[e^{-(r-\Psi(1))\tau}(e^x - e^{- X_{\tau}} K)^+  1_{\{ T_N^\lambda < \tau < \infty \}}   \Big].
 \end{align}
This is dominated by $e^x$ for the case $r\geq \Psi(1)$ 
and by $e^x\E^{\Psi_1}_x\Big[e^{-(r-\Psi(1))\hat{T}_{\text{last}}(K^{-1})} 1_{\{ \hat{T}_{\text{last}}(K^{-1}) > T_N^\lambda   \}}  \Big]$ for the case $r < \Psi(1)$
where $\hat{T}_{\text{last}}\left(K^{-1} \right) := \sup\left\{t\geq0: e^{- X_{\tau}} \leq K^{-1} \right\}$,  
which is under $\p^{\Psi_1^d}_{-x}$ given by  \eqref{tau_last_change_measure}. 
Hence, by dominated convergence (as in the proof of Lemma \ref{assump_tail_value_function}), upon taking $N \rightarrow \infty$, we have the claim.
\end{proof}

If the conditions (1) or (2) in Lemma \ref{lemma_assumption_finiteness_call} are satisfied, then Theorem \ref{theorem_barrier_optimal}(2) applies.  Moreover, Theorem \ref{prop_symmetry} immediately suggests the following.
\begin{remark} 
\begin{enumerate}
\item If $r - \Psi(1) > 0$, then because the put option with a positive discount factor admits an optimal barrier strategy, we must have $\mathcal{D}_c = [L_c^*, \infty)$ for some $L_c^* > K$.
\item If $r - \Psi(1) < 0$ with \eqref{finiteness_call}, then by Lemma \ref{lemma_nonempty_call},  (i) of Theorem \ref{theorem_barrier_optimal}(2) holds, with $U_c^* < \infty$.
\end{enumerate}
\end{remark}


By this remark and Theorem \ref{prop_symmetry}, we obtain the following result.
\begin{corollary}
If $r - \Psi(1) < 0$ and \eqref{finiteness_call} holds, then, for any $x \in \R$, 
\begin{align*}
		(l_c^*, u_c^*) &\in \textrm{arg} \max_{\log K<l < u} v_p^{\Psi_1^d}(\log K,e^x,r- \Psi (1);\log K+x-u,\log K+x-l).
\end{align*}
\end{corollary}

Hence, the computation of the optimal barriers $(L_c^*, U_c^*)$ can be reduced to that of the corresponding put option problem, driven by the \lev process with Laplace exponent  $\Psi_1^d$, with the new discount $r - \Psi(1)$ and the new strike $e^x$ for any fixed $x > 0$ (note that  $(L_c^*, U_c^*)$ to be obtained are invariant of the selection of $x$). 

\subsection{Spectrally one-sided cases} 


Suppose $r - \Psi(1) < 0$ with \eqref{finiteness_call}. For the case $X$ is spectrally one-sided, we can use directly the results obtained for the put option case by following the same procedures as those in Sections \ref{subsection_SN} and \ref{subsection_SP}.

\begin{remark}[Sufficient condition for \eqref{finiteness_call}] \label{remark_Phi_call}
\begin{enumerate} 
\item Suppose $X$ is spectrally negative with $\Psi(z) = \psi(z)$
 (spectrally positive under $\Psi_1^d$). By Remark \ref{remark_Phi_sign_SP},  the condition \eqref{finiteness_call} is satisfied if  $\Phi_1(r-  \Psi (1)) :=\sup \{ z:  \Psi_1(z)= r-  \Psi (1)\} = \sup \{ z:  \psi(1+z) = r \} = \Phi(r) - 1$
  is well defined and positive. 
 \item Suppose $X$ is spectrally positive with $\Psi(z) = \psi(-z)$ (spectrally negative under $\Psi_1^d$). By Remark \ref{remark_Phi_sign},  the condition \eqref{finiteness_call} is satisfied if  $\Phi_1^d(r-\Psi (1)) := \sup \{ z: \Psi_1^d(z) = r-  \Psi (1)\}
  =  \sup \{ z:  \psi(z-1) - \psi(-1) = r - \Psi(1) \} = \sup \{ z:  \psi(z-1) = r  \}  = \Phi(r)+1$ 
  is well defined and negative. 
 \end{enumerate}
\end{remark}
If the solutions  $(\tilde{l}_p, \tilde{u}_p)$ to the first-order  conditions $\mathfrak{C}_l^{SP}$ and  $\mathfrak{C}_u^{SP}$ (resp. $\mathfrak{C}_l^{SN}$ and  $\mathfrak{C}_u^{SN}$) when the original process $X$ is spectrally negative (resp.\ spectrally positive)  are unique, then the optimal barriers to the original call option problem become  $L_c^* := e^{l_c^*}$ and $U_c^* := e^{u_c^*}$ with 
\begin{equation}\label{symmetry}
 u_c^*  =  -\tilde{l}_p  + x + \log K\qquad \text{and}\qquad  l_c^* =  -\tilde{u}_p + x + \log K.
\end{equation}
\section{Numerical results} \label{section_numerics}

In this section, we confirm the analytical results obtained in the previous sections and further analyze the sensitivity with respect to the rate of observation $\lambda$.  Here, we focus on the case driven by spectrally negative and positive \lev processes consisting of a Brownian motion and  i.i.d.\ exponential-size jumps, which are special cases of the double exponential jump diffusion of  Kou \cite{Kou}. They admit explicit forms of  scale functions  as in \cite{Egami_Yamazaki_2010_2, KKR}.

For the spectrally negative case, we assume
\begin{equation}
 X_t - X_0= c t+ \eta B_t - \sum_{n=1}^{N_t} Z_n, \quad 0\le t <\infty, \label{X_phase_type}
\end{equation}
where $B=( B_t: t\ge 0)$ is a standard Brownian motion, $N=(N_t: t\ge 0 )$ is a Poisson process with arrival rate $\alpha$, and  $Z = ( Z_n: n = 1,2,\ldots )$ is an i.i.d.\ sequence of exponential random variables with parameter $\beta$. The processes $B$, $N$, and $Z$ are assumed mutually independent.  For the spectrally positive case, we set $X$ to be the negative of the right hand side of  \eqref{X_phase_type}.  For the parameters describing the problem, we set $K=50$, ${r} = -0.05$, and $\lambda=1$ (so that Assumption \ref{assump_lambda_r} is fulfilled), unless stated otherwise. Other parameters are set so that the optimal strategy is of interval-type.

%
%

\subsection{Put option}
\subsubsection{Spectrally negative case}  \label{subsub_put_SN}

We first consider the put option when $X$ is spectrally negative  and obtain the optimal solutions using the procedure  described in Section \ref{subsection_SN}. Here, we set $c =1$, $\eta = 0.2$, $\alpha = 1$ and $\beta = 2$ for the parameters of $X$. We have $\Phi(r) \approx -0.1064 < 0$ and hence the condition \eqref{phineg} (equivalently Assumptions \ref{assum_phi}  and \ref{assump_X_SN}) and  Assumption \ref{assump_tail_value_function}
 as well are satisfied. 


The plots of the functions  $l \mapsto h(l,u;1)$, $u \mapsto l(u)$ and $u \mapsto \tilde{j}(l(u), u; 1)$ are given in Figure \ref{figure_h_SN} in Section \ref{subsub_first_SN}. Because the pair $(\tilde{l}_p^{SN}, \tilde{u}_p^{SN})$ satisfying simultaneously the conditions $\mathfrak{C}_l^{SN}$ and $\mathfrak{C}_u^{SN}$ is unique in this case, it is the unique maximizer of $(l,u) \mapsto \tilde{v}_p(l,u;1)$ and hence becomes the optimal solution $(l^*_p, u^*_p)$. This is confirmed in Figure \ref{figure_put_SN}(1) where we plot the mapping $(e^l,e^u) \mapsto \tilde{v}_p(l,u;1)$ along with the point at $(\tilde{l}_p^{SN}, \tilde{u}_p^{SN})$. The corresponding value function $V_p(s) = v_p^{SN}(\log s; \tilde{l}_p^{SN}, \tilde{u}_p^{SN})$ is plotted along with the payoff function $G_p$ in Figure \ref{figure_put_SN}(2). Notice that the value function of the auxiliary problem (allowing immediate stopping) becomes $\bar{V}_p(s) = V_p(s) \vee G_p(s)$, $s > 0$. In order to confirm the optimality and study the impact of the choice of the lower barrier $l$ and upper barrier $u$, we plot in Figure \ref{figure_put_SN}(3) and (4) the differences $s \mapsto V_p(s) - v_p(\log s; l, \tilde{u}_p^{SN})$ and $s \mapsto V_p(s)  - v_p(\log s; \tilde{l}_p^{SN}, u)$
for suboptimal choices of $l$ and $u$, including the case $l = -\infty$ computed using the results in \cite{Albrecher, PY_American}. We confirm that these differences are indeed uniformly positive.

\begin{figure}[htbp]
\begin{center}
\begin{minipage}{1.0\textwidth}
\centering
\begin{tabular}{cc}
 \includegraphics[scale=0.4]{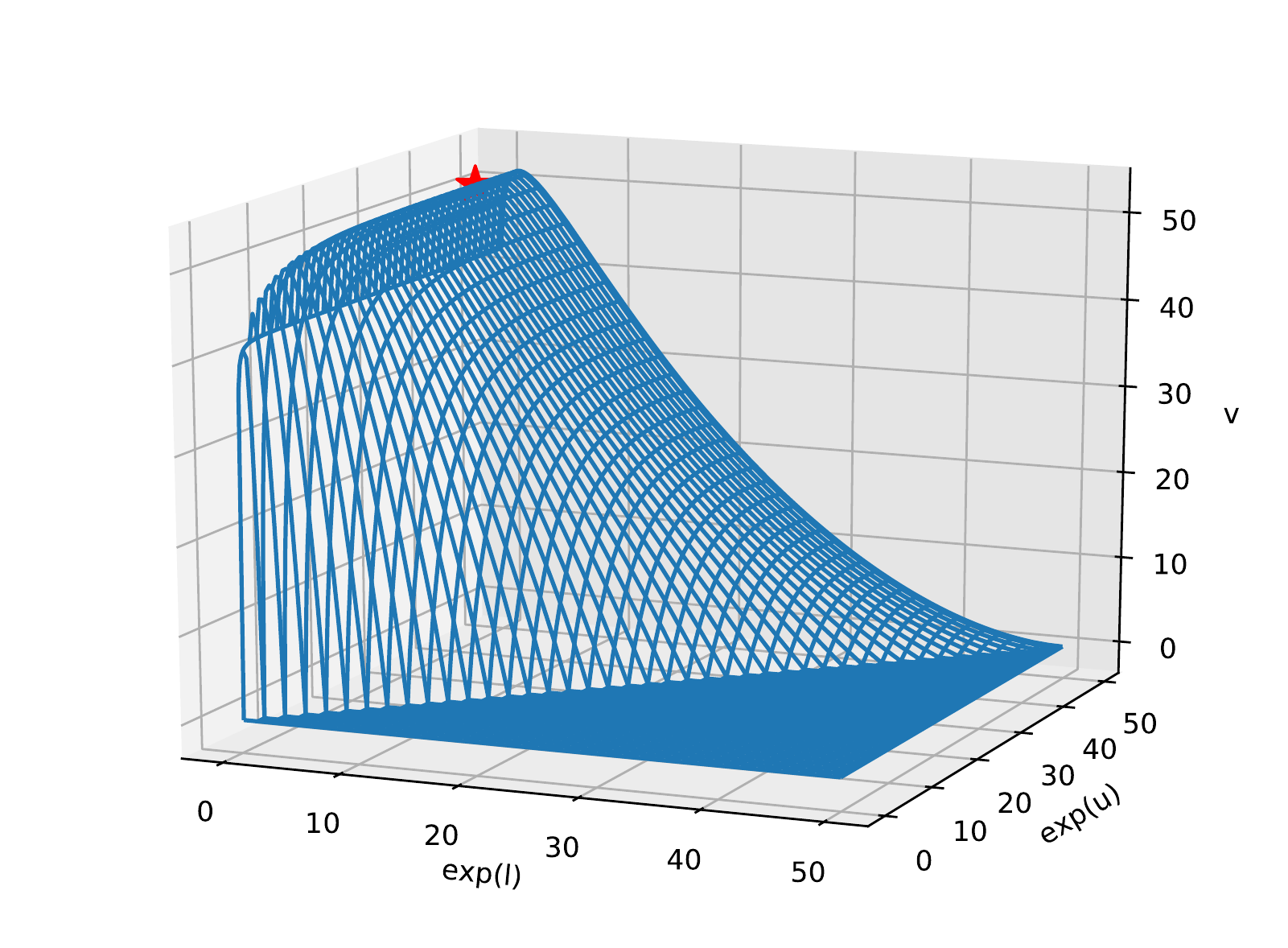} & \includegraphics[scale=0.4]{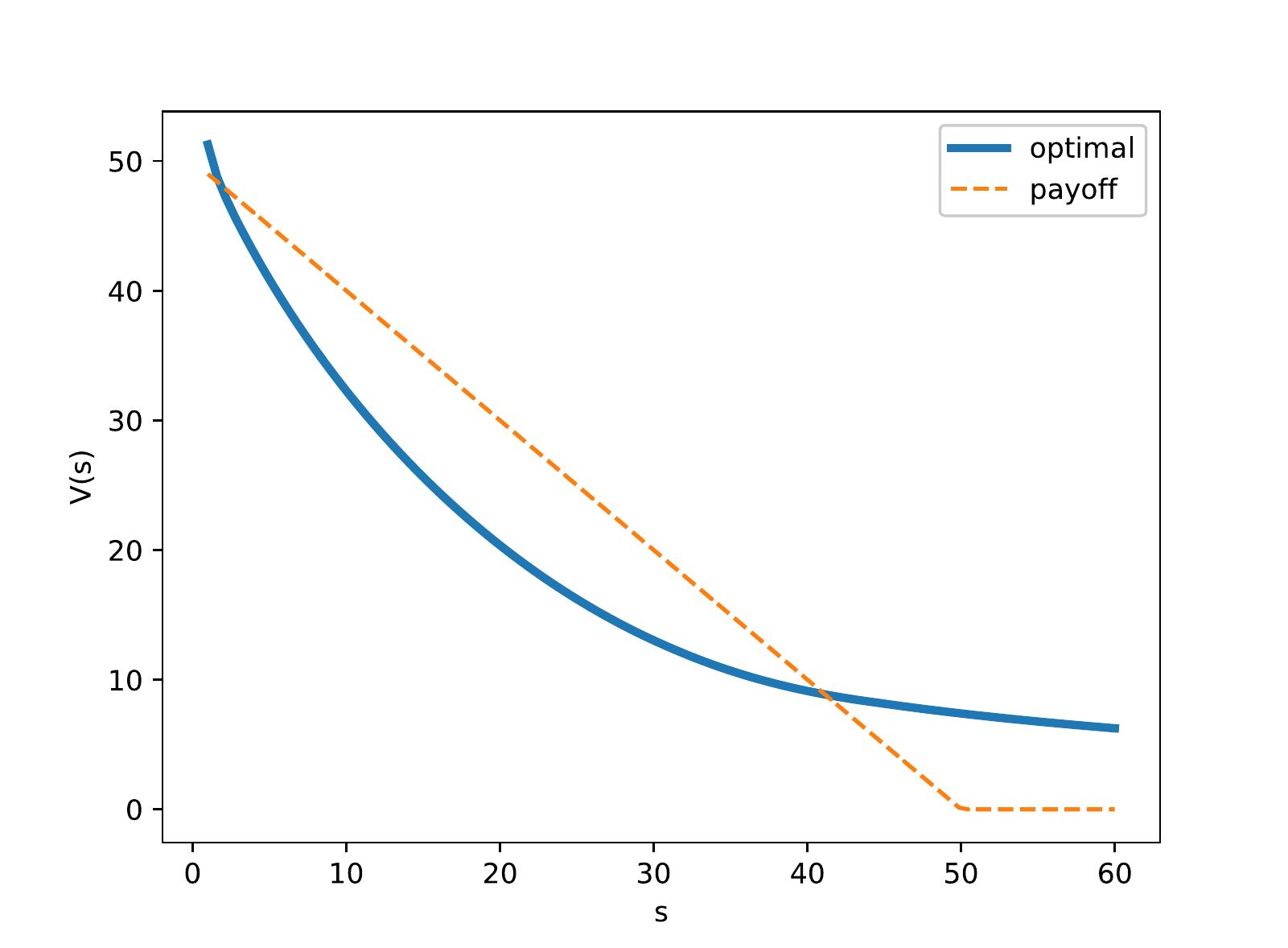} \\
(1) $(e^l,e^u) \mapsto \tilde{v}_p(l,u;1)$ & (2) $V_p(s)  = v_p^{SN}(\log s; \tilde{l}_p^{SN}, \tilde{u}_p^{SN})$ and $G_p(s)$ \\
  \includegraphics[scale=0.4]{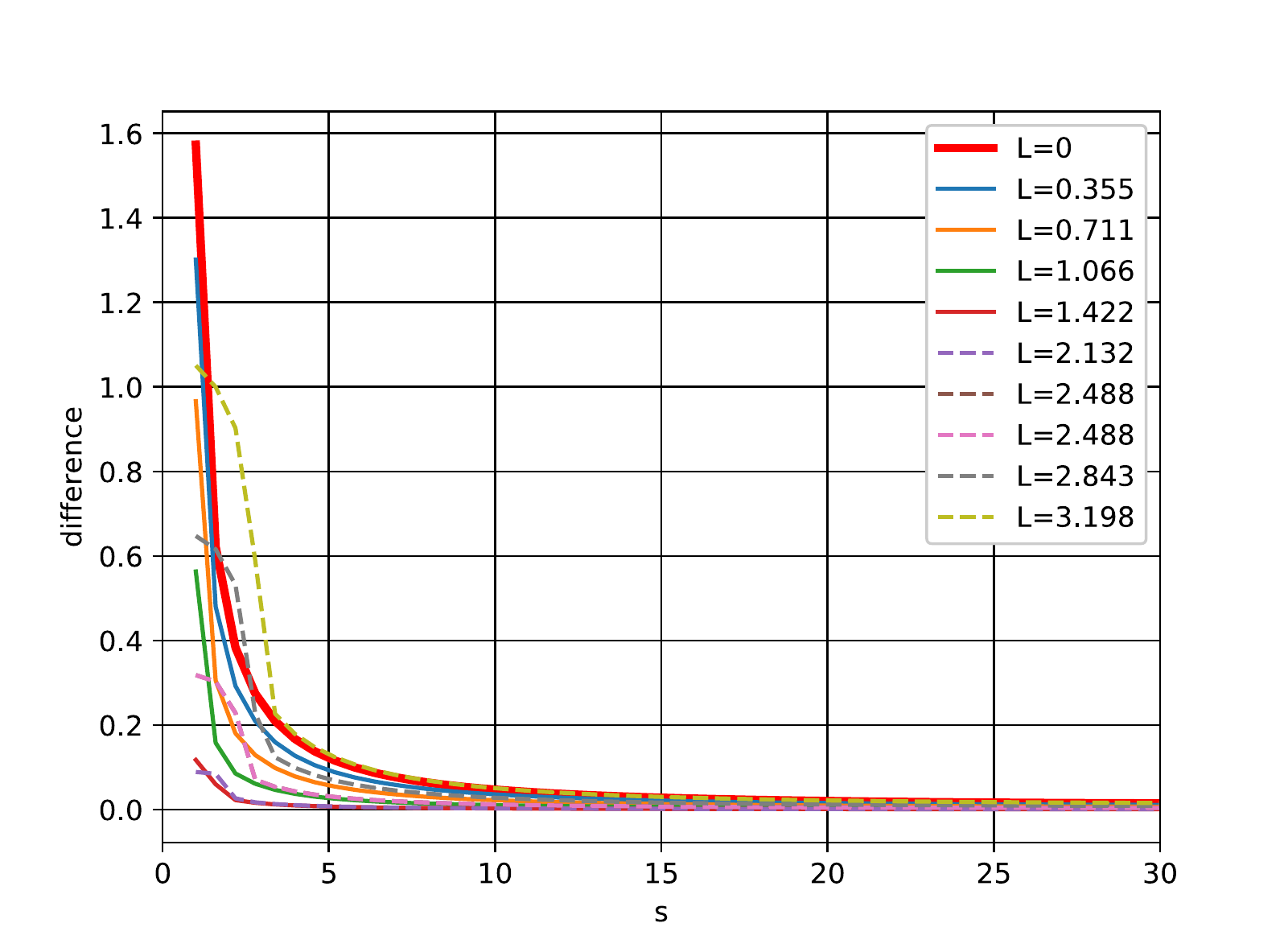} & \includegraphics[scale=0.4]{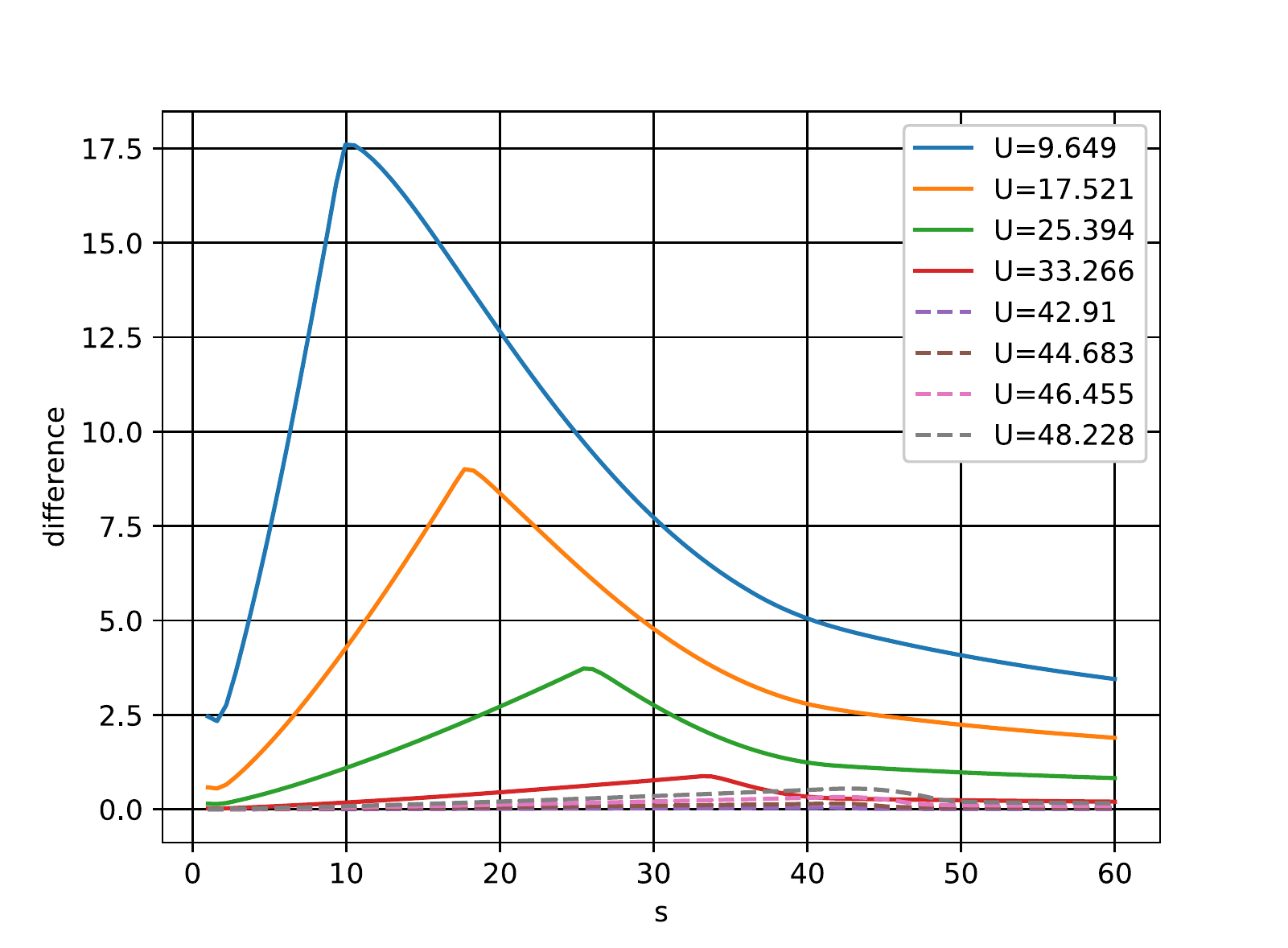} \\
 (3) $V_p(s)  - v_p^{SN}(\log s; l, \tilde{u}_p^{SN})$ & (4) $V_p(s)  - v_p^{SN}(\log s; \tilde{l}_p^{SN}, u)$
 \end{tabular}
\end{minipage}
\caption{\footnotesize Put option when $X$ is spectrally negative. (1) The mapping $(e^l,e^u) \mapsto \tilde{v}_p(l,u;1)$ for $0 < e^l < e^u < K$ (for $u < l$, we set the value to be zero). The point at $(L_p^*, U_p^*) = (e^{\tilde{l}_p^{SN}},e^{\tilde{u}_p^{SN}})$ is shown by the red star. (2) The value function $s \mapsto V_p(s)  = v_p^{SN}(\log s;\tilde{l}_p^{SN}, \tilde{u}_p^{SN})$ along with the payoff function $s \mapsto G_p(s)$. (3) The differences  $s \mapsto V_p(s)  - v_p^{SN}(\log s; l, \tilde{u}_p^{SN})$ for $l = \log L$ with $L = 0$ and $L= i L_p^*/5$ for $i = 1,2,3,4,6,7,8,9$ 
(solid lines when $L < L_p^*$ and dashed lines when $L > L_p^*$).
 (4) The differences $s \mapsto V_p(s)  - v_p^{SN}(\log s; \tilde{l}_p^{SN}, u)$ for $u = \log U$ with $U = \frac {5-i} 5L_p^* + \frac i 5 U_p^*$ and $\frac {5-i} 5U_p^* + \frac i 5 K$ for $i = 1,2,3,4$
 (solid lines when $U < U_p^*$ and dashed lines when $U > U_p^*$).
} \label{figure_put_SN}
\end{center}
\end{figure}

\subsubsection{Spectrally positive case} \label{subsub_put_SP}

We now move on to the put option when $X$ is  spectrally positive and confirm the results obtained in Section \ref{subsection_SP}. Here, the dual (spectrally negative) \lev process $X^d = -X$ is assumed to be given by the right-hand side of \eqref{X_phase_type} with  $c =0.2$, $\eta = 0.3$, $\alpha = 1$ and $\beta = 2$. We have $\Phi(r) \approx 1.3568 > 0$ and hence the condition \eqref{Phi_cond_SP} (equivalently Assumptions \ref{assum_phi}  and \ref{assum_phi_SP}) and  Assumption \ref{assump_tail_value_function}
 as well are satisfied. 


The plots of the corresponding functions  $u \mapsto h(-u,-l; -1)$, $l \mapsto u(l)$ and $l \mapsto \tilde{j}(-u(l), -l; -1)$  are given in Figure \ref{figure_h_SP} in Section \ref{subsubsection_first_order_SP}. We again attain a unique pair $(\tilde{l}_p^{SP}, \tilde{u}_p^{SP})$ satisfying simultaneously the conditions $\mathfrak{C}_l^{SP}$ and $\mathfrak{C}_u^{SP}$ and hence it becomes $(l^*_p, u^*_p)$.   In Figure \ref{figure_put_SP}, we plot the results analogous to those given in Figure \ref{figure_put_SN}. The optimality is confirmed similarly.

\begin{figure}[htbp]
\begin{center}
\begin{minipage}{1.0\textwidth}
\centering
\begin{tabular}{cc}
 \includegraphics[scale=0.4]{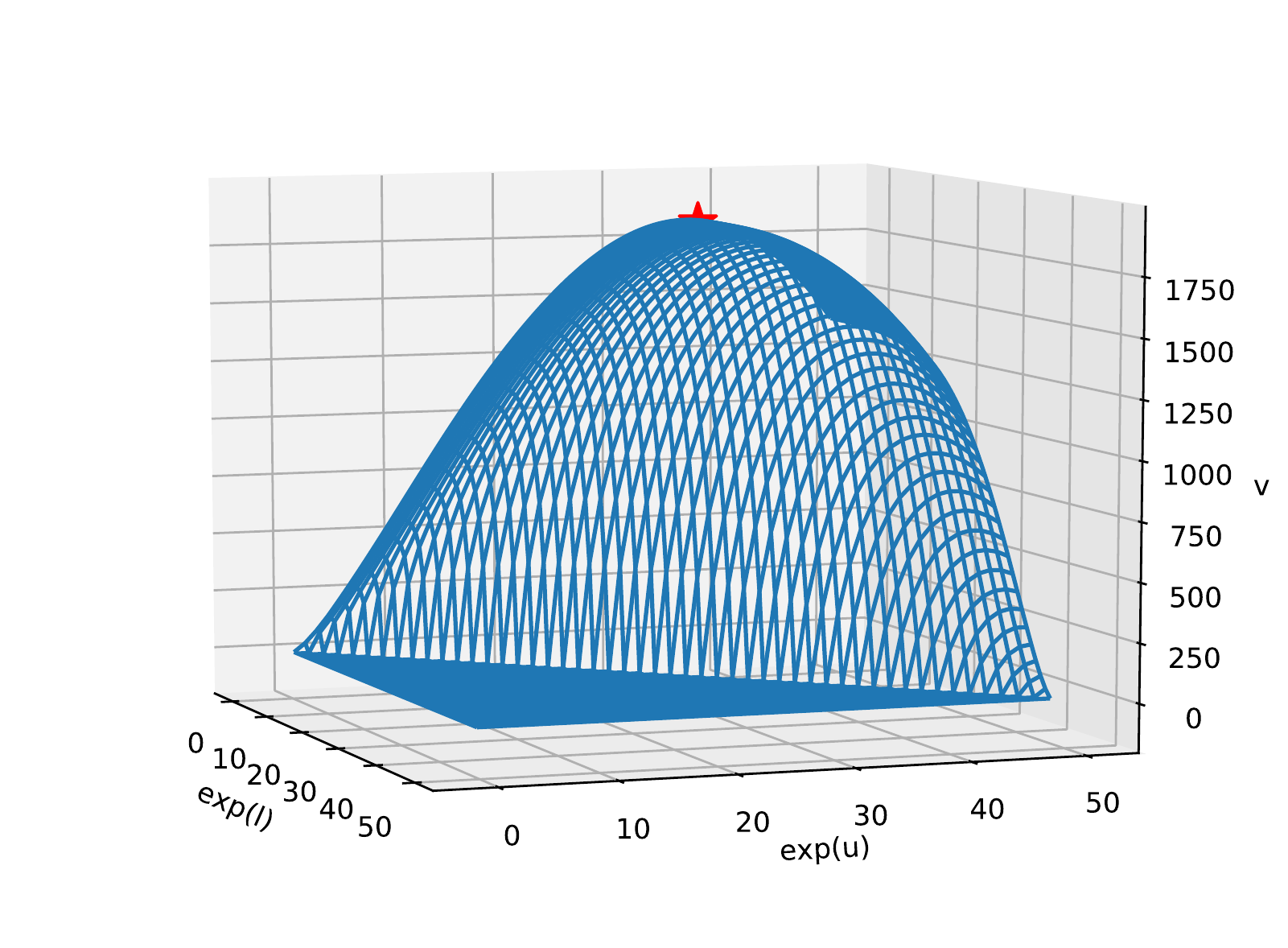} & \includegraphics[scale=0.4]{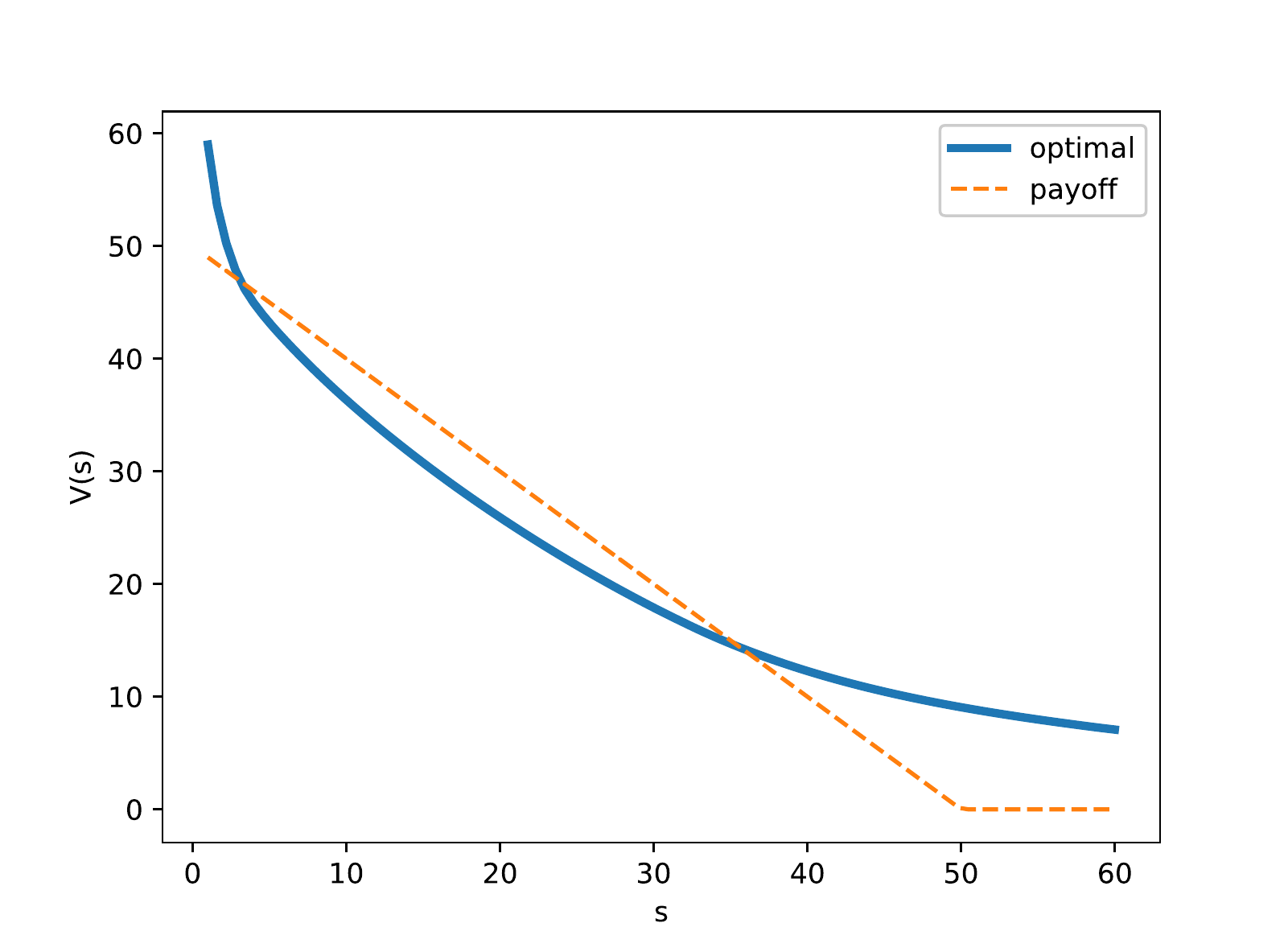} \\
 (1) $(e^l,e^u) \mapsto \tilde{v}_p(-u,-l;-1)$ & (2) $V_p(s) = v_p^{SP}(\log s; l_p^*, u_p^*)$ and $G_p(s)$ \\
  \includegraphics[scale=0.4]{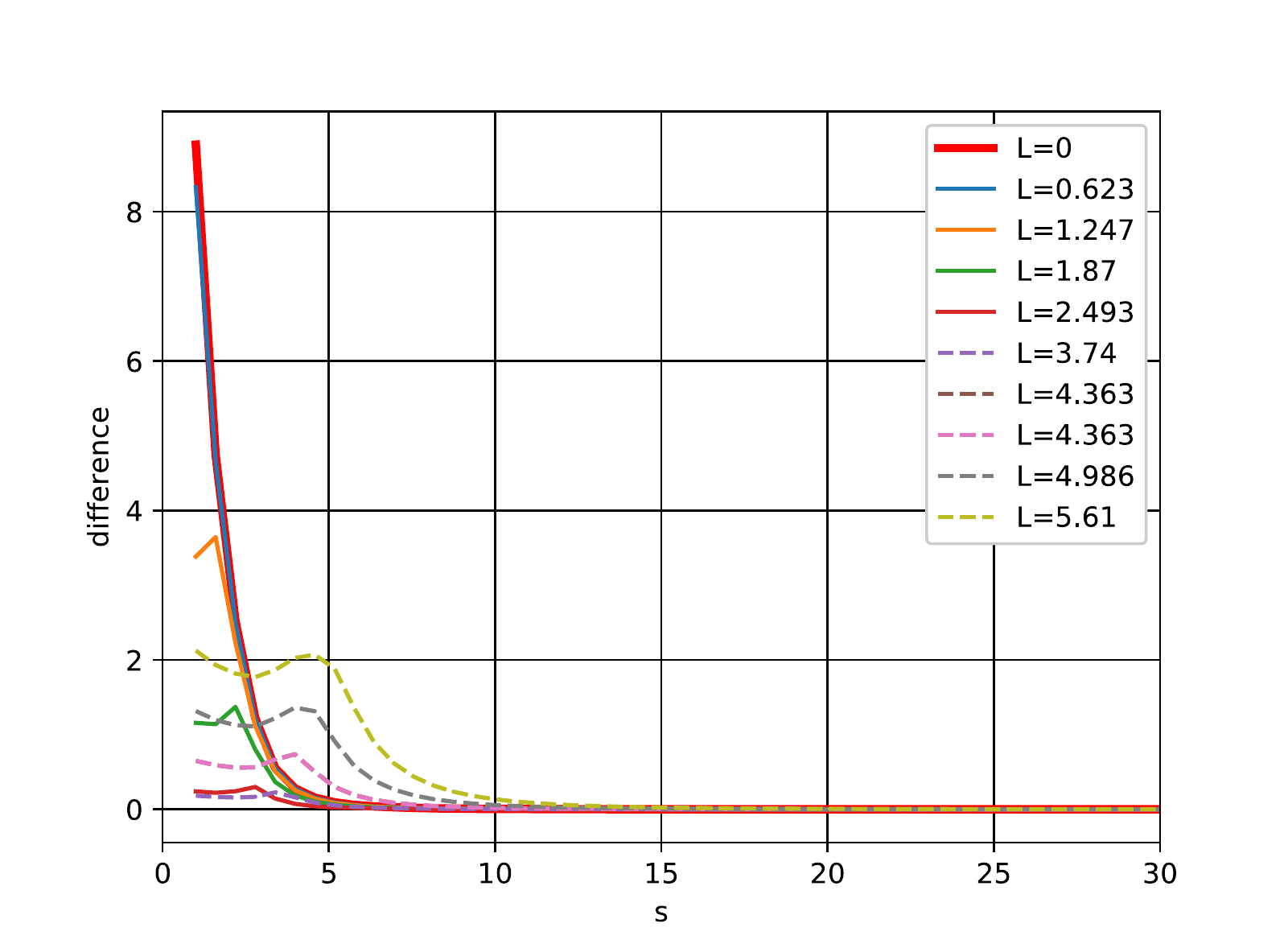} & \includegraphics[scale=0.4]{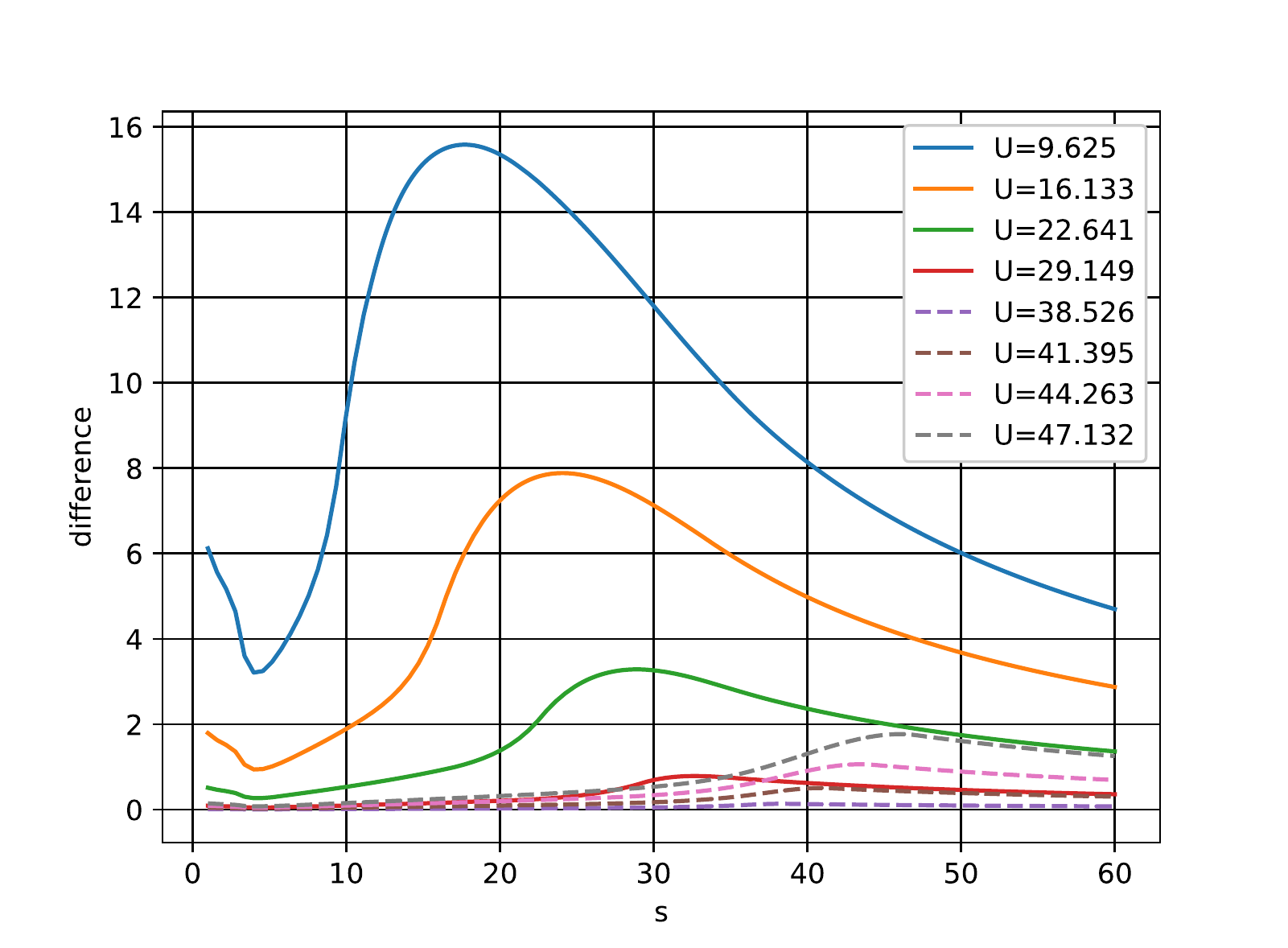} \\
   (3) $V_p(s) - v_p^{SP}(\log s; l, u^*_p)$ & (4) $V_p(s)  - v_p^{SP}(\log s; l^*_p, u)$
 \end{tabular}
\end{minipage}
\caption{\footnotesize  Put option when $X$ is spectrally positive. (1) The mapping $(e^l,e^u) \mapsto \tilde{v}_p(-u,-l;-1)$ for $0 < e^l < e^u < K$ (for $u < l$, we set the value to be zero). The point at $(L_p^*, U_p^*) = (e^{\tilde{l}_p^{SP}},e^{\tilde{u}_p^{SP}})$ is shown by the red star. (2) The value function $s \mapsto V_p(s) = v_p^{SP}(\log s; \tilde{l}_p^{SP}, \tilde{u}_p^{SP})$ along with the payoff function $s \mapsto G_p(s)$. (3) The difference  $s \mapsto V_p(s)  - v_p^{SP}(\log s; l, \tilde{u}_p^{SP})$  for $l = \log L$ with $L = 0$ and $L= i L_p^*/5$ for $i = 1,2,3,4,6,7,8,9$ 
(solid lines when $L < L_p^*$ and dashed lines when $L > L_p^*$).
 (4) The difference $s \mapsto V_p(s)  - v_p^{SP}(\log s; \tilde{l}_p^{SP}, u)$  for $u = \log U$ with $U = \frac {5-i} 5L_p^* + \frac i 5 U_p^*$ and $\frac {5-i} 5U_p^* + \frac i 5 K$ for $i = 1,2,3,4$
 (solid lines when $U < U_p^*$ and dashed lines when $U > U_p^*$).
}  
 \label{figure_put_SP}
\end{center}
\end{figure}

\subsection{Call options}

For call options,  the computation of the optimal solution boils down to that of a put option, thanks to the put-call symmetry as studied in Section \ref{section_call}.  Here, we first transform the problem to the corresponding put option problem and solve it using the same procedures used in Sections \ref{subsub_put_SN} and  \ref{subsub_put_SP}.

\subsubsection{Spectrally negative case}

We consider the call option when $X$ is a spectrally negative \lev process, whose Laplace exponent $\Psi = \psi$ is given by \eqref{X_phase_type} with $c = 0.140625$, $\eta = 0.15$, $\alpha =2.193125$ and $\beta = 10$. 

To solve this, we consider the put option driven by the spectrally positive \lev process with its Laplace exponent given by $\Psi_1^d(z) = \psi(1-z) - \psi(1)$ as in \eqref{Psi_1_d} and a new discount factor $\tilde{r} = r - \Psi(1) = -0.0025 < 0$. In order to use Theorem \ref{prop_symmetry}, we set $x = 5$ and a new strike $\tilde{K} = e^x = 148.41$.  After the optimal barriers for this auxiliary problem are computed, those for the original call option problem are  recovered by  \eqref{symmetry}. Notice that the recovered values of the barriers are invariant of the selection of $x$. 

Recall Remark \ref{remark_Phi_call}(1). In Figure \ref{u_CALL_SN}(1), we plot  $z \mapsto \Psi_1(z) := \Psi_1^d(-z)$ (see \eqref{Psi_1_d}), corresponding to the Laplace exponent of the dual (spectrally negative) \lev process  $\tilde{X}^d = \tilde{X}$. Here,  $\Phi_1(\tilde{r})  \approx 0.4760> 0$ and hence the condition \eqref{Phi_cond_SP} (equivalently Assumptions \ref{assum_phi}  and \ref{assum_phi_SP}) and  Assumption \ref{assump_tail_value_function} are satisfied. 
The solutions  $(\tilde{l}_p^{SP}, \tilde{u}_p^{SP})$ to the first-order conditions $\mathfrak{C}_l^{SP}$ and $\mathfrak{C}_u^{SP}$ are computed in the same way as in Section \ref{subsub_put_SP}, and here again we obtain a unique pair. In Figure \ref{u_CALL_SN}(2), we plot   $(l,u) \mapsto \tilde{v}_p(-u,-l;-1)$, confirming that $(\tilde{l}_p^{SP}, \tilde{u}_p^{SP})$ indeed maximizes it.  The optimal barriers for the original call option problem become $\tilde{u}_c^{SN} =x+ \log K - \tilde{l}_p^{SP}$ and $\tilde{l}_c^{SN} = x+\log K-\tilde{u}_p^{SP}$ by \eqref{symmetry}.




The value function $V_c(s) = v_c^{SN}(\log s; \tilde{l}_c^{SN}, \tilde{u}_c^{SN})$ of the original problem is plotted along with the payoff function $G_c$ in Figure \ref{figure_call_SN}(2). Notice that $\bar{V}_c(s) = V_c(s) \vee G_c(s)$, $s > 0$. We also plot in Figure \ref{figure_call_SN}(3) and (4) the differences $s \mapsto V_c(s)  - v_c(\log s; l, \tilde{u}_c^{SN})$ and $s \mapsto V_c(s) - v_c(\log s; \tilde{l}_c^{SN}, u)$
for suboptimal choices of $l$ and $u$, including the case $u = \infty$ computed using  the results in \cite{Albrecher, PY_American}. We confirm that these differences are indeed uniformly positive.

\begin{figure}[htbp]
\begin{center}
\begin{minipage}{1.0\textwidth}
\centering
\begin{tabular}{cc}
 \includegraphics[scale=0.4]{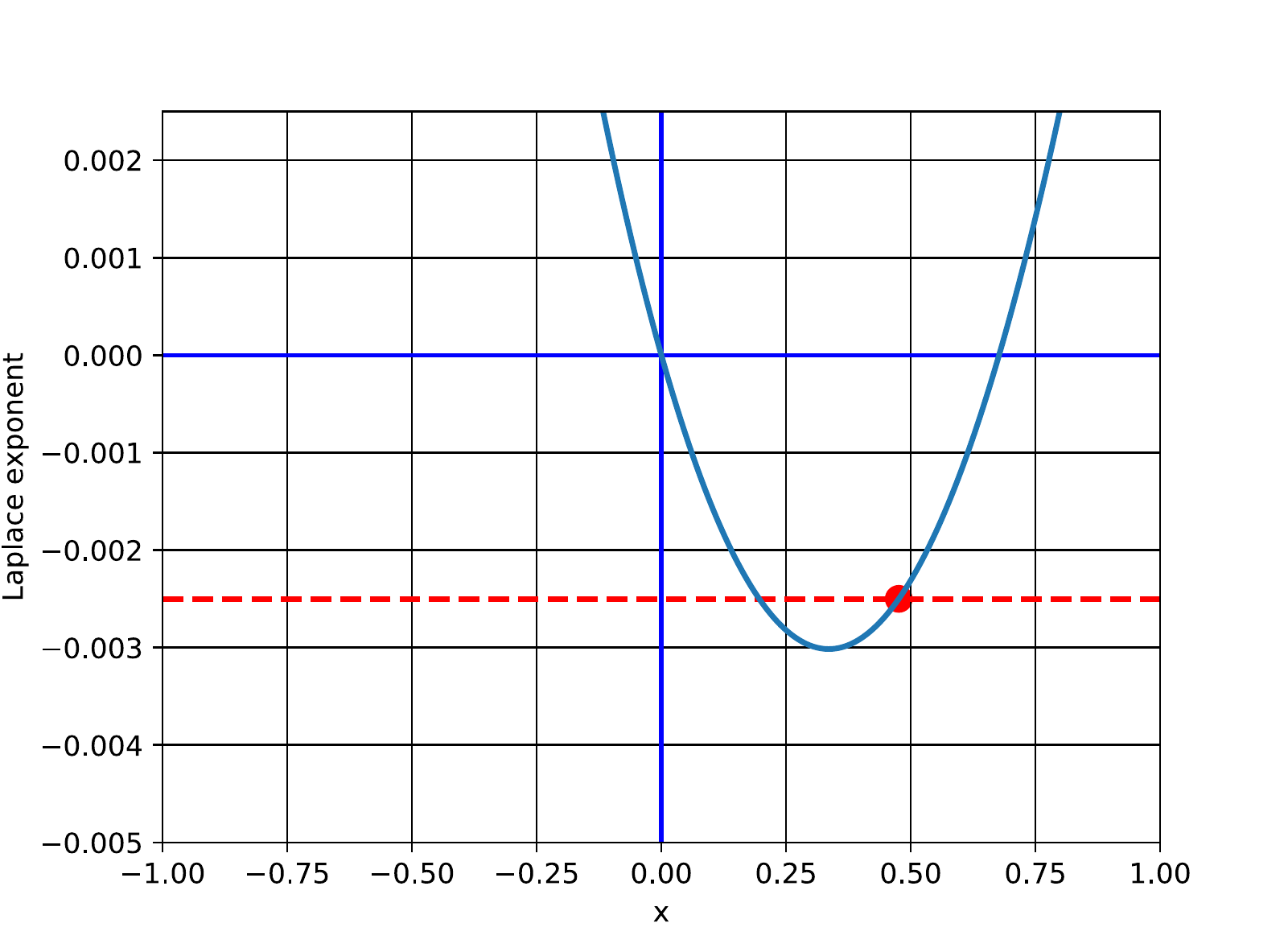} &  \includegraphics[scale=0.4]{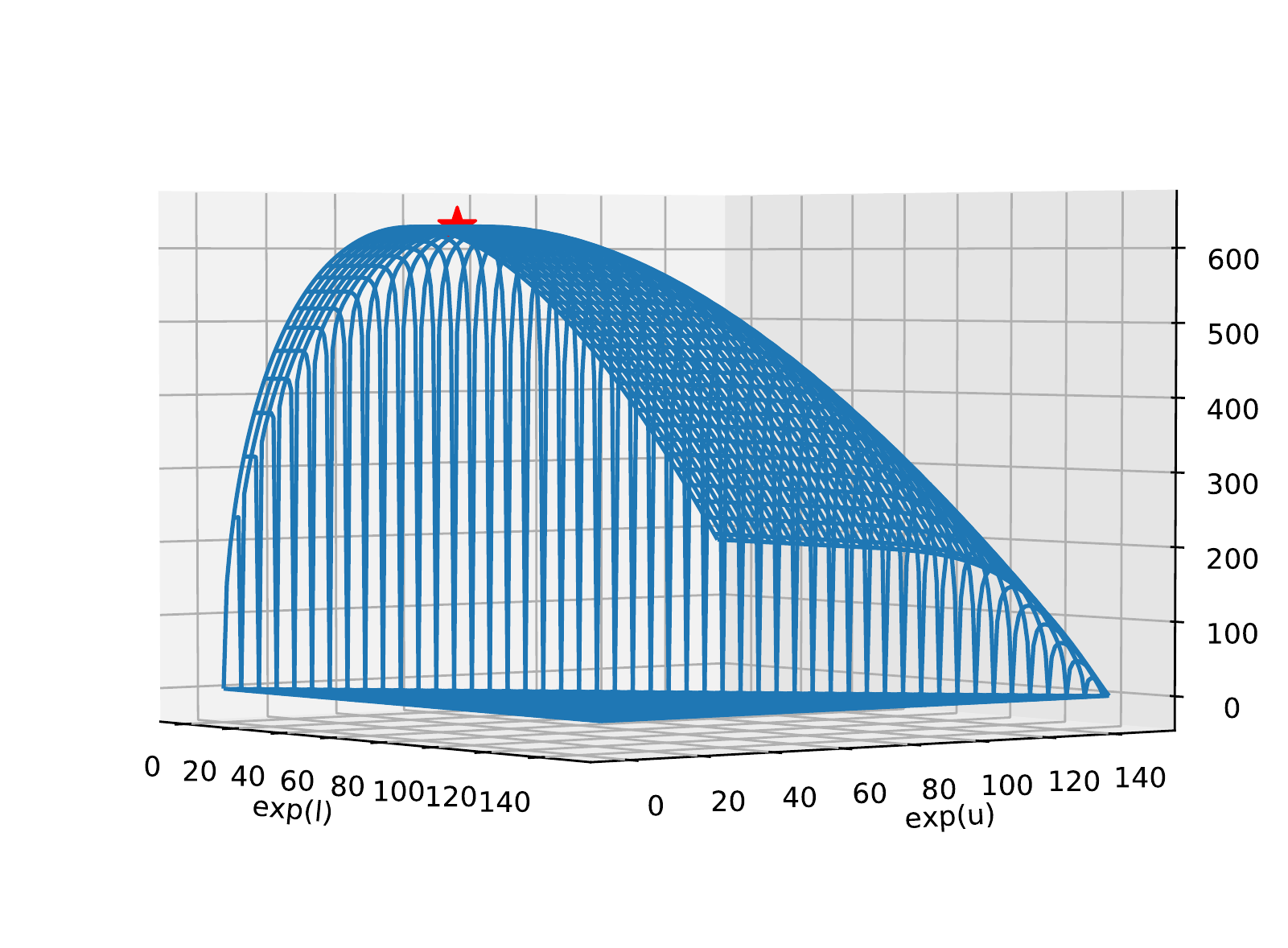}  \\
(1) $x \mapsto \Psi_1(x)$ & (2) $(e^l,e^u) \mapsto \tilde{v}_p(-u,-l;-1)$ 
 \end{tabular}
\end{minipage}
\caption{\footnotesize Auxiliary put option problem for the call option when $X$ is spectrally negative. (1) The Laplace exponent of the dual (spectrally negative) \lev process $\Psi_1$ (solid) along with the horizontal line $y = \tilde{r}$ (dashed). The point at 
$\Phi_1(\tilde{r})$
is given by the red circle. (2) The mapping $(e^l,e^u) \mapsto \tilde{v}_p(-u,-l;-1)$ for $0 < e^l < e^u < \tilde{K}$ (for $u < l$, we set the value to be zero). The point at $(e^{\tilde{l}_p^{SP}},e^{\tilde{u}_p^{SP}})$ is shown by the red star.
} \label{u_CALL_SN}
\end{center}
\end{figure}

\begin{figure}[htbp]
\begin{center}
\begin{minipage}{1.0\textwidth}
\centering
\begin{tabular}{cc}
\multicolumn{2}{c}{ \includegraphics[scale=0.4]{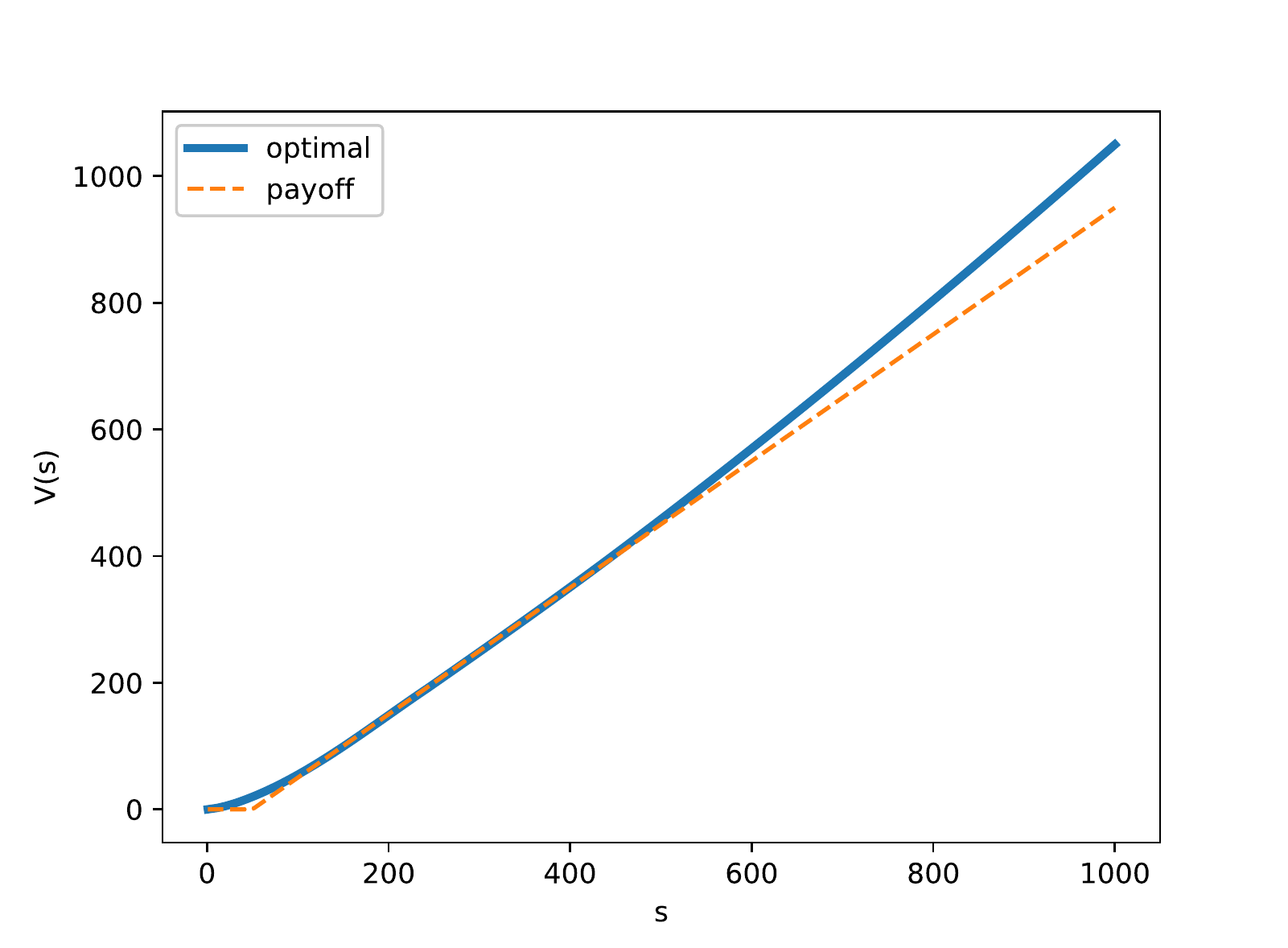} }  \\
\multicolumn{2}{c}{ (1) $V_c(s) = v_c^{SN}(\log s; l_c^*, u_c^*)$ and $G_c(s)$ }  \\
  \includegraphics[scale=0.4]{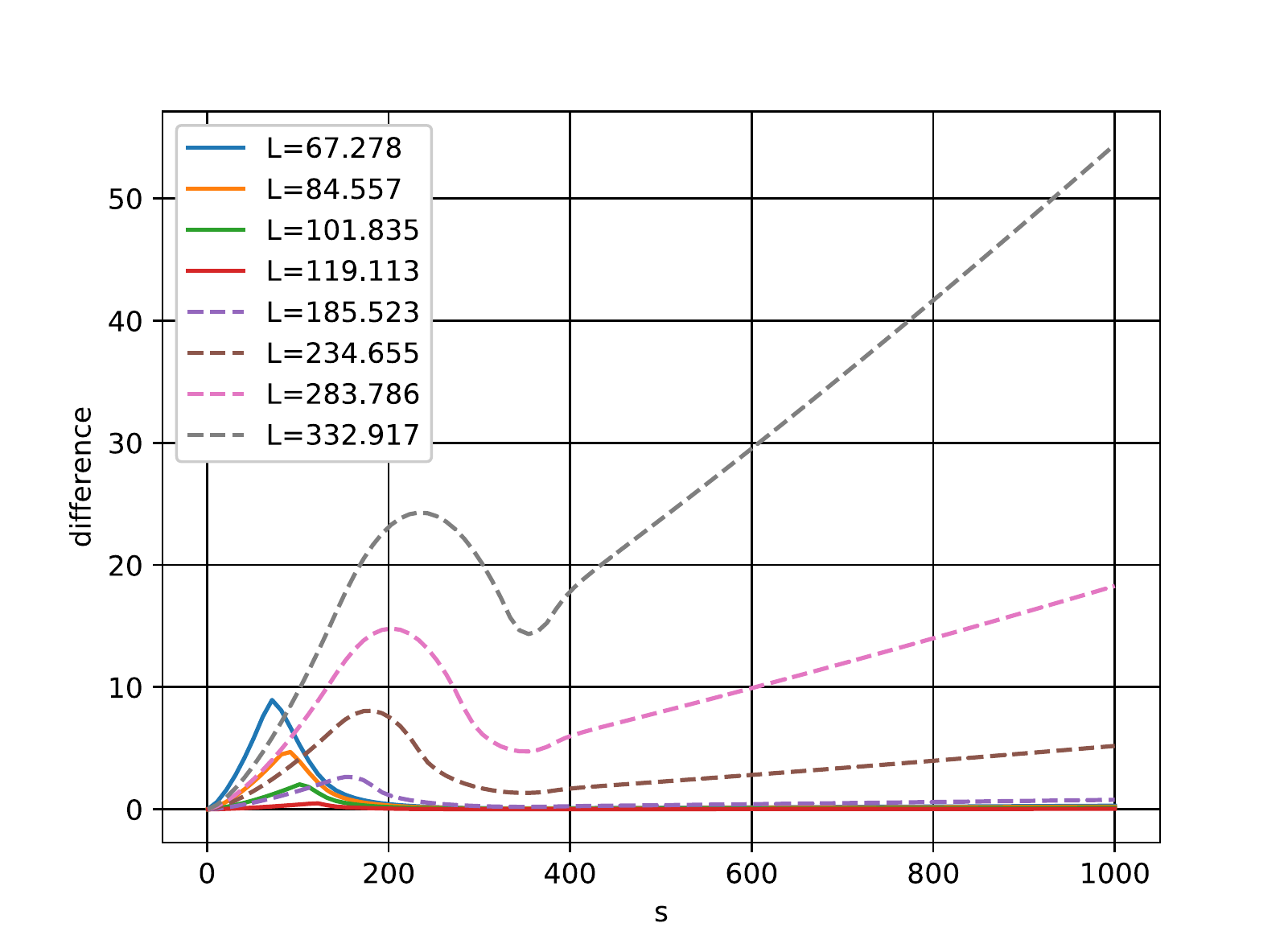} & \includegraphics[scale=0.4]{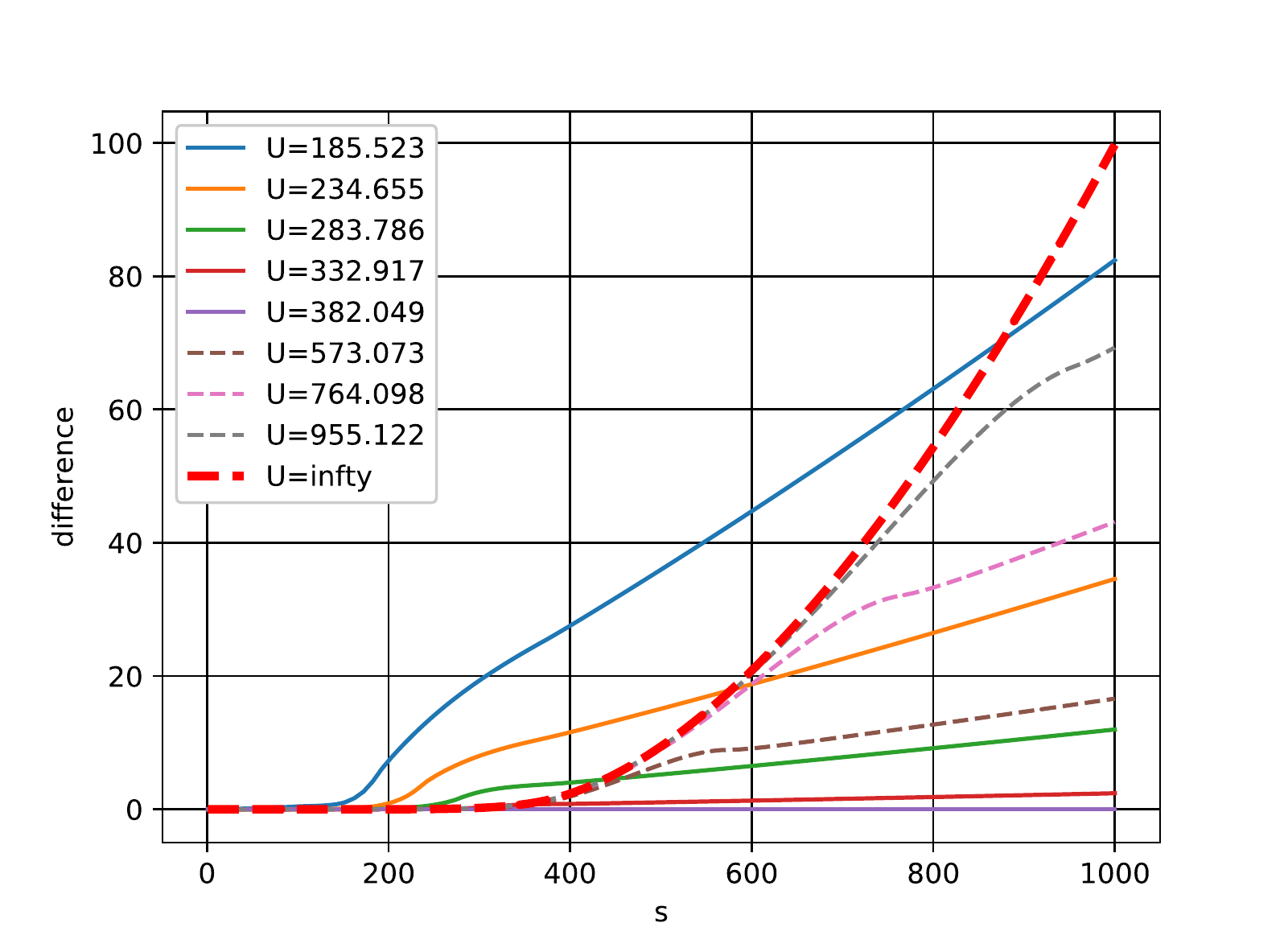} \\
   (2) $V_c(s)- v_p^{SN}(\log s; l, u^*_c)$ & (3) $V_c(s) - v_c^{SN}(\log s; l^*_c, u)$
 \end{tabular}
\end{minipage}
\caption{\footnotesize  
 Call option when $X$ is spectrally negative.  (1) The value function $s \mapsto V_c(s) = v_c^{SN}(\log s; l_c^*, u_c^*)$ along with the payoff function $s \mapsto G_c(s)$. (2) The difference  $s \mapsto V_c(s) - v_c^{SN}(\log s; l, u^*_c)$  for $l = \log L$ with $L= \frac {5-i} 5 K + \frac i 5 L_c^*$ and $L= \frac {5-i} 5 L_c^* + \frac i 5 U_c^*$ for $i = 1,2,3,4$    (solid when $L < L_c^*$ and dashed when $L > L_c^*$).
 (3) The difference $s \mapsto V_c(s) - v_c^{SN}(\log s; l^*_c, u)$  for $u = \log U$ with $U = \frac {5-i} 5L_c^* + \frac i 5 U_c^*$ for $i = 1,2,3,4$, $U=1.5U_c^*$, $2U_c^*$, $2.5U_c^*$, and $U=\infty$  (solid lines when $U < U_c^*$ and dashed lines when $U > U_c^*$).
}  \label{figure_call_SN}
\end{center}
\end{figure}

\subsubsection{Spectrally positive case}

Similarly, we solve the call option case when $X$ is a spectrally positive \lev process with its Laplace exponent $\Psi$. We let its dual $X^d = - X$ be given by the right-hand side of \eqref{X_phase_type} with $c = 0.221875$, $\eta = 0.15$, $\alpha =1.4681$ and $\beta = 10$ and its  Laplace exponent be $\psi(s) = \Psi(-s)$. 

Recall Remark \ref{remark_Phi_call}(2). We consider the put option driven by the spectrally negative \lev process with its Laplace exponent given by $\Psi_1^d(z) = \Psi(1-z) - \Psi(1) = \psi(z-1) - \psi(-1)$, which is plotted in Figure \ref{u_CALL_SP}(1). The new discount factor becomes $\tilde{r} := r - \Psi(1) = -0.0025 < 0$. We again set $x = 5$ and the new strike becomes $\tilde{K} = e^x = 148.41$.    Here,  $\Phi_1^d(\tilde{r})  \approx -0.23632 < 0$ and hence 
the condition \eqref{phineg} (equivalently Assumptions \ref{assum_phi}  and \ref{assump_X_SN}) and Assumption \ref{assump_tail_value_function}
 as well are satisfied. 
Figure \ref{u_CALL_SP}(2) plots $(l,u) \mapsto \tilde{v}_p(l,u;1)$ confirming that  the solution to the first-order conditions in the auxiliary problem $(\tilde{l}_p^{SN}, \tilde{u}_p^{SN})$ is unique. The optimal barriers of the original call option problem are  recovered by  \eqref{symmetry}. Analogously to Figure \ref{figure_call_SN}, as seen in Figure \ref{figure_call_SP}, the optimality of the selected strategy is confirmed.

\begin{figure}[htbp]
\begin{center}
\begin{minipage}{1.0\textwidth}
\centering
\begin{tabular}{cc}
 \includegraphics[scale=0.4]{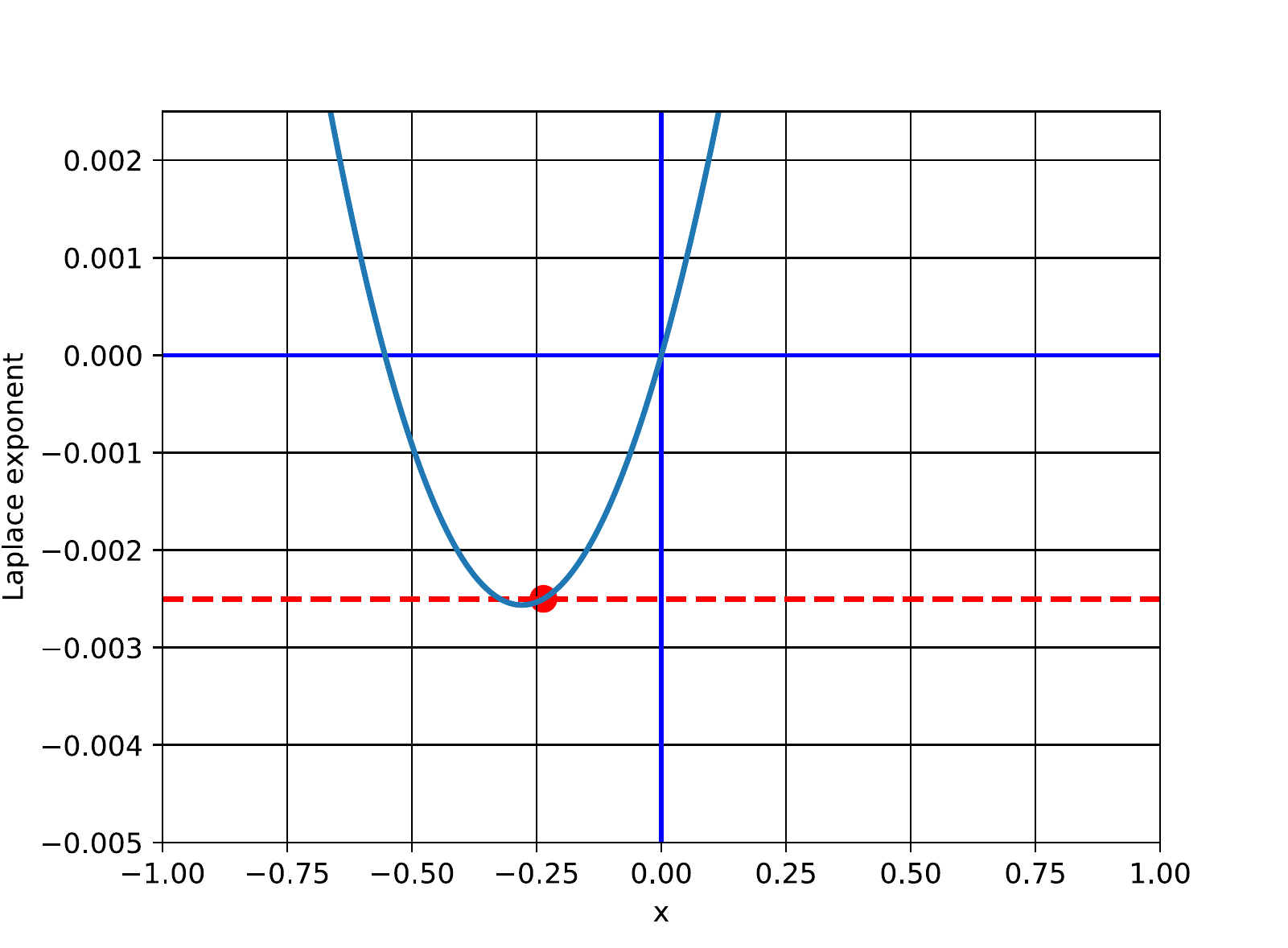} &  \includegraphics[scale=0.4]{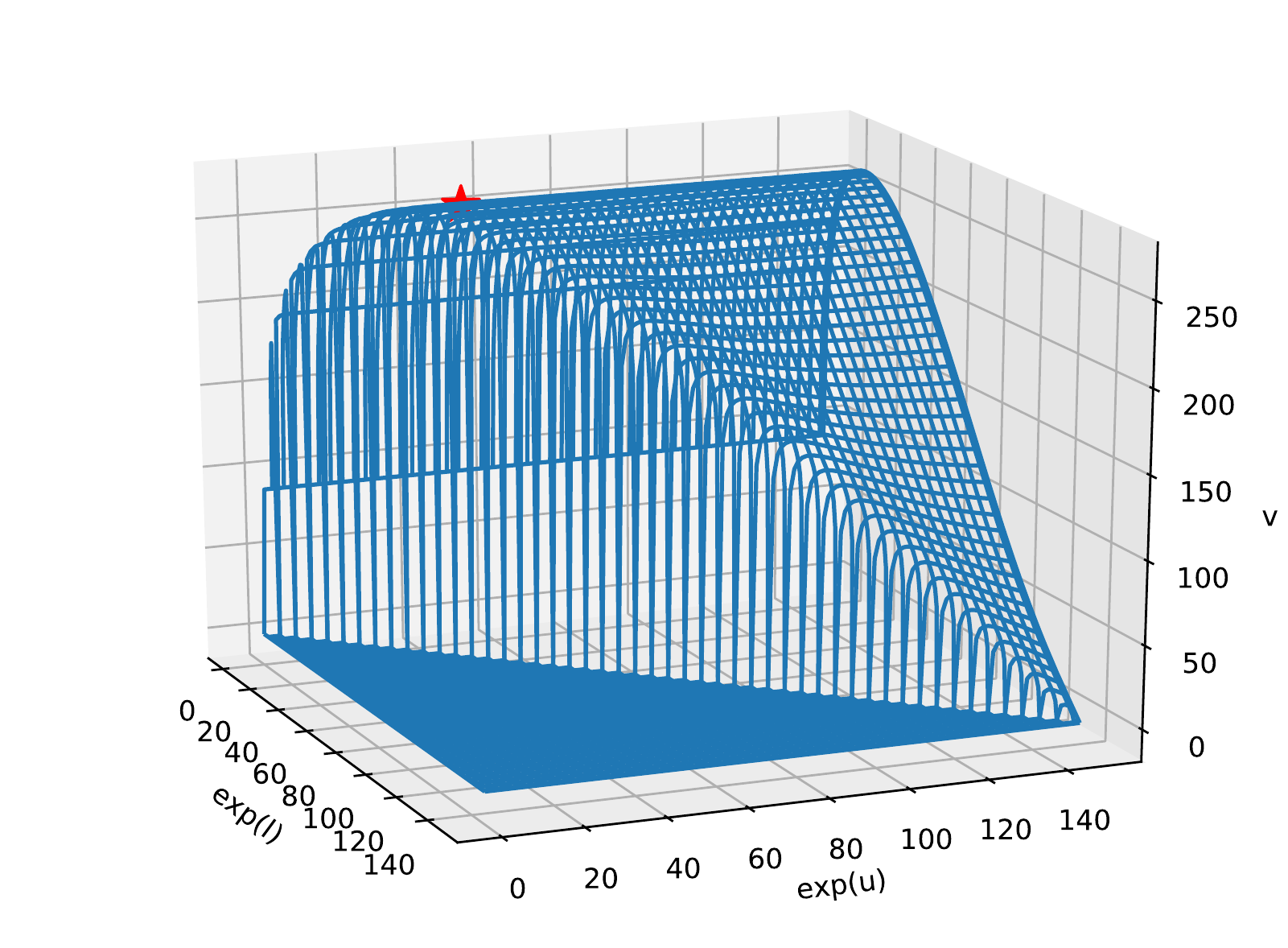}  \\
(1) $z \mapsto \Psi_1^d(z)$ & (2) $(e^l,e^u) \mapsto \tilde{v}_p(l,u;1)$ 
 \end{tabular}
\end{minipage}
\caption{\footnotesize  Auxiliary put option problem for the call option when $X$ is spectrally positive. (1) The Laplace exponent  $\Psi_1^d$ (solid) along with the horizontal line $y = \tilde{r}$ (dashed). The point at $\Phi_1^d(\tilde{r})$
  is given by the red circle. (2) The mapping $(e^l,e^u) \mapsto \tilde{v}_p(l,u,1)$ for $0 < e^l < e^u < \tilde{K}$ (for $u < l$, we set the value to be zero). The point at $(e^{\tilde{l}_p^{SN}},e^{\tilde{u}_p^{SN}})$ is shown by the red star.
} \label{u_CALL_SP}
\end{center}
\end{figure}


\begin{figure}[htbp]
\begin{center}
\begin{minipage}{1.0\textwidth}
\centering
\begin{tabular}{cc}
\multicolumn{2}{c}{ \includegraphics[scale=0.4]{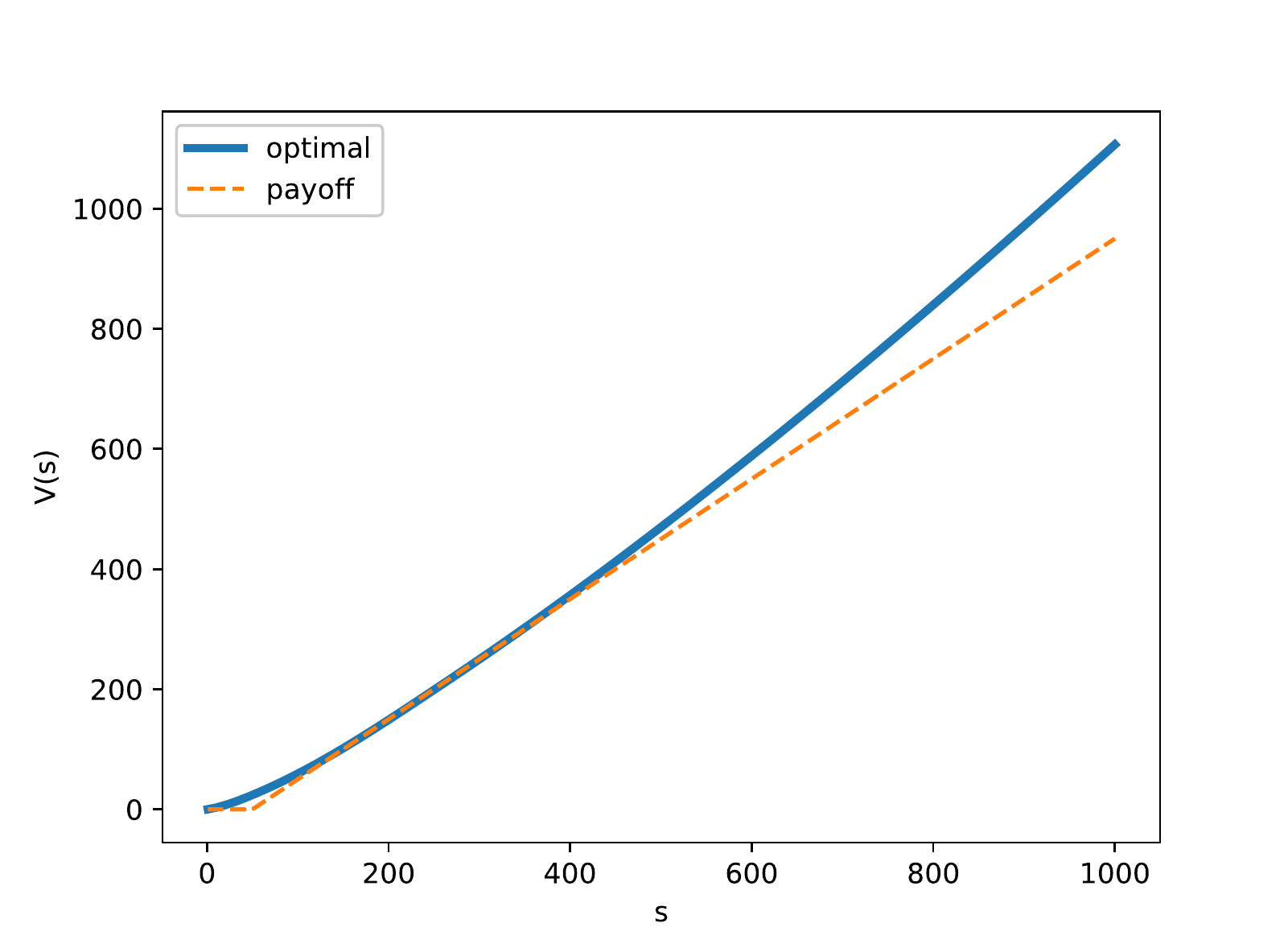} }  \\
\multicolumn{2}{c}{ (1) $V_c(s) = v_c^{SP}(\log s; l_c^*, u_c^*)$ and $G_c(s)$ }  \\
  \includegraphics[scale=0.4]{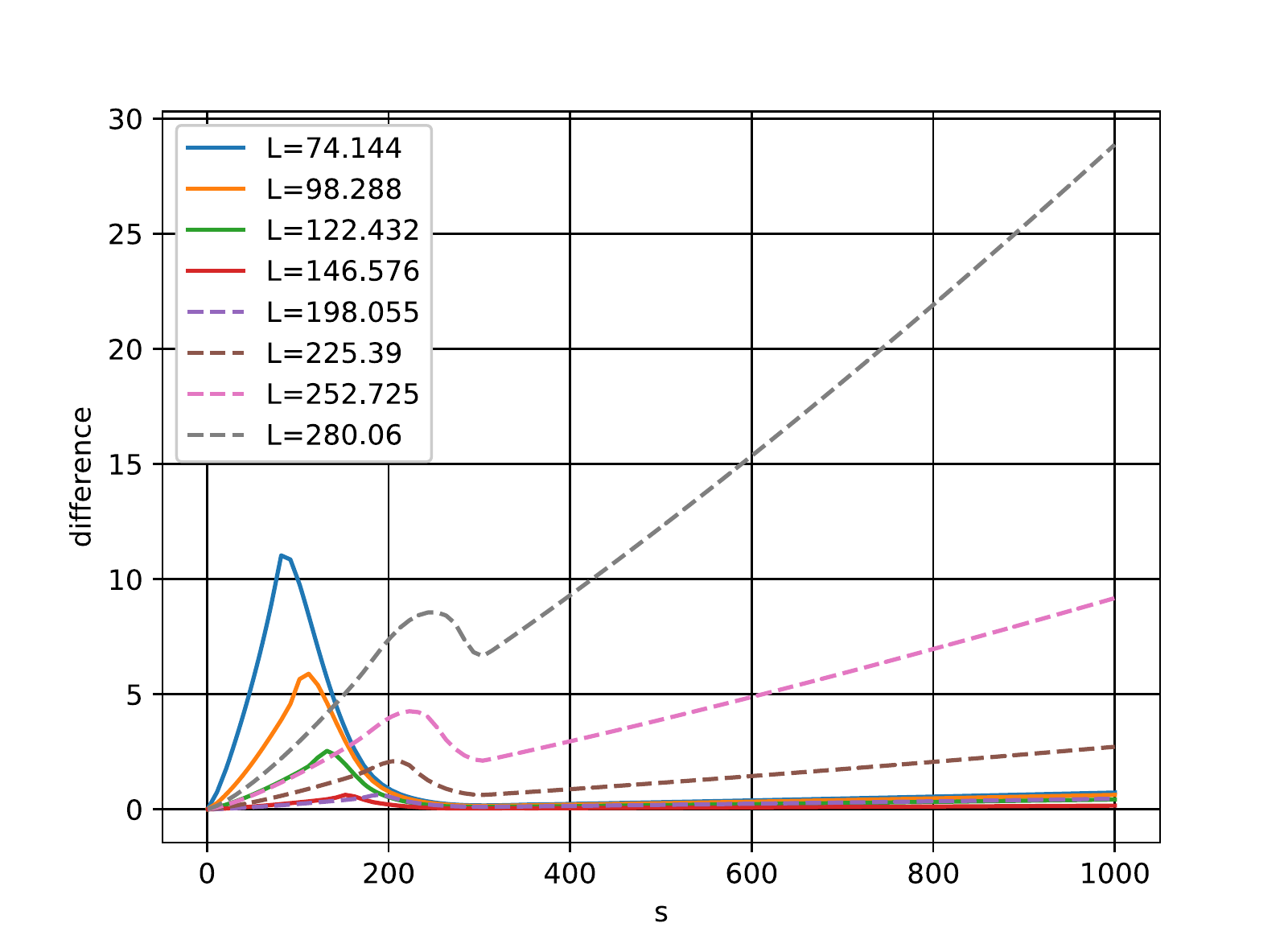} & \includegraphics[scale=0.4]{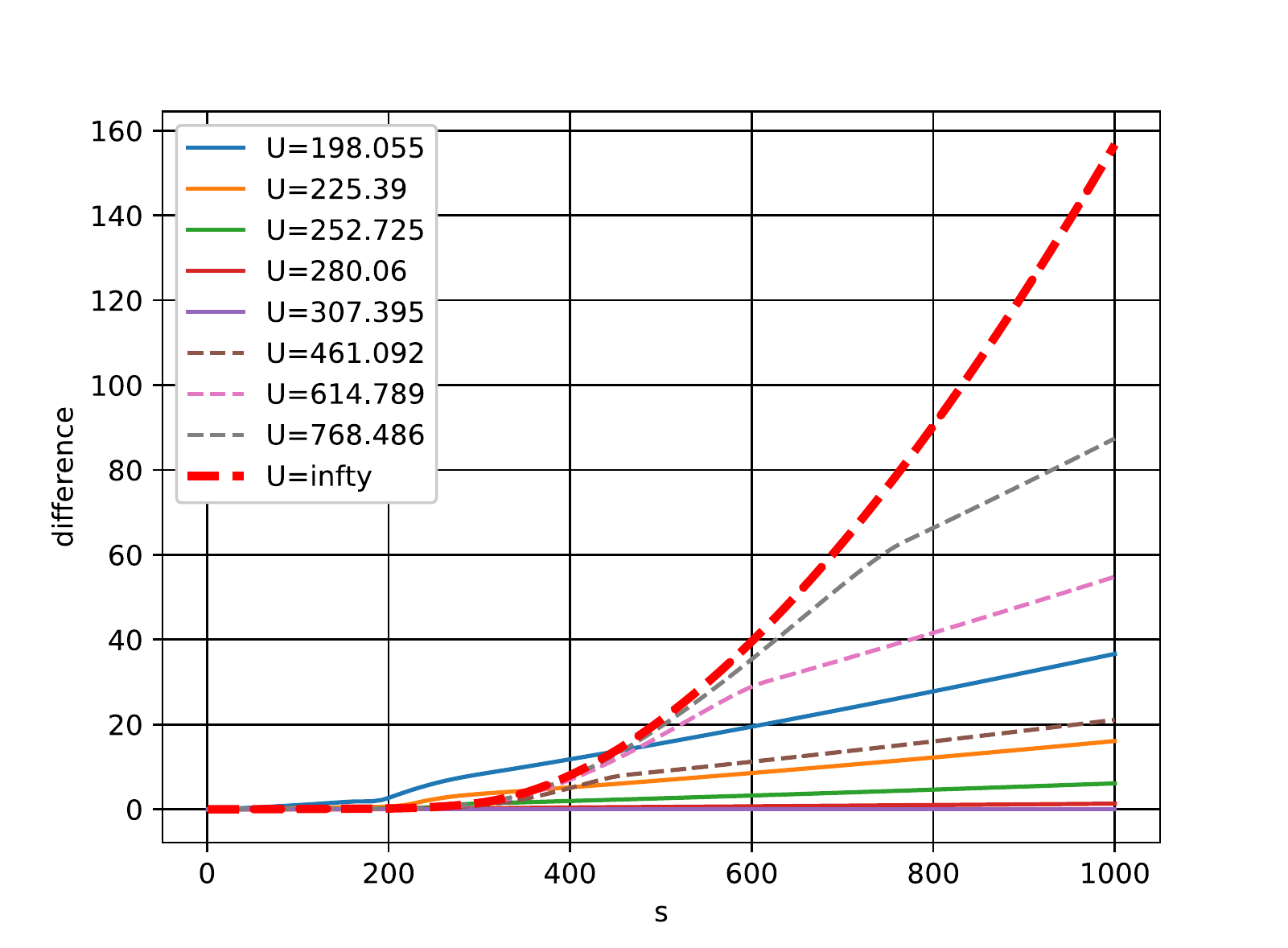} \\
   (2) $V_c(s)- v_p^{SP}(\log s; l, u^*_c)$ & (3) $V_c(s) - v_c^{SP}(\log s; l^*_c, u)$
 \end{tabular}
\end{minipage}
\caption{\footnotesize  Call option when $X$ is spectrally positive.  (1) The value function $s \mapsto V_c(s) = v_c^{SP}(\log s; l_c^*, u_c^*)$ along with the payoff function $s \mapsto G_c(s)$. (2) The difference  $s \mapsto V_c(s) - v_c^{SP}(\log s; l, u^*_c)$  for $l = \log L$ with $L= \frac {5-i} 5 K + \frac i 5 L_c^*$ and $L= \frac {5-i} 5 L_c^* + \frac i 5 U_c^*$ for $i = 1,2,3,4$    (solid when $L < L_c^*$ and dashed when $L > L_c^*$).
 (3) The difference $s \mapsto V_c(s) - v_c^{SP}(\log s; l^*_c, u)$  for $u = \log U$ with $U = \frac {5-i} 5L_c^* + \frac i 5 U_c^*$ for $i = 1,2,3,4$, $U=1.5U_c^*$, $2U_c^*$, $2.5U_c^*$, and $U=\infty$  (solid lines when $U < U_c^*$ and dashed lines when $U > U_c^*$).
}  \label{figure_call_SP}
\end{center}
\end{figure}


\subsection{Sensitivity with respect to $\lambda$} 

We now analyze the sensitivity of the optimal solutions with respect to the rate of observation $\lambda$. Here, we use the same parameters above for the four cases and solve them for different values of $\lambda$. Figures \ref{lambda_put} and \ref{lambda_call} show for the put and call cases, respectively, the value functions for various $\lambda$ as well as the optimal barriers as functions of $\lambda$.  As discussed in Lemma \ref{lemma_increasing_lambda},  as $\lambda$ increases, the value function increases and the stopping region becomes smaller. The convergence results in Theorem \ref{theorem_convergence} are also confirmed for all cases.


\begin{figure}[htbp]
\begin{center}
\begin{minipage}{1.0\textwidth}
\centering
\begin{tabular}{cc}
 \includegraphics[scale=0.4]{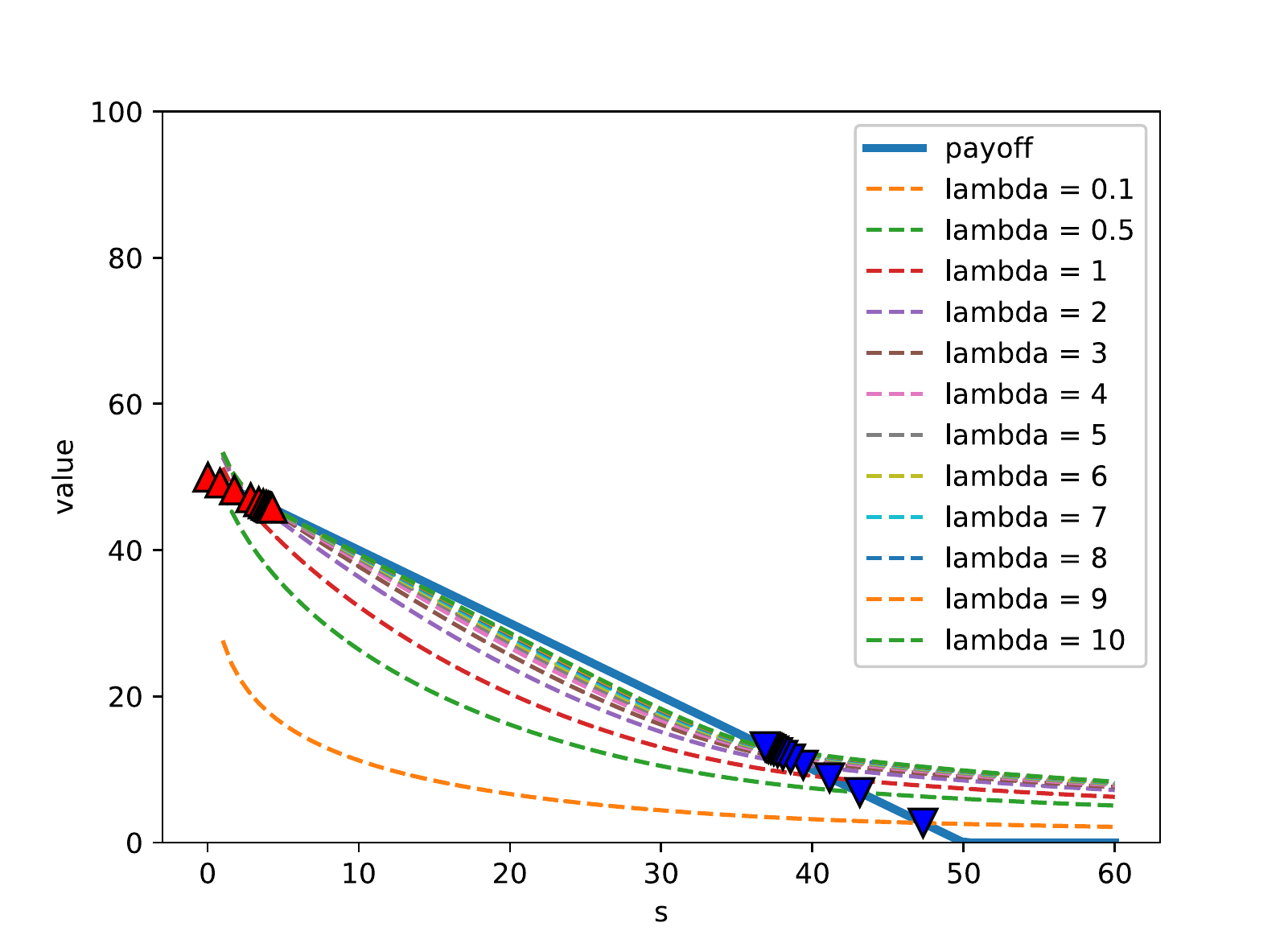} & \includegraphics[scale=0.4]{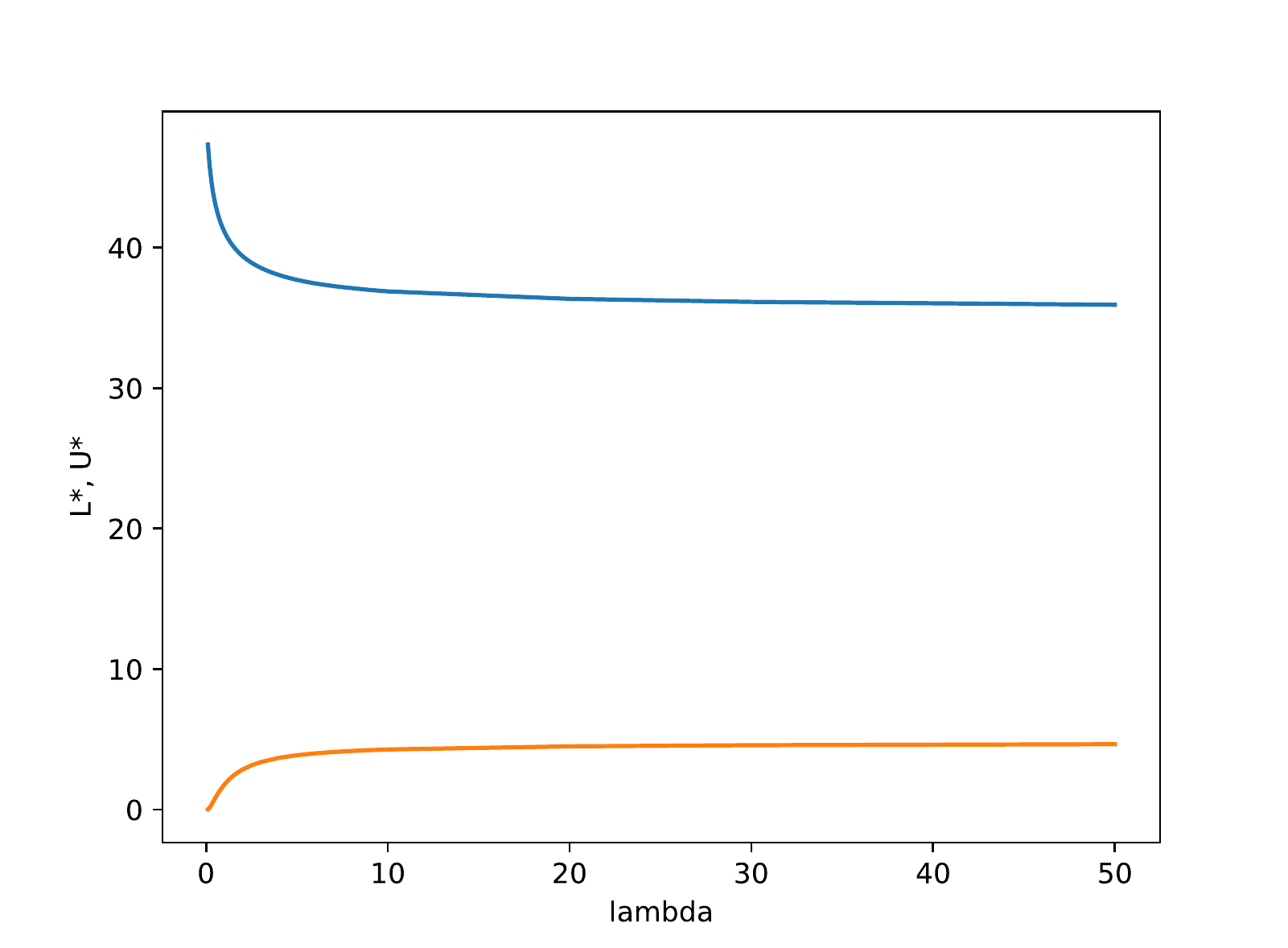}  \\
$s \mapsto V_{p,\lambda}(s)$ (spectrally negative) & $\lambda \mapsto L_{p,\lambda}^*, U_{p,\lambda}^*$ (spectrally negative)  \\
 \includegraphics[scale=0.4]{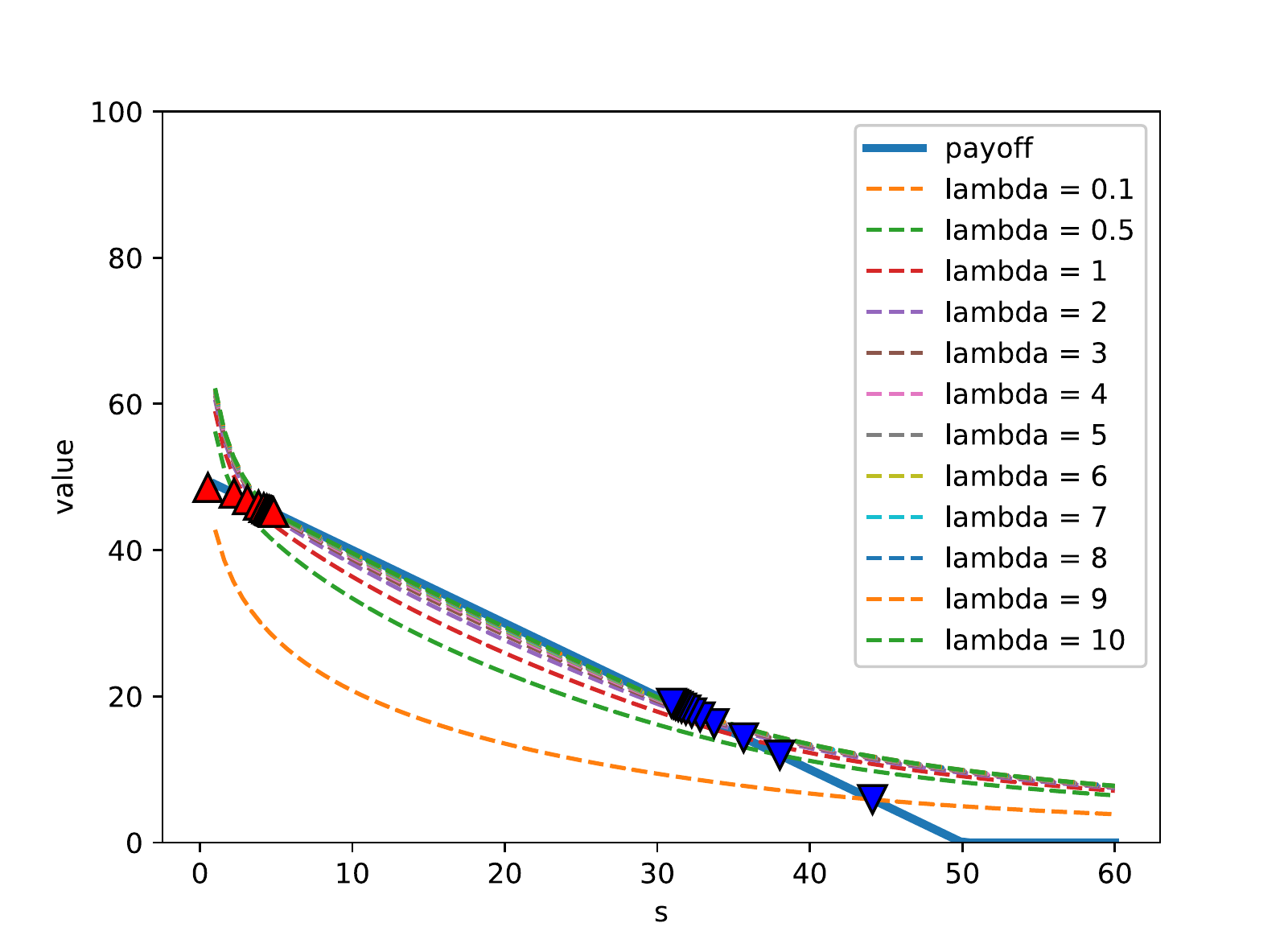} & \includegraphics[scale=0.4]{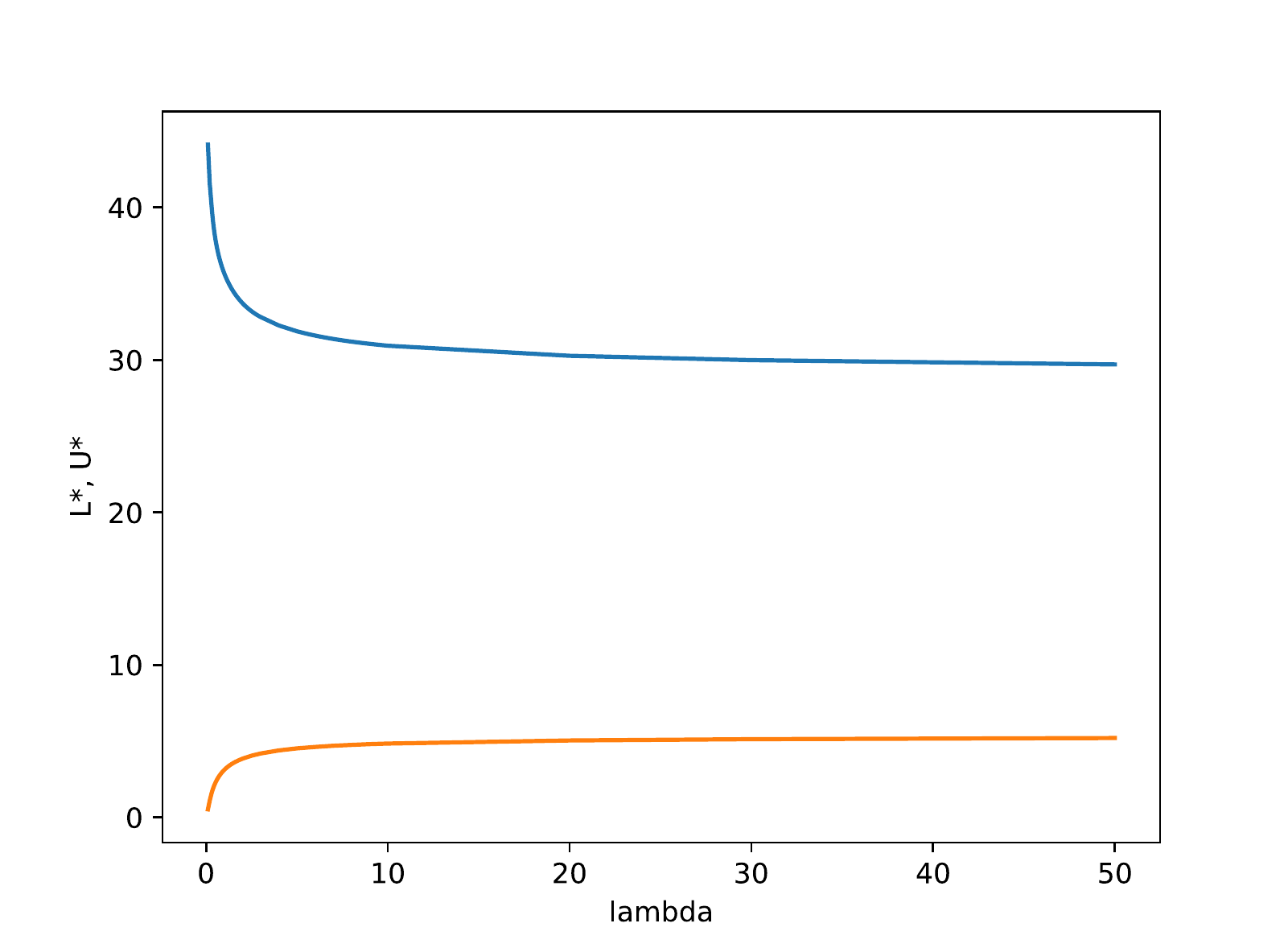}  \\
$s \mapsto V_{p,\lambda}(s)$ (spectrally positive)  &   $\lambda \mapsto L_{p,\lambda}^*, U_{p,\lambda}^*$ (spectrally positive) 
 \end{tabular}
\end{minipage}
\caption{\footnotesize Put option  for various $\lambda$ when $X$ is spectrally negative (top) and spectrally positive (bottom). (Left) The value function for $\lambda = 0.1,0.5,1,2,\ldots, 10$  (dashed lines) along with the payoff function $G_p$ (solid line). The points at $L_{p,\lambda}^*$ and $U_{p,\lambda}^*$ are indicated by up-pointing and down-pointing triangles, respectively. (Right) The optimal barriers $L_{p,\lambda}^*$ and $U_{p,\lambda}^*$ for $\lambda$ ranging from $0.1$ to $50$.
} \label{lambda_put}
\end{center}
\end{figure}

\begin{figure}[htbp]
\begin{center}
\begin{minipage}{1.0\textwidth}
\centering
\begin{tabular}{cc}
\includegraphics[scale=0.4]{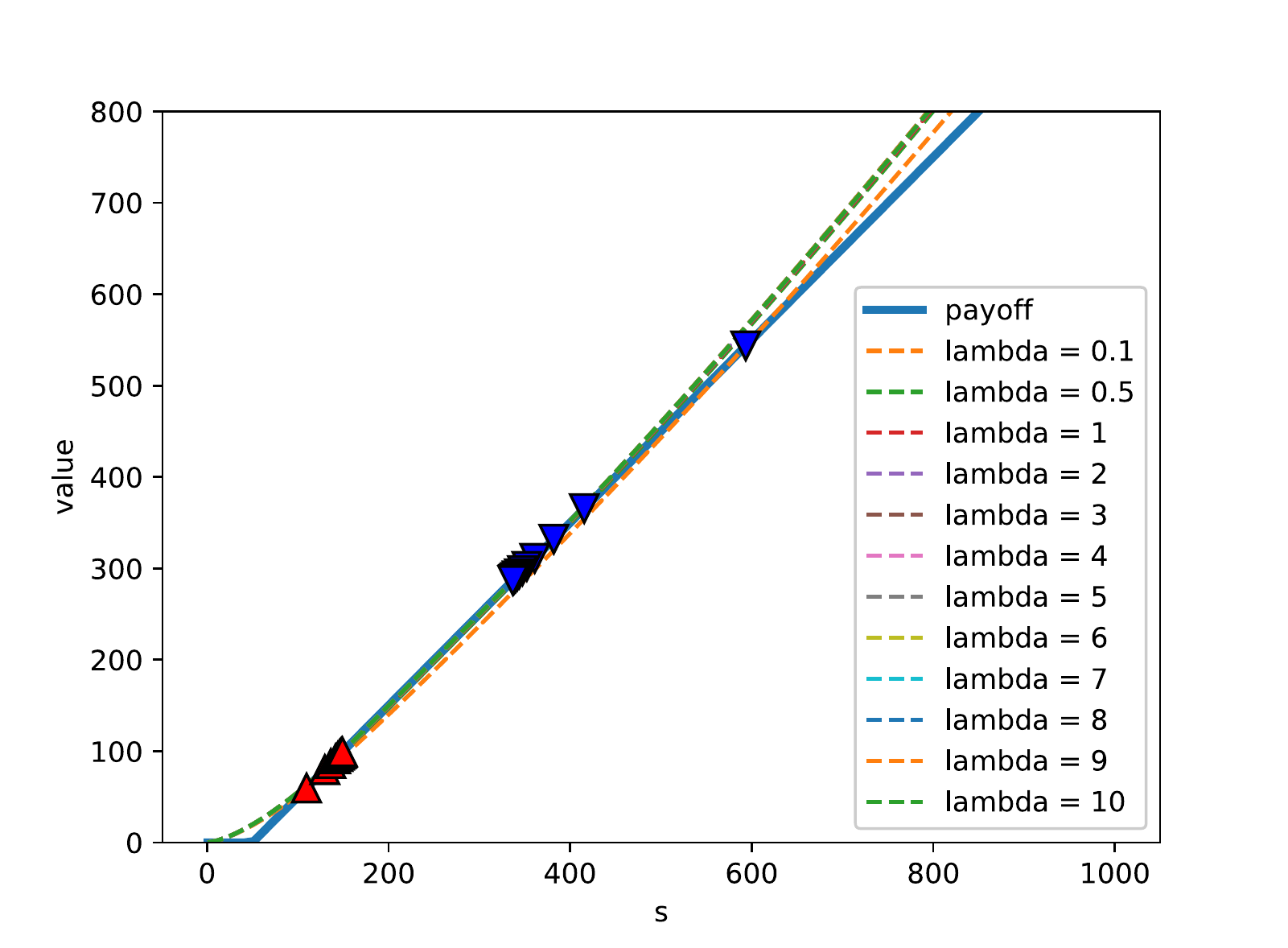} & \includegraphics[scale=0.4]{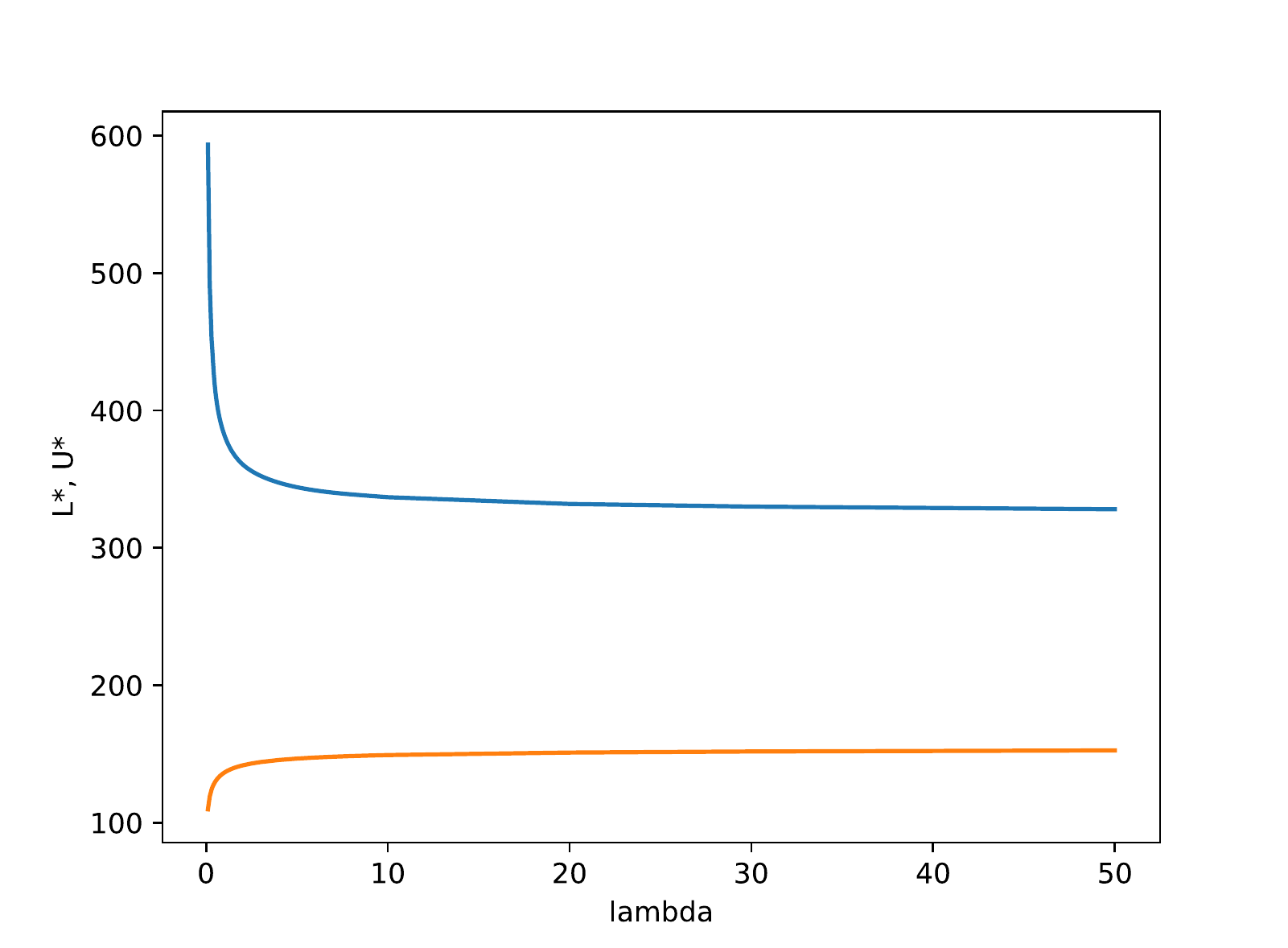}  \\
$s \mapsto V_{c,\lambda}(s)$ (spectrally negative) & $\lambda \mapsto L_{c,\lambda}^*, U_{c,\lambda}^*$ (spectrally negative)  \\
  \includegraphics[scale=0.4]{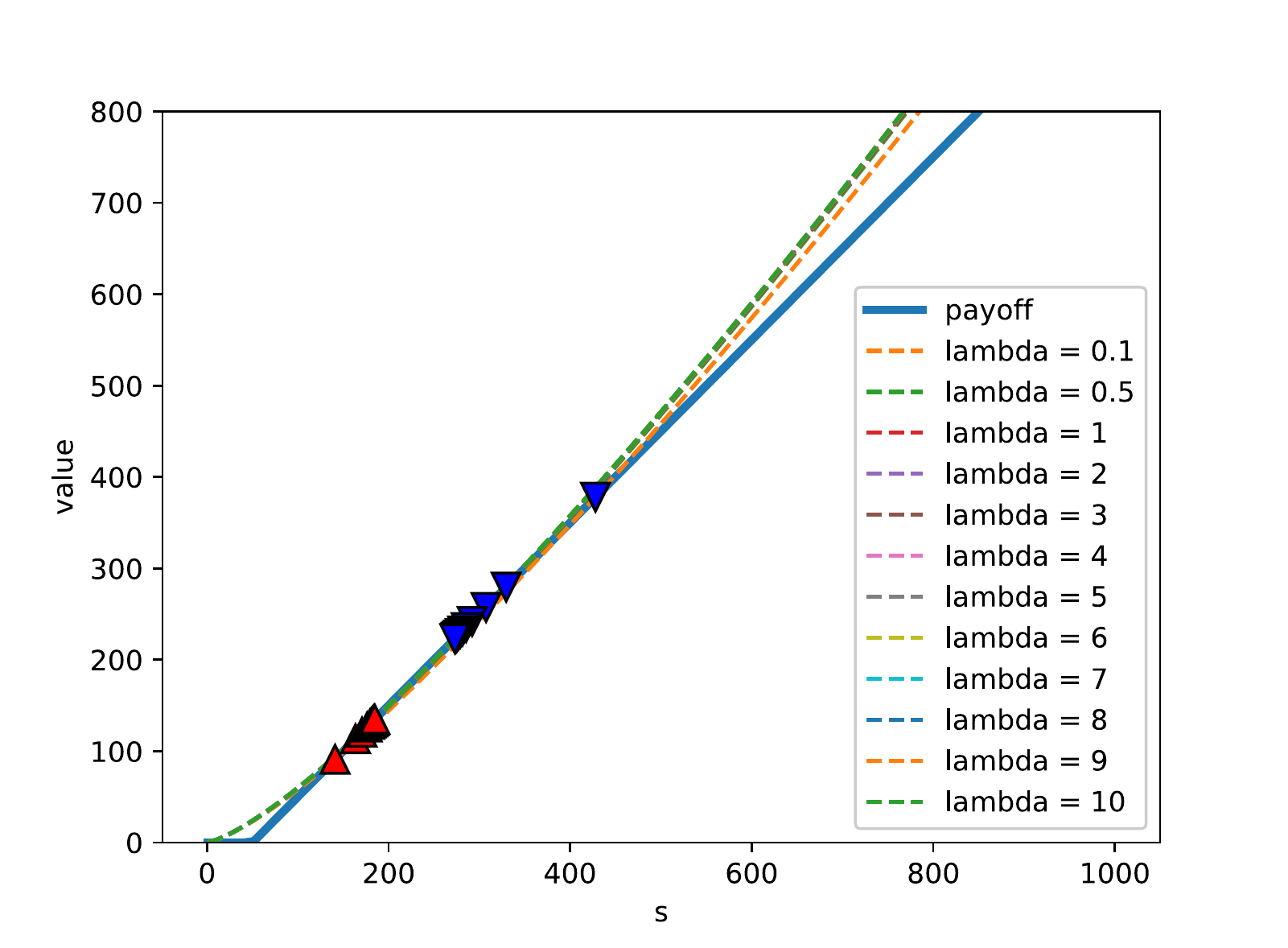} & \includegraphics[scale=0.4]{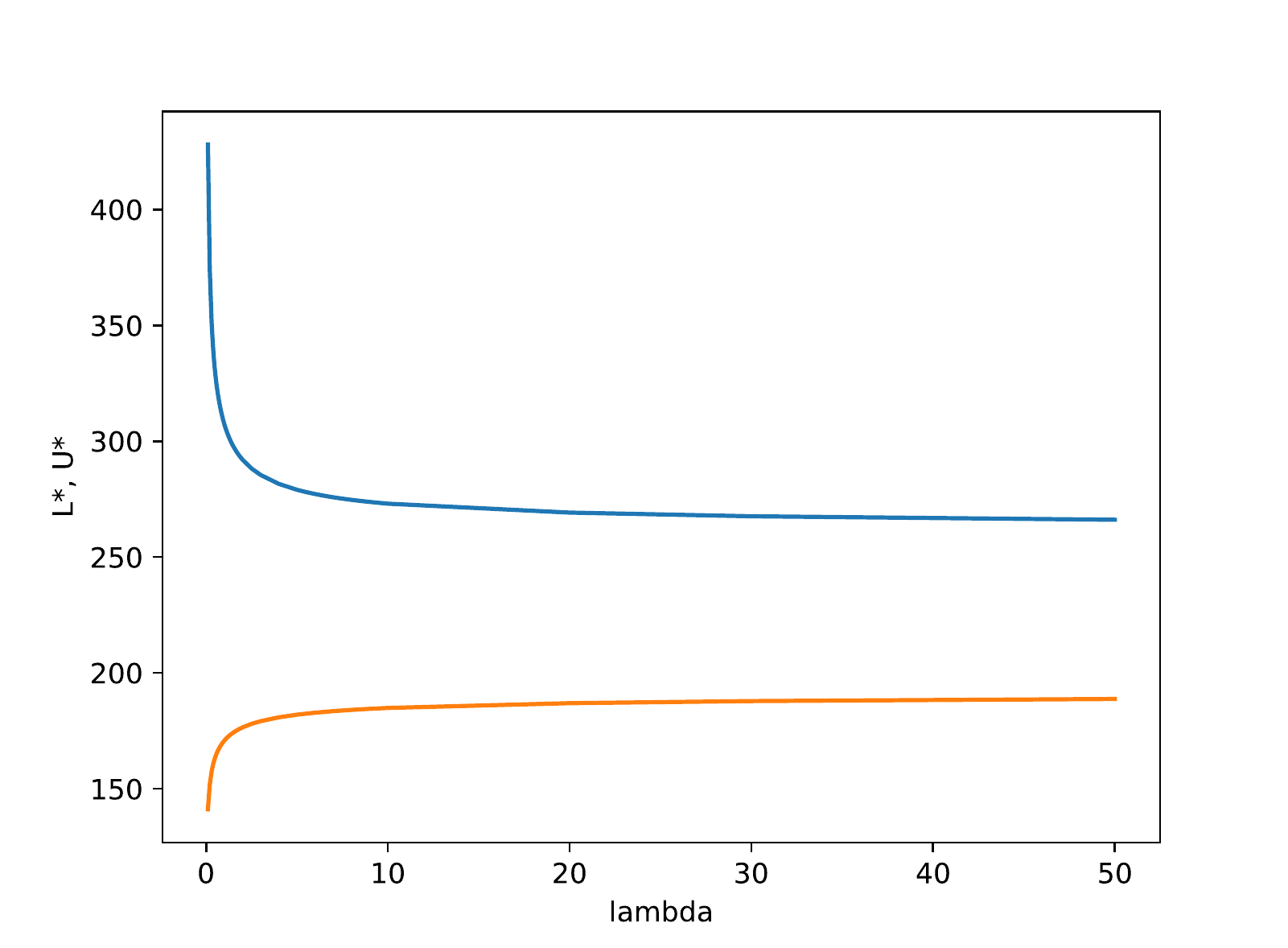}  \\
$s \mapsto V_{c,\lambda}(s)$ (spectrally positive)  &   $\lambda \mapsto L_{c,\lambda}^*, U_{c,\lambda}^*$ (spectrally positive) 
 \end{tabular}
\end{minipage}
\caption{\footnotesize 
Call option  for various $\lambda$ when $X$ is spectrally negative (top) and spectrally positive (bottom). (Left) The value function for $\lambda = 0.1,0.5,1,2,\ldots, 10$ (dashed lines) along with the payoff function $G_c$ (solid line). The points at $L_{c,\lambda}^*$ and $U_{c,\lambda}^*$ are indicated by up-pointing and down-pointing triangles, respectively. (Right) The optimal barriers $L_{c,\lambda}^*$ and $U_{c,\lambda}^*$ for $\lambda$ ranging from $0.1$ to $50$.
} \label{lambda_call}
\end{center}
\end{figure}

\appendix

\section{Proofs}
\subsection{Proof of Lemma \ref{lemma_varepsilon_stopping}}\label{appen_lemma_varepsilon_stopping}	 
We first consider the put case.
To derive a contradiction,  suppose $\inf_{0 < s' \leq K} (\bar{V}_p(s') - G_p(s')) = \varepsilon > 0$. 
Fix $s \in \R$.
 By the definition of the value function and because it is suboptimal to stop when $S\geq K$, we can choose a sequence of strategies $(\tau_n \in \mathcal{A}: n \geq 1)$ such that  $S_{\tau_n} < K$ a.s.\ on $\{ \tau_n < \infty \}$ for $n \geq 1$ and
\begin{align} \label{inf_diff}
\bar{V}_p(s) - \eee_s [e^{-r \tau_n} G_{p}(S_{\tau_n})1_{\{ \tau_n < \infty\}}] \xrightarrow{n \uparrow \infty} 0. 
\end{align}

By the dynamic programming principle (see, e.g.\ Theorem 1.11 of Peskir and Shiryaev \cite{PS}), we have 
\begin{align*}
\bar{V}_{p}(S_{T_k^\lambda}) = \max \big((K- S_{T_k^\lambda})^+, \eee [e^{-{r} (T_{k+1}^\lambda - T_k^\lambda )} \bar{V}_{p}(S_{T_{k+1}^\lambda}) |\mathcal{G}_k ] \big) \geq \eee [e^{-{r} (T_{k+1}^\lambda - T_k^\lambda )} \bar{V}_{p}(S_{T_{k+1}^\lambda})  | \mathcal{G}_k], \quad k \geq 0,
\end{align*}
where $\mathcal{G}_k := \mathcal{F}_{T_k^\lambda}$.
Therefore the process $(e^{-{r} T_k^\lambda} \bar{V}_{p}(S_{T_k^\lambda}):k \geq 0)$ is a supermartingale with respect to the filtration $\mathbb{G} := (\mathcal{G}_k: k \geq 0)$. Because,  for each $n\geq 1$, $\tau_n\in\bar{\mathcal{A}}$ can be written as $\tau_n=T^{\lambda}_{N(n)}$ for a $\mathbb{G}$-stopping time $N(n)$,  using optional sampling together with Fatou's lemma, 
\begin{align*}
	\bar{V}_{p} (s) \geq \liminf_{N \rightarrow \infty} &\eee_s [e^{-r (\tau_n \wedge T_N^\lambda)} \bar{V}_{p}(S_{\tau_n \wedge T_N^\lambda})] \geq \eee_s  [e^{-r \tau_n} \bar{V}_{p}(S_{\tau_n})1_{\{ \tau_n < \infty\}}] \\&=  \eee_s  [e^{-r \tau_n} G_{p}(S_{\tau_n})1_{\{ \tau_n < \infty\}}]  + \eee_s  [e^{-r \tau_n} (\bar{V}_{p}(S_{\tau_n}) - G_p(S_{\tau_n})) 1_{\{ \tau_n < \infty\}}]\\ &\geq  \eee_s  [e^{-r \tau_n} G_{p}(S_{\tau_n})1_{\{ \tau_n < \infty\}}] + \varepsilon   \eee_s  [e^{-r \tau_n} 1_{\{ \tau_n < \infty\}}].
\end{align*}
Now taking limits as $n \rightarrow \infty$ and by \eqref{inf_diff},
$\bar{V}_{p} (s) \geq  \bar{V}_{p} (s) + \varepsilon  \limsup_{n \rightarrow \infty} \eee_s  [e^{-r \tau_n} 1_{\{ \tau_n < \infty\}}],$ implying that $\eee_s  [e^{-r \tau_n} 1_{\{ \tau_n < \infty\}}]$ vanishes as $n \rightarrow \infty$. However, we have 
\begin{align*}
\eee_s  [e^{-r \tau_n} 1_{\{ \tau_n < \infty\}}] \geq K^{-1}  \eee_s  [e^{-r \tau_n} G_{p}(S_{\tau_n})1_{\{ \tau_n < \infty\}}] \xrightarrow{n \uparrow \infty} K^{-1} \bar{V}_{p} (s) > 0,
\end{align*}
which is a contradiction. 

For the call case, it can be shown by first transforming the problem to the equivalent put option problem as in \eqref{call_to_put} (recall our assumption that $\eee S_1$ is finite) and following the same arguments as above. 



\subsection{Proof of Lemma \ref{lemma_L_same}}\label{appen_lemma_4.1_a}	
	Because $\int_a^xW^{(q)}(x-y)e^{\theta y}  \diff y = \int_0^{x-a} W^{(q)}(z)e^{\theta (x-z)} \diff z = e^{\theta x}  \overline{W}^{(q)} (x-a; \theta)$, 
	\begin{align*}
		&e^{\theta a}Z^{(q)}(x-a;\theta)-\mathscr{Z}^{(q, \lambda)}_a (x; \theta) \\
		&= e^{\theta x} \left( 1 + (q- \psi(\theta ))\overline{W}^{(q)} (x-a; \theta) \right) \\
		&- \Big[ e^{\theta x} \left( 1 + (q + \lambda - \psi(\theta ))\overline{W}^{(q+\lambda)} (x; \theta) \right) -\lambda  \int_a^xW^{(q)}(x-y)e^{\theta y} \left( 1 + (q + \lambda - \psi(\theta ))\overline{W}^{(q + \lambda)} (y; \theta) \right)\diff y \Big] \\
		&= e^{\theta x}   (q + \lambda- \psi(\theta ))\overline{W}^{(q)} (x-a; \theta) \\
		&- \Big[ e^{\theta x}   (q + \lambda - \psi(\theta ))\overline{W}^{(q+\lambda)} (x; \theta)  -\lambda  \int_a^xW^{(q)}(x-y)e^{\theta y}  (q + \lambda - \psi(\theta ))\overline{W}^{(q + \lambda)} (y; \theta)\diff y \Big].
	\end{align*}
Dividing both sides by $q+\lambda - \psi(\theta)$, we have the claim.
\exit
\subsection{Proof of Lemma \ref{lemma_resolvent}}\label{appen_lemma_4.1}
Because $\{\mathcal{E}_t :=\exp\{-qt+\Phi(q)X_t\}:t\geq0\}$ is a martingale (see, e.g., page 82 of \cite{K}), we have that 
		 $\E_{\alpha}[ \mathcal{E}_{\tilde{T}_{\beta}^+ \wedge t}
		 ]=e^{\Phi(q)\alpha}$ for $t > 0$, and hence
		$\E_{\alpha}[
		 \mathcal{E}_{\tilde{T}_{\beta}^+ \wedge \mathbf{e}_\lambda}
		]= \lambda \int_0^\infty e^{-\lambda t}  \E_{\alpha}[
		\mathcal{E}_{\tilde{T}_{\beta}^+ \wedge t}
		] \diff t=e^{\Phi(q) \alpha}$.
		
		Because $X$ does not have positive jumps and by \eqref{upcrossing_identity} and \eqref{killed_resolvent},
		\begin{multline*}
			e^{\Phi(q) \alpha}=\E_{\alpha}\Big[
			\mathcal{E}_{\tilde{T}_{\beta}^+ \wedge \mathbf{e}_\lambda}
			\Big]=\E_{\alpha}\left[
			\mathcal{E}_{\mathbf{e}_\lambda}
			1_{\{\mathbf{e}_\lambda<\tilde{T}_{\beta}^+\}}\right]+e^{\Phi(q)\beta}\E_{\alpha}\left[e^{-q\tilde{T}_{\beta}^+}1_{\{\tilde{T}_{\beta}^+<\mathbf{e}_\lambda\}}\right]\\
			=\lambda\int_{-\infty}^{\beta}e^{\Phi(q) y}
r^{(q+\lambda)}(\alpha,y;\beta)
\diff y+
			e^{\Phi(q) \beta}e^{-\Phi(q+\lambda)(\beta-\alpha)}.
		\end{multline*}
		Hence
		\begin{align}
			\lambda \int_{-\infty}^0e^{\Phi(q) y}
r^{(q+\lambda)}(\alpha,y;\beta)
\diff y &= 
			e^{\Phi(q)\alpha}-e^{-\Phi(q+\lambda)(\beta-\alpha)}e^{\Phi(q) \beta}
			- \lambda \int_{0}^{\beta}e^{\Phi(q) y}
r^{(q+\lambda)}(\alpha,y;\beta)
\diff y\notag\\
			&= 
			e^{\Phi(q)\alpha}\left(1+\lambda \overline{W}^{(q+\lambda)}(\alpha;\Phi(q))\right)-e^{-\Phi(q+\lambda)(\beta-\alpha)}e^{\Phi(q) \beta}
			\left(1+ \lambda \overline{W}^{(q+\lambda)} (\beta; \Phi(q)) \right),\notag
		\end{align}
		which equals \eqref{g_0_ub_2}.
\exit
\subsection{Proof of Lemma \ref{lemma_aux}}\label{aux_lemma}
(i) 
By the first identity of (6) in \cite{LRZ},
\begin{align}\label{g(a)_1}
	\lambda \int_a^bW^{(q+\lambda)}(b-y)W^{(q)}(y-a)\diff y=W^{(q+\lambda)}(b-a)-W^{(q)}(b-a).
\end{align}
By Fubini's theorem together with \eqref{g(a)_1},
\begin{align*}
	\lambda \int_a^bW^{(q+\lambda)}(b-y)\int_a^yW^{(q)}(y-z)Z^{(q+\lambda)}(z;\theta)\diff z\diff y &= \lambda \int_a^bZ^{(q+\lambda)}(z;\theta)\int_z^bW^{(q+\lambda)}(b-y)W^{(q)}(y-z)\diff y\diff z\\
	&=\int_a^bZ^{(q+\lambda)}(z;\theta)\left(W^{(q+\lambda)}(b-z)-W^{(q)}(b-z)\right)\diff z.
\end{align*}
Hence
$\int_a^bW^{(q+\lambda)}(b-y)\mathscr{Z}^{(q, \lambda)}_a (y; \theta)\diff y=\int_a^bZ^{(q+\lambda)}(z;\theta) W^{(q)}(b-z) \diff z$.
By this and using \eqref{g(a)_1} again, we obtain \eqref{g(a)_2}.

(ii) 
By taking Laplace transforms on both sides (as in the second identity of (6) in \cite{LRZ}), it can be checked that
\begin{align}\label{LRZ_2}
	\lambda \int_a^bZ^{(q)}(y-a;\theta)W^{(q+\lambda)}(b-y)\diff y=
	Z^{(q+\lambda)}(b-a;\theta)-Z^{(q)}(b-a;\theta).
\end{align}
Hence, using \eqref{g(a)_1} together with \eqref{LRZ_2}, we obtain \eqref{g(a)_3}.
\exit
\subsection{Proof of Theorem \ref{theorem_limiting_case}}\label{appen_Theorem_4.2}
	


	(1) First suppose $\theta\geq0$ with $\psi(\theta) \neq q + \lambda$ and $\theta\neq \Phi(q)$. 
	
From \cite[Lem. 3.3]{KKR}, we have that $\lim_{x\rightarrow\infty}e^{-\Phi(q)x}W^{(q)}(x)=1/\psi^\prime(\Phi(q))$.
	For  $\theta> \Phi(q)$ (i.e.\ $\psi(\theta) > q$), by \eqref{scale_function_laplace} and \eqref{Z_theta}, we can write $Z^{(q)}(b;\theta) = (\psi(\theta)-q) \int_0^{\infty}e^{-\theta z}W^{(q)}(z+b)\diff z$ and hence
	\begin{align*}
		\lim_{b\to\infty}\frac{Z^{(q)}(b;\theta)}{W^{(q)}(b)}
		&=(\psi(\theta)-q)\lim_{b\to\infty}\frac{\int_0^{\infty}e^{-\theta z}W^{(q)}(z+b)\diff z}{W^{(q)}(b)}=(\psi(\theta)-q)\int_0^{\infty}e^{-\theta z}e^{\Phi(q)z}\diff z=\frac{\psi(\theta)-q}{\theta-\Phi(q)}.
	\end{align*}
On the other hand, if $\Phi(q)>\theta$ (where $e^{-\theta b} W^{(q)}(b) \xrightarrow{b \uparrow \infty} \infty$), by  L'Hospital rule,
		\begin{align*}  
\begin{split}
			\lim_{b\to\infty}\frac{Z^{(q)}(b;\theta)}{W^{(q)}(b)}
			&=\lim_{b\to\infty}\frac{1}{e^{-\theta b}W^{(q)}(b)}+(q-\psi(\theta))\lim_{b\to\infty}\frac{\int_0^{b}e^{-\theta z}W^{(q)}(z)\diff z}{e^{-\theta b}W^{(q)}(b)}\\
			&=(q-\psi(\theta))\lim_{b\to\infty}\frac{e^{-\theta b}W^{(q)}(b)}{-\theta e^{-\theta b}W^{(q)}(b)+e^{-\theta b}W^{(q)\prime}(b+)}=\frac{\psi(\theta)-q}{\theta-\Phi(q)},
\end{split}
		\end{align*}
where $W^{(q)\prime}(b+)$ is the right-hand derivative and we used that $W^{(q)\prime}(b+)/W^{(q)}(b) \xrightarrow{b \uparrow \infty} \Phi(q)$ (which can be derived, e.g., by taking limits on (8.24) of \cite{K}).
For the case $\theta=\Phi(q)$, we have, by Lemma 3.3 in \cite{KKR},
\begin{equation}\label{lim_phi_q}
\lim_{b\to\infty}\frac{Z^{(q)}(b;\Phi(q))}{W^{(q)}(b)}=\lim_{b\to\infty}\frac{e^{\Phi(q)b}}{W^{(q)}(b)}=\psi'(\Phi(q)).
\end{equation}

	By \eqref{mathcal_L}, we have
	\begin{align*}
		\lim_{b\to\infty}\frac{\mathscr{Z}^{(q, \lambda)}_a (b; \theta)}{W^{(q)}(b)}&=\lim_{b\to\infty}\frac{1}{W^{(q)}(b)}\left(Z^{(q)}(b;\theta)+ \lambda \int_0^aW^{(q)}(b-y)Z^{(q+\lambda)}(y;\theta)\diff y\right)\\
		&=\frac{\psi(\theta)-q}{\theta-\Phi(q)}+ \lambda \int_0^ae^{-\Phi(q)y}Z^{(q+\lambda)}(y;\theta)\diff y, \qquad \theta\neq\Phi(q)
	\end{align*}
and by \eqref{lim_phi_q}
	\begin{align*}
		\lim_{b\to\infty}\frac{\mathscr{Z}^{(q, \lambda)}_a (b; \Phi(q))}{W^{(q)}(b)}&=\psi'(\Phi(q))+ \lambda \int_0^ae^{-\Phi(q)y}Z^{(q+\lambda)}(y;\Phi(q))\diff y = N^{(q,\lambda)}(a).
	\end{align*}
	Hence, by taking limits as $b\to\infty$ in \eqref{fun_L} (noting ${W^{(q)}(b-a)} / {W^{(q)}(b)} \xrightarrow{b \uparrow \infty} e^{-\Phi(q) a}$), we get when $\theta \neq \Phi(q)$
	\begin{align}\label{limit_L_b}
		\lim_{b\to\infty}& \lambda\frac{L^{(q,\lambda)}(b,a;\theta)}{
W^{(q)}(b)}=\frac{\lambda}{q+\lambda-\psi(\theta)}\Bigg[\frac{\psi(\theta)-q}{\theta-\Phi(q)}(e^{(\theta-\Phi(q))a}-1)- \lambda \int_0^ae^{-\Phi(q)z}Z^{(q+\lambda)}(z;\theta)\diff z\Bigg].
	\end{align}
	Now we note that we can write
	\begin{align*}
		\int_0^a&e^{-\Phi(q)z}Z^{(q+\lambda)}(z;\theta)\diff z=\int_0^ae^{-\Phi(q)z}e^{\theta z}\left(1+(q+\lambda-\psi(\theta))\overline{W}^{(q+\lambda)} (z; \theta)\right) \diff z\\
		&=\frac{1}{\theta-\Phi(q)}(e^{(\theta-\Phi(q))a}-1)+(q+\lambda-\psi(\theta))\int_0^ae^{(\theta-\Phi(q)) z} \overline{W}^{(q+\lambda)} (z; \theta)\diff z \\
		&=\frac{1}{\theta-\Phi(q)}(e^{(\theta-\Phi(q))a}-1) +\frac{q+\lambda-\psi(\theta)}{\theta-\Phi(q)}\Big[e^{(\theta-\Phi(q))a}\overline{W}^{(q+\lambda)} (a; \theta)- \overline{W}^{(q+\lambda)} (a; \Phi(q))\Big],
	\end{align*}
where the last equality holds because $\int_0^ae^{(\theta-\Phi(q)) z} \overline{W}^{(q+\lambda)} (z; \theta)\diff z = \int_0^a  \int_w^{a} e^{(\theta-\Phi(q)) z} \diff z  e^{- \theta w}W^{(q+\lambda)}(w)\diff w$.
	Using the previous identity in \eqref{limit_L_b}, we obtain
	\begin{align*} 
		\begin{split}
			\lim_{b\to\infty}\lambda \frac{L^{(q,\lambda)}(b,a;\theta)}{
W^{(q)}(b)}
			&=\frac{\lambda}{{\Phi(q)}-\theta}\Big[(e^{(\theta-\Phi(q))a}-1)+ \lambda \Big(e^{(\theta-\Phi(q))a} \overline{W}^{(q+\lambda)} (a; \theta)-\overline{W}^{(q+\lambda)} (a; \Phi(q))\Big)\Big] \\
			&= M^{(q,\lambda)}(a; \theta).
		\end{split}
	\end{align*}
	Hence by taking limits as $b\to\infty$ in \eqref{fun_g} we obtain \eqref{flu_int}.
	
	(2) The case $\theta=\Phi(q)$ can be obtained by taking $\theta\to\Phi(q)$ in the result obtained in (1). Using L'Hospital rule,
	\begin{align*}
		\lim_{\theta \rightarrow \Phi(q)}M^{(q,\lambda)}(a; \theta)  &= \lim_{\theta \rightarrow \Phi(q)}\frac{\lambda e^{-\Phi(q)a}}{\Phi(q)-\theta}\Big[e^{\theta a}+\lambda e^{\theta a}\overline{W}^{(q+\lambda)} (a; \theta)-e^{\Phi(q) a}-\lambda e^{\Phi(q) a}\overline{W}^{(q+\lambda)} (a; \Phi(q)) \Big] \\
		&= -\lim_{\theta \rightarrow \Phi(q)}\lambda e^{-\Phi(q)a}\Big[ a e^{\theta a}+\lambda \int_0^a (a-y) e^{\theta (a-y)}W^{(q+\lambda)}(y) \diff y \Big] \\
		&= - \lambda \Big[ a +\lambda \int_0^a (a-y) e^{-\Phi(q) y}W^{(q+\lambda)}(y) \diff y \Big],
	\end{align*}
which coincides with $M^{(q,\lambda)}(a; \Phi(q))$ as defined in \eqref{fun_M}.
	
	(3) 
Finally the cases $\psi(\theta) = q + \lambda$, 
and $\theta<0$ hold by analytic continuation. 
\exit
\subsection{Proof of Lemma \ref{lemma_M_derivative}}\label{appen_M_derivative}
If $\theta \neq \Phi(q)$, we have
	\begin{align*}
		M^{(q,\lambda)\prime}(a; \theta) &= - \Phi(q)  M^{(q,\lambda)}(a; \theta)
		\\
		&+ \frac{\lambda e^{-\Phi(q)a}}{\Phi(q)-\theta}\Big[\theta e^{\theta a}+\lambda \Big( \theta e^{\theta a}\overline{W}^{(q+\lambda)} (a; \theta)  + W^{(q+\lambda)} (a) \Big)  - \Phi(q) Z^{(q+\lambda)}(a;\Phi(q)) - \lambda W^{(q+\lambda)}(a) \Big] \\
		&= - \Phi(q)  \frac{\lambda e^{-\Phi(q)a}}{\Phi(q)-\theta}e^{\theta a} (1+\lambda \overline{W}^{(q+\lambda)} (a; \theta) )  + \frac{\lambda e^{-\Phi(q)a}}{\Phi(q)-\theta} \theta e^{\theta a} ( 1 +\lambda \overline{W}^{(q+\lambda)} (a; \theta) ) \\
		&= -  \lambda e^{-(\Phi(q)-\theta) a}  (1+\lambda \overline{W}^{(q+\lambda)} (a; \theta) ).
	\end{align*}
For the case $\theta = \Phi(q)$, straightforward differentiation gives the result.
\exit

		\end{document}